\documentclass[12pt]{amsart}
\usepackage[totalwidth=495pt, totalheight=640pt]{geometry}
\usepackage{amsmath, amssymb, amsthm, amsfonts, 
eucal, enumerate, natbib, layout}
\usepackage[normalem]{ulem}
\usepackage[usenames]{color}
\usepackage{verbatim}
\usepackage{soul}
\usepackage{mathtools}
\usepackage{hyperref}
\hypersetup{colorlinks=true,citecolor={red},linkcolor={blue},urlcolor={violet}}

\newcommand{\details}[1]{}

\newtheorem{theorem}{Theorem}[section]
\newtheorem*{theorem*}{Theorem}
\newtheorem{corollary}[theorem]{Corollary}
\newtheorem*{corollary*}{Corollary}
\newtheorem{lemma}[theorem]{Lemma}
\newtheorem*{lemma*}{Lemma}

\newtheorem*{claim*}{Claim}

\newtheorem{proposition}[theorem]{Proposition}
\newtheorem*{proposition*}{Proposition}
\newtheorem{conjecture}[theorem]{Conjecture}
\newtheorem*{conjecture*}{Conjecture}
\newtheorem{def-proposition}[theorem]{Definition-Proposition}
\theoremstyle{definition}

\newtheorem*{definition*}{Definition}
\newtheorem{remark}[theorem]{Remark}
\newtheorem{notation}[theorem]{Notation}
\newtheorem{example}[theorem]{Example}
\newtheorem*{example*}{Example}
\numberwithin{equation}{section}

\newcommand{\rme}{\mathrm {e}}
\newcommand{\rmi}{\mathrm {i}}

\newcommand{\ZZ}{\mathbb{Z}}
\newcommand{\QQ}{\mathbb{Q}}

\newcommand{\CC}{\mathbb{C}}
\newcommand{\GG}{\mathbb{G}}
\newcommand{\PP}{\mathbb{P}}

\newcommand{\Galmot}{{\mathcal{G}}{\mathrm{al}}_{\mathrm{mot}}}

\newcommand{\End}{\mathrm{End}}

\newcommand{\UR}{\mathrm{UR}}

\newcommand{\Lie}{\mathrm{Lie}\,}

\newcommand{\tor}{\mathrm{tor}}
\newcommand{\tR}{\widetilde{R}}

\newcommand{\oQQ}{\overline{\QQ}}

\newcommand{\cP}{\mathcal{P}}
\newcommand{\cE}{\mathcal{E}}

\newcommand{\e}{\mathrm{e}}
\newcommand{\ii}{\mathrm{i}}

\def\showanswers{0}

\newcommand{\soluzioni}[1]{
	\ifnum\showanswers=1
	
	#1 \vspace{\baselineskip}
	\fi
	\ifnum\showanswers=1
	\vspace{0\baselineskip} \hspace{2cm}
	\fi
}

\begin{document}

\title[A conjecture in Schanuel style for 1-motives]
{A conjecture in Schanuel style for 1-motives}

\author{Cristiana Bertolin}
\address{Dipartimento di Matematica, Universit\`a di Padova, Via Trieste 63, Padova}
\email{cristiana.bertolin@unipd.it}



\date{\today}



\begin{abstract}

Schanuel Conjecture contains all ``reasonable" statements that can be made on the values of \textit{the exponential function}. In particular it implies the Lin\-de\-mann-Weierstrass Theorem. In \cite{B02} we showed that Schanuel Conjecture has a geometrical origin: it is \textit{equivalent to} the Grothendieck-Andr\'{e} period Conjecture applied to a 1-motive without abelian part. 

In this paper, we state a conjecture in Schanuel style, which will imply conjectures in Lin\-de\-man\-n-Weierstrass style, for \textit{the semi-elliptic exponential function}, that is for the exponential map of an extension $G$ of an elliptic curve $\cE$ by the multiplicative group $\GG_m^r$. We propose the \textit{semi-elliptic Conjecture}, which concerns the exponential function, the Weierstrass $\wp \,$, $\zeta$ functions and Serre functions. The case of a trivial extension $G= \GG_m \times \cE$ has been treated in \cite{BW}, where we introduced the \textit{split semi-elliptic Conjecture}.
As in Schanuel's case, we expect that the semi-elliptic Conjecture contains all ``reasonable" statements that can be made on the values of the exponential function, of the Weierstrass $\wp \,$, $\zeta$ functions and of Serre $f_q$ functions.

We show that the semi-elliptic Conjecture has a geometrical origin (as Schanuel Conjecture): it is \textit{equivalent to} the Grothendieck-Andr\'{e} period Conjecture applied to a 1-motive whose underlying abelian part is an elliptic curve.

We also introduce the $\sigma$-Conjecture which involves the Weierstrass $\wp \,$, $\zeta$ and $\sigma$ functions and we show that this conjecture is a consequence of the Grothendieck-Andr\'{e} period Conjecture applied to an auto-dual 1-motive (or equivalently a consequence of the semi-elliptic Conjecture applied to adequate points).

\end{abstract}


\maketitle



\section*{Introduction}

Let $\oQQ$ be the algebraic closure of the field $\QQ$ of rational numbers. Throughout the paper, unless otherwise specified, transcendence degrees (denoted $\mathrm{t.d.}$)
 are taken over $\QQ$. 

Consider the {\em the exponential function} of the torus $\GG_{m}$
\[
\begin{aligned} 
	\exp: \Lie (\GG_{m})_{\CC} =\CC & \longrightarrow \GG_m(\CC) = \CC^\times \\
	\nonumber t & \longmapsto \rme^t. 
\end{aligned}
\] 
In 1960 Schanuel proposed the following conjecture, which is expected to contain all ``reasonable" statements that can be made on the values of the exponential function:

\begin{conjecture} 
	[Schanuel's Conjecture] 
	If $t_1,\dots,t_s$ are $\QQ$--linearly independent complex numbers, then at least $s$ of the $2s$ numbers $t_1,\dots,t_s, \rme^{t_1},\dots, \rme^{t_s}$ are algebraically independent.
\end{conjecture}

\par\noindent Equivalently, the transcendence degree of the field $\QQ(t_1,\dots,t_s, \rme^{t_1},\dots, \rme^{t_s})$ is at least $s$. Schanuel conjecture implies the {\em Lindemann--Weierstrass Theorem} (1885): 
	if $\alpha_1,\dots,\alpha_s$ are $\QQ$--linearly independent algebraic numbers, then the numbers $\rme^{\alpha_1},\dots,\rme^{\alpha_s}$ are algebraically independent.

This article introduces a conjecture in Schanuel style, that will imply conjectures in Lin\-de\-man\-n-Weierstrass style, for \textit{the semi-elliptic exponential function}, that is for the exponential map of an extension $G$ of an elliptic curve $\cE$ by the multiplicative group $\GG_m^r$.

Let $\Omega = \ZZ \omega_1 + \ZZ  \omega_2$ be a lattice in $\CC$ with elliptic invariants $g_2,g_3.$
Let $\cE$ be the elliptic curve associated to $\Omega$ and denote by $k$ its field of endomorphisms, namely:
\[
k:= \End (\cE) \otimes_\ZZ \QQ=\begin{cases}
	\QQ & \text{ in the non--CM case,}
	\\
	\QQ(\tau) & \text{ in the CM case}
\end{cases}
\]
where $\tau:=\omega_2/\omega_1$. In both cases $\QQ \subseteq k \subseteq \oQQ.$

Consider the Weierstrass $\wp$, $\zeta$ and $\sigma$ functions relative to the lattice $\Omega$. 
Let $q_1, \dots,q_r$ be $r$ complex numbers which do not belong to $\Omega,$ and denote by $G$ the algebraic group which is an extension of the elliptic curve $\cE$ by $\GG_m^r$ parametrized by the points 
$Q_j=\exp_{\cE^*}(q_j)$\footnote{In the whole text we use small letters for elliptic logarithms of points of $\cE(\CC)$ or $\cE^*(\CC)$  which are written with capital letters.} for $j=1, \dots,r$
of the dual elliptic curve $\cE^*,$ that we identify with $\cE.$ Denote by 
\[f_{q_j}(z) = \frac{\sigma(z+q_j)}{\sigma(z)\sigma(q_j)} \rme^{-\zeta(q_j) z }\]
the Serre function associated to the complex number $q_j \in \CC \smallsetminus \Omega.$ \textit{The semi-elliptic exponential function} of $G$ (composed with a projective embedding) is
\begin{align*}
	\nonumber	\exp_{G}:  \Lie G_\CC & \longrightarrow   G_\CC (\CC) \subset \PP^{2+2r} (\CC) \\
	\nonumber (z, t_1, \dots, t_r) & \longmapsto \sigma(z)^3\Big[ \wp(z):\wp'(z):1 : \rme^{t_j} f_{q_j}(z): \rme^{t_j} f_{q_j}(z) \Big( \wp(z) + \frac{\wp'(z)-\wp'(q_j)}{\wp(z)-\wp(q_j)} \Big) \Big]_{j=1, \dots, r}.
\end{align*}

Before to state our conjecture \`{a}  la Scha\-nu\-el for this exponential function we need more notations.
If $q_1, \dots, q_r,p_1,\dots,p_n,$ are $k$--linearly independent complex numbers in $\CC \smallsetminus \Omega,$ we set
\[\tor (q_j, p_i): = \dim_k \Big( (\Omega\otimes_\ZZ\QQ ) \cap \big(\sum_{j=1}^r k q_j +  \sum_{i=1}^n k p_i \big) \Big) . \]
In the CM case, the integer $\tor (q_j,p_i)$ can be 0 or 1, in the non CM-case it can be 0, 1 or 2. In particular the dimension of the sub $k$--vector space $ <p_i,q_j>_{i,j}$ of  $ \CC / (\Omega\otimes_{\ZZ}\QQ)$ generated by the classes of
$p_1,\ldots,p_n,q_1,\ldots,q_r$ modulo $\Omega\otimes_{\ZZ}\QQ$ is $ n+r - \tor(p_i,q_j).$

Let $p_1,\dots,p_{n}, q_1, \dots,q_{r}$ be complex numbers in $\CC \smallsetminus \Omega.$ Denote by $B$ the smallest abelian sub-variety (modulo isogenies) of $\cE^{n} \times \cE^{*r}$ which contains a multiple of the point 
$(P_1, \dots, P_{n},Q_1, \dots , \\ Q_{r}) \in \cE^{n} \times \cE^{*r}(\CC).$
The inclusion $I:B \hookrightarrow \cE^{n}\times  \cE^{*r}$ induces group morphisms
\begin{align}
	\nonumber I :B &\longrightarrow  \cE^{n}\times  \cE^{*r}\\
	\nonumber b &\longmapsto \big(\gamma_1(b),\dots,\gamma_{n}(b),\gamma_{1}^*(b),\dots,\gamma_{r}^*(b)\big)
\end{align}
where  $\gamma_i\in\mathrm{Hom}_{\mathbb Q}(B,\cE) := \mathrm{Hom}(B,\cE) \otimes_\ZZ \QQ $ ({\it resp.} $\gamma^*_j \in\mathrm{Hom}_{\mathbb Q}(B,\cE^*)$) is the composition of $I$ with the projection on the $i$-th factor $\cE$ of $\cE^{n}$ ({\it resp.} on the $j$-th factor $\cE^*$ of $\cE^{*r}$) for $i=1,\dots,n$ ({\it resp.} $j=1,\dots,r$).
We denote with an upper-index ${}^t$ the transpose of a group morphism. 
Set $\beta_{i,j}: =\gamma_i^t\circ\gamma_j^* \in \mathrm{Hom}_{\mathbb Q}(B,B^*).$  The points $P_1,\dots,P_{n}$ in $ \cE(\CC)$ and $Q_1,\dots,Q_{r} $ in $ \cE^{*}(\CC)$ define two group homomorphisms $v: \ZZ^{n} \to \cE$ and $v^*: \ZZ^{r} \to \cE^{*} $ respectively. Denote by $x_1, \dots, x_{n}$ (\textit{resp}. $y^\vee_1, \dots , y^\vee_{r}$) the basis of $\ZZ^{n}$ (\textit{resp.} $\ZZ^{r}$)  such that
$v(x_i)=P_i \in \cE(\CC)$ and $ v^*(y^\vee_j) =Q_j \in \cE^*(\CC)$
for $ i=1, \dots,n $ and $ j=1, \dots,r. $ By \cite[Theorem 4.2]{BP} the kernel of the surjective map
\begin{equation}
	\begin{matrix}
		f:(X\otimes Y^\vee)_\QQ &\twoheadrightarrow &\sum_{i,j}\QQ (\beta_{i,j} + \beta_{i,j}^t) \subset \mathrm{Hom}_\QQ (B,B^*)\\
		\hfill x_i\otimes y_j^\vee &\mapsto & \beta_{i,j} + \beta_{i,j}^t \hfill
	\end{matrix}
\end{equation}
is the space of all $\QQ$--linear relations among the homomorphisms
$\beta_{ij}+\beta_{ij}^t$. The motivic interpretation of the $\QQ$--vector space
$< \beta_{ij}+\beta_{ij}^t > $ generated by the $\beta_{ij}+\beta_{ij}^t$
and of its space of relations $
\ker(f) $
will be discussed in Section \ref{MotGalGroup}, where these spaces will be related to the toric part of the Lie algebra of the unipotent radical of the corresponding $1$-motive.
A pair $(p_i,q_j)$ is said to admit a Cartier-dual pair if
$q_j$ coincides with one of the complex numbers $p_1,\ldots,p_n$ and
$p_i$ coincides with one of the complex numbers $q_1,\ldots,q_r$.
Equivalently, there exist indices $i'$ and $j'$ such that
$q_j=p_{i'}$ and $ p_i=q_{j'}.$
Set
\[
\mathcal{ D}_{n,r}
:=
\Big\{
(i,j)  \; \big\vert \quad
 p_i\neq q_j,\;
p_i,q_j\notin \Omega\otimes_{\ZZ}\QQ,\;
(p_i,q_j)\text{ admits a Cartier-dual pair}
\Big\}.
\]
Choose a subset $
\mathcal{ D}^{+}_{n,r}\subseteq \mathcal{ D}_{n,r} $
containing exactly one element from each Cartier-dual pair.
We can now state our conjecture \`{a}  la Scha\-nu\-el for the semi-elliptic exponential function:

\begin{conjecture}[Semi-elliptic Conjecture]\label{s-eSC}
Let $\Omega$ be a lattice in $\CC$. Let $s,n,n',n'',r,r',r''$ be integers such that
$s \geqslant 0,  0 \leqslant n \leqslant n'\leqslant n'', 0 \leqslant r \leqslant r'\leqslant r''.$ Let
\begin{itemize}
	\item $t_1, \dots,t_s$ be $\QQ$--linearly independent complex numbers;
	\item 	$ q_1, \dots,q_r,p_1,\dots,p_n$ be $k$--linearly independent complex numbers in $\CC \smallsetminus \Omega.$ Choose subsets $I\subseteq \{1,\ldots,n\}$ and $J\subseteq \{1,\ldots,r\}$ such that the classes of $
	\{p_i, q_j\}_{\substack{ i\in I \\ j\in J}}$ form a $k$--basis of the sub $k$--vector space $ <p_i,q_j>_{\substack{i=1, \dots,n \\ j=1, \dots,r}}$ of  $ \CC / (\Omega\otimes_{\ZZ}\QQ)$ generated by the classes of
	$p_1,\ldots,p_n,q_1,\ldots,q_r$ modulo $\Omega\otimes_{\ZZ}\QQ.$ In particular $\dim_k  <p_i,q_j>_{\substack{i=1, \dots,n \\ j=1, \dots,r}}= |I|+|J|;$
	\item 	$ q_{r+1}, \dots,q_{r'},p_{n+1},\dots,p_{n'}$ be complex numbers in $\CC \smallsetminus \Omega$ such that
		\begin{align} \label{condition}
		\nonumber	\dim_k  <p_i,q_j>_{\substack{i=1, \dots,n' \\ j=1, \dots,r'}} &= 	\dim_k  <p_i,q_j>_{\substack{i=1, \dots,n \\ j=1, \dots,r}} ,\\
			\dim_\QQ  < \beta_{i,j}+\beta_{i,j}^t>_{\substack{i \in I \cup \{n +1, \dots,n'\} \\ j \in J \cup \{ r+1 , \dots,r'\}}}& = (|I|+n'-n)(|J|+r'-r) 	-\frac{1}{2}|\mathcal{ D}_{n',r'}|,
		\end{align}
		where 
		 $< \beta_{i,j}+\beta_{i,j}^t>_{ \substack{i \in I \cup \{n +1, \dots,n'\} \\ j \in J \cup \{ r+1 , \dots,r'\}}}$ is the sub $\QQ$--vector space of $\mathrm{Hom}_\QQ(B,B^*) $ generated by the group homomorphisms $\beta_{i,j}+\beta_{i,j}^t$ for $i\in I \cup \{n +1, \dots,n'\}$ and $ j \in J \cup \{ r+1 , \dots,r'\}$ (here $B$ is the smallest abelian sub-variety of $\cE^{n'} \times \cE^{*r'}$ which contains the point 
		 $(P_1, \dots, P_{n'},Q_1, \dots ,  Q_{r'})$);
 \item 	$ q_{r'+1}, \dots,q_{r''},p_{n'+1},\dots,p_{n''}$ be complex numbers in $\CC \smallsetminus \Omega$ and for $i=n'+1,\ldots,n'', j=r'+1,\ldots,r'',$ $t_{ij}$ be complex numbers such that
 \begin{align} \label{condition2}
 	\nonumber	\dim_k  <p_i,q_j>_{\substack{i=1, \dots,n'' \\ j=1, \dots,r''}} &= 	\dim_k  <p_i,q_j>_{\substack{i=1, \dots,n \\ j=1, \dots,r}}, \\
 	\dim_\QQ  < \beta_{i,j}+\beta_{i,j}^t>_{\substack{i \in I \cup \{n +1, \dots,n''\} \\ j \in J \cup \{ r+1 , \dots,r''\}}}& = \dim_\QQ  < \beta_{i,j}+\beta_{i,j}^t>_{\substack{i \in I \cup \{n +1, \dots,n'\} \\ j \in J \cup \{ r+1 , \dots,r'\}}},\\
 	\nonumber \dim_\QQ 	< \textstyle{\sum_{\substack{ n'+1\leqslant i \leqslant n'' \\ r'+1 \leqslant j \leqslant r''}}}	\alpha_{ij}t_{ij}> & = (n''-n')(r''-r'),
 \end{align}
 where $\dim_\QQ 	< \textstyle{\sum_{\substack{ n'+1\leqslant i \leqslant n'' \\ r'+1 \leqslant j \leqslant r''}}}	\alpha_{ij} t_{ij} >$ is the sub $\QQ$--vector subspace of  $\CC/ 2 \pi \ii \QQ$ generated by the classes of  $\sum_{i,j}\alpha_{ij} t_{ij}$, with $\sum_{i,j}\alpha_{ij} x_i\otimes y_j^\vee \in \ker (f).$
\end{itemize}
Then the transcendence degree of the field generated over $\QQ$ by the numbers 
$$t_1, \dots,t_s, \rme^{t_1},\dots ,\rme^{t_s}, g_2,g_3, q_1, \dots,q_{r}, p_1, \dots,p_{n},$$
$$ \wp(q_1), \dots,\wp(q_{r}), \zeta(q_1), \dots,\zeta (q_{r}),\wp(p_1), \dots,\wp(p_{n}), \zeta(p_1), \dots,\zeta (p_{n}), $$
$$ f_{q_j}(p_i) \quad
\text{for } (i,j)\in \big( I\cup\{n+1,\ldots,n'\} \big) \times \big( J\cup\{r+1,\ldots,r'\}\big) \setminus \mathcal{ D}^{+}_{n',r'},$$
$$ t_{n'+1 \; r'+1}, \dots, t_{n''r''}, \rme^{t_{ij}}f_{q_j}(p_i) \quad \mathrm{for}\; i \in \{ n'+1, \dots, n''\} \; \mathrm{ and \; for}\; j \in \{ r'+1, \dots, r''\},$$
is at least  $s+2(r+n)+ (|I|+n'-n)(|J|+r'-r) -\frac{1}{2}
|\mathcal{D}_{n',r'}|+ (n''-n')(r''-r')$, unless $ 2 \pi \ii \subset \sum_l \QQ t_l$ and $\Omega \subset \sum_i  k p_i + \sum_j k q_j$ in which case it is at least $s+2(r+n)+(|I|+n'-n)(|J|+r'-r)-\frac{1}{2}
|\mathcal{D}_{n',r'}|+(n''-n')(r''-r') -1.$
\end{conjecture}

Some remarks about this conjecture:

(1) The parameters $t_{ij}$ with
$i\leqslant n'$ and $
j\leqslant r'$ do not appear explicitly in the above conjecture. Indeed, Theorem \ref{Teo:dimGal(M)} shows that they do not  contribute additional dimensions beyond those already accounted for by
$ \dim_k < p_i,q_j >$
and $\dim_{\QQ} <\beta_{ij}+\beta_{ij}^t>.$
Likewise, these parameters do not contribute to the additional term governed by $\ker(f)$. This phenomenon is compatible with the description of the unipotent radical in \cite[\S 6 2 and 1.1.1]{BP} and is illustrated by Examples (d2) and (e) of \cite[\S 2]{B-Futur}.

 (2)By \cite[Corollary 4.5]{BP} conditions \eqref{condition} and \eqref{condition2} imply that for $ i \in I \cup \{n +1, \dots,n''\} $ and $ j \in J \cup \{ r+1 , \dots,r''\}, $ the points $Q_{j},P_{i}$ 

- are not torsion points, and hence $\tor(q_j,p_i)_{\substack{i=1, \dots,n \\ j=1, \dots,r}} = \tor(q_j,p_i)_{\substack{i=1, \dots,n'' \\ j=1, \dots,r''}},$ 

- are not $k$--linearly dependent via an antisymmetric homomorphism, that is we do not have that $\phi(P_i)=Q_j$ (or $\phi(Q_j)=P_i$) with $\phi+\overline{\phi}=0.$

(3) A subtle point arises from Cartier duality. Indeed, two distinct pairs
$(p_i,q_j)$ and 
$(p_{i'},q_{j'})$
such that
$q_j=p_{i'},p_i=q_{j'}$
may define Cartier dual $1$-motives. In this situation they contribute the same element
$
\beta_{ij}+\beta_{ij}^t$
 and therefore should not be counted twice in the computation of
$
\dim_{\QQ}\langle \beta_{ij}+\beta_{ij}^t\rangle.$ 
Consequently, each Cartier-dual pair contributes only once to the expected lower bound in the above Conjecture.

(4) Set $\tau_p:=n-|I|$ and $\tau_q:=r-|J|.$ Since $|I|+|J|=n+r-\tor(q_j,p_i),$ $\tau_p+\tau_q= \tor (q_j, p_i).$
Notice that $\tau_p$ and $\tau_q$ depend on the chosen basis and are therefore not intrinsic invariants of the family
$p_1,\ldots,p_n,q_1,\ldots,q_r;$ only their sum $\tor (q_j, p_i)$ is intrinsic. Moreover we have the equality
\[ |I|+n'-n = n' - \tau_p       \qquad  \qquad |J|+r'-r =r' - \tau_q. \]

(5) The semi-elliptic Conjecture implies Schanuel Conjecture associated to the exponential function and the split semi-elliptic Conjecture associated to the split semi-elliptic exponential function in \cite[Conjecture 2.1]{BW}.

\medspace

In 1966 Grothendieck made an allusion to a conjecture about the transcendental degree of the field generated by the periods of an abelian variety. André wrote down Grothendieck conjecture in greater generality (\cite[Appendix 2]{B19}):

\begin{conjecture}[Grothendieck-André period Conjecture]\label{GAC}
	For any sub field $K$ of $\CC$ and for any (pure or mixed) motive $M$ defined over $K$,
	\[ \mathrm{t.d.}\, K \big(\mathrm{periods}(M)\big)\geq \dim \Galmot (M). \]
If the motive $M$ is defined over $\overline{\QQ},$ then equality holds.
\end{conjecture}

\par\noindent In this conjecture, $K (\mathrm{periods}(M)) $ denotes the field generated by the periods of $M$ over $K$ and $\Galmot (M)$ is the motivic Galois group of $M,$ that is the fundamental group of the tannakian category generated by $M$. Remark that without loss of generality, in this conjecture we may assume $K$ to be algebraically closed: just take $K=\overline{K}$. The Grothendieck-André period Conjecture furnishes the geometrical origin of the semi-elliptic Conjecture. In fact, consider the 1-motive
\begin{equation} \label{GPC-1-motive}
	M=[u:\ZZ \longrightarrow  G^n],  \; 
 u(1)=\big(R_1, \dots, R_n\big)  \in G^n(\CC),
\end{equation}
where $G$ is the extension of the elliptic curve $\cE$ by $\GG_m^r$ parametrized by the points 
$Q_1,   \ldots,  Q_{r}$, and 
\begin{align}
\label{R}	R_{i}&=\exp_G(p_i,t_{i1}, \dots ,t_{ir})\\
\nonumber	&=  \sigma(p_i)^3\Big[ \wp(p_i):\wp'(p_i):1 : \rme^{t_{ij}} f_{q_j}(p_i): \rme^{t_{ij}} f_{q_j}(p_i) \Big( \wp(p_i) + \frac{\wp'(p_i)-\wp'(q_j)}{\wp(p_i)-\wp(q_j)} \Big) \Big]_{j=1, \dots, r}
\end{align}
for $i=1, \dots,n.$	Remark that the 1-motive $M$ is univocally defined by the $2+r+n+rn$ numbers 
	\begin{equation}\label{Points}
		 g_2 \in \CC, \;\; g_3 \in \CC, \;\; q_j \in \CC \smallsetminus \Omega, \;\;p_i \in \CC \smallsetminus \Omega, \;\; t_{ij} \in \CC.
	\end{equation}
While the elliptic invariants $g_2$ and $g_3$ should satisfy $g_2^3- 27 g_3^2 \not= 0,$ the numbers $q_j,p_i$ and $t_{ij}$ are completely arbitrary. Throughout the paper, we assume that the field of definition $K$ of the $1$-motive $M$ is algebraically closed. The main theorem of this paper is

\begin{theorem}\label{GA<=>Schanuel}
	The semi-elliptic Conjecture 
	is equivalent to the Grothendieck-André period Conjecture applied to the 1-motive \eqref{GPC-1-motive}.
\end{theorem}

Our proof generalizes 

- the equivalence between Schanuel Conjecture and Conjecture \ref{GAC} applied to the 1-motive with no abelian part $M=[u:\ZZ \rightarrow \GG_m^s  ],
u(1)=(\rme^{t_1}, \dots, \rme^{t_s} ) \in  \GG_m^s (\CC)$ (see \cite[Corollaire 1.3]{B02}).	

- the equivalence between the split semi-elliptic Conjecture and  Conjecture \ref{GAC} applied to the 1-motive $
M=[u:\ZZ \rightarrow \GG_m^s \times \cE^n ], u(1) =(\rme^{t_1}, \dots, \rme^{t_s} , P_1, \dots, P_n ) \in (\GG_m^s\times \cE^n )(\CC),$
with $ P_i=[\wp(p_i): \wp'(p_i):1] $ for $i=1, \dots,n $ (see \cite[Theorem 4.3]{BW}).

If the points $q_j,p_i$ and $t_{ij}$ \eqref{Points} defining the 1-motive $M$ \eqref{GPC-1-motive} are linearly independents, the Grothendieck-André period Conjecture applied to $M$ furnishes a more readible conjecture (or respectively if the points involved in the semi-elliptic Conjecture \ref{s-eSC} are linearly independent, i.e. $n=n'=n''$ and $r=r'=r''$, this conjecture reads):

\begin{conjecture}[Semi-elliptic Conjecture for linearly independent complex numbers]\label{s-eSC-LI}
	Let $\Omega$ be a lattice in $\CC$. Let
	\begin{itemize}
		\item $t_1, \dots,t_s$ be $\QQ$--linearly independent complex numbers, and
		\item 	$ q_1, \dots,q_r,p_1,\dots,p_n,$ be $k$--linearly independent complex numbers in $\CC \smallsetminus \Omega.$ Choose subsets $I\subseteq \{1,\ldots,n\}$ and $J\subseteq \{1,\ldots,r\}$ such that the classes of $
		\{p_i, q_j\}_{\substack{ i\in I \\ j\in J}}$ form a $k$--basis of the sub $k$--vector space of  $ \CC / (\Omega\otimes_{\ZZ}\QQ)$ generated by the classes of
		$p_1,\ldots,p_n,q_1,\ldots,q_r$ modulo $\Omega\otimes_{\ZZ}\QQ.$ 
	\end{itemize}
	Then the transcendence degree of the field generated over $\QQ$ by the $2s+2+3r+3n+|I||J|$ numbers 
	$$t_1, \dots,t_s,\rme^{t_1},\dots ,\rme^{t_s}, g_2,g_3, q_1, \dots,q_r, p_1, \dots,p_n,$$
	$$ \wp(q_1), \dots,\wp(q_r), \zeta(q_1), \dots,\zeta (q_r),	\wp(p_1), \dots,\wp(p_n), \zeta(p_1), \dots,\zeta (p_n), $$
	$$  f_{q_j}(p_i) \quad \mathrm{for} \, i \in I  \; \mathrm{and} \;  j \in J ,  $$ 
	is at least  $s+2(r+n)+|I||J|$, unless $ 2 \pi \ii \QQ \subset \sum_l \QQ t_l$ and $\Omega \subset \sum_i  k p_i + \sum_j k q_j$ in which case it is at least $s+2(r+n)+|I||J| -1.$
\end{conjecture}


Starting from the semi-elliptic Conjecture we can propose several conjectures à la Lindemann-Weierstrass. Here we state two of them.
If in Conjecture \ref{s-eSC-LI} we assume the lattice to have algebraic invariants and the complex numbers $q_1, \dots,q_r,p_1,\dots,p_n,t_1, \dots,t_s$ to be algebraic, we get

\begin{conjecture}[Semi-elliptic LW Conjecture for linearly independent complex numbers]\label{s-eLWC}
	Let $\Omega$ be a lattice in $\CC$ with algebraic invariants. If 
	\begin{itemize}
		\item $t_1, \dots,t_s$ are $\QQ$--linearly independent algebraic numbers, and
		\item  $q_1, \dots, q_r,p_1,\dots,p_n$ are $k$--linearly independent algebraic numbers,
	\end{itemize} 
	then the $s+2(r+n)+rn$ numbers 
	$$\rme^{t_1},\dots ,\rme^{t_s} ,\wp(q_1), \dots,\wp(q_r), \zeta(q_1), \dots,\zeta (q_r), \wp(p_1), \dots,\wp(p_n), \zeta(p_1), \dots,\zeta (p_n), $$
	$$  f_{q_1}(p_1),\dots ,f_{q_1}(p_n),\dots , f_{q_r}(p_1), \dots ,f_{q_r}(p_n) $$
	are algebraically independent over $\oQQ $.
\end{conjecture}

We apply Conjecture \ref{s-eSC-LI}  with $ 2 \pi \ii \QQ \nsubseteq \sum_l \QQ t_l$ and $\tor(p_i,q_j)=0,$ since $\pi$ is transcendental and by Schneider's theorem the poles $\not=0$ of a Weierstrass $\wp$ function with algebraic invariants are transcendental. For $r=n=0$ Conjecture \ref{s-eLWC} is the Lindemann-Weierstrass Theorem, and for
 $r=0$ it is just the split semi-elliptic LW Conjecture stated in \cite[Conjecture 2.3]{BW}.

If we want a conjecture in Lindemann-Weierstrass style for an extension $G$ defined over $\oQQ$, that is the lattice has algebraic invariants, the complex numbers $\wp(q_1), \dots,\wp(q_r),p_1,\dots,p_n$ and $t_1, \dots,t_s$ are algebraic, from Conjecture \ref{s-eSC-LI}, we get

\begin{conjecture}[LW Conjecture for $G$ defined over $\oQQ$]
	Let $\Omega$ be a lattice in $\CC$ with algebraic invariants. If 
	\begin{itemize}
		\item $t_1, \dots,t_s$ are $\QQ$--linearly independent algebraic numbers, and
		\item  	$ q_1, \dots,q_r,p_1,\dots,p_n$ are $k$--linearly independent complex numbers in $\CC \smallsetminus \Omega$, such that $\tor (q_j) =0 $ and  $\wp(q_1), \dots,\wp(q_r),p_1,\dots,p_n$ are algebraic,
	\end{itemize} 
	then the $s+2(n+r)+rn$ numbers 
	$$\rme^{t_1},\dots ,\rme^{t_s}, q_1, \dots,q_r, \zeta(q_1), \dots,\zeta (q_r),\wp(p_1), \dots,\wp(p_n), \zeta(p_1), \dots,\zeta (p_n),$$
	$$ f_{q_1}(p_1), \dots ,f_{q_1}(p_n),\dots , f_{q_r}(p_1), \dots , f_{q_r}(p_n)  $$
	are algebraically independent over $\oQQ $.
\end{conjecture}

In the Grothendeick-André period Conjecture applied to the 1-motive \eqref{GPC-1-motive} (or in the semi-elliptic Conjecture) only the Weierstrass $\wp$ and $\zeta$ functions and Serre functions play a role. We now want to involve also the Weierstrass $\sigma$ function.
As we will show in Theorem \ref{GA=>Sigma}, Conjecture \ref{GAC} applied to the auto-dual 1-motive \eqref{eq:sigma2} implies the following conjecture:

\begin{conjecture}
	[$\sigma$-Conjecture]\label{SigmaC}
	Let $\Omega$ be a lattice in $\CC.$ If 
	$ p_1,\dots,p_n$ are $k$--linearly independent complex numbers in $\CC \smallsetminus \Omega,$
	then the transcendence degree of the field  
	$$ \QQ \big(g_2,g_3,p_1, \dots,p_n, \wp(p_1), \dots,\wp(p_n), \zeta(p_1), \dots,\zeta (p_n),  \sigma(p_1), \dots,\sigma (p_n) \big)$$ 
	is at least $ 2n + \dim_k < p_i>_{i} + \dim_\QQ < \zeta(p_i) p_i>_{i=1,\tor(p_i)}$
	where 
	\begin{itemize}
		\item $< p_i>_{i}$ is the sub $k$--vector space of $\CC / (\Omega \otimes_\ZZ \QQ)$ generated by the classes of  $ p_1, \dots, p_n $  modulo $\Omega \otimes_\ZZ \QQ,$ and 
		\item $< \zeta(p_i) p_i>_{i=1,\tor(p_i)}$ is the sub $\QQ$--vector space of $\CC / 2 \pi \ii \QQ$ generated by the classes of the complex numbers $ \zeta(p_i) p_i$  modulo $2 \pi \ii \QQ$ if $\tor(p_i) \not =0$ (the space
		$< \zeta(p_i) p_i>_{i=1,\tor(p_i)}$
		is understood to be trivial if $\tor(p_i)=0$).
	\end{itemize} 
\end{conjecture}

The values of the Weierstrass $\sigma$ function at $k$--linearly independent complex numbers do not appear in the Gro\-then\-dieck-Andr\'{e} period Conjecture as periods but as coordinates of points defining the auto-dual 1-motive \eqref{eq:sigma2}.

If $g_2$ and $g_3$ are algebraic, by Corollary \ref{ImplicationsGA} the dimension of the $\QQ $-vector $ < \zeta(p_i) p_i>_{i=1,\tor(p_i)}$ is $\tor(p_i)$ and so the $\sigma$-Conjecture becomes 

\begin{conjecture}
	[$\sigma$-Conjecture with $g_2$ and $g_3$ algebraic]\label{SigmaCg_23Algebraic}
	Let $\Omega$ be a lattice in $\CC.$ Assume $g_2$ and $g_3$ to be algebraic. If 
	$ p_1,\dots,p_n,$ are $k$--linearly independent complex numbers in $\CC \smallsetminus \Omega,$
	then at least $3n$ of the numbers 
	$$p_1, \dots,p_n, \wp(p_1), \dots,\wp(p_n), \zeta(p_1), \dots,\zeta (p_n),  \sigma(p_1), \dots,\sigma (p_n)$$ 
	are algebraically independent over $\oQQ $.
\end{conjecture}

We finish furnishing some literature. This paper is the sequel of \cite{B02,BW}. In \cite{BT} the authors prove the functional analogue of the Grothendieck-André period Conjecture. In \cite{Z} Conjecture \ref{GAC} applied to 1-motives is reformulated in terms of model theory. Finally the book \cite{HW} explains in all aspects the motivic notions used in questions of transcendence involving periods of 1-motives.

I am very grateful to Yves André, Daniel Bertrand, Patrice Philippon and Michel Waldschmidt for their comments which allowed to greatly improve this text.

\section{Notation}\label{Notation}

Let $\Omega = \ZZ \omega_1 + \ZZ  \omega_2$ be a lattice in $\CC$ with elliptic invariants $g_2,g_3$. We refer to the elements of 
$\Omega$ as periods. Consider the Weierstrass  $\wp, \zeta, \sigma$ functions relative to the lattice $\Omega.$
Recall that the Weierstrass function $\wp$ is even, while the functions $\zeta, \sigma$ are odd.
The quasi--periodicity of the Weierstrass $\zeta$ function
\begin{equation}\label{zeta(z+omega)}
	\zeta(z+\omega)=\zeta(z)+\eta(\omega) \qquad 	\forall \omega\in\Omega
\end{equation}
defines  
a linear map $\eta: \Omega \to \CC, \omega \mapsto  \eta(\omega) :=\eta$. We set $\eta_i= \eta(\omega_i)$ for $i=1,2$.

According to \cite[page 205]{F}\footnote{In \cite[page 205]{F}, there is a typographical error: instead of $m_1m_2+m_1$ one should read $m_1m_2+m_1+m_2$}
and by \cite[\S 1]{Wald84} the quasi--periodicity formula and the multiplication formula of the $\sigma$ function state that if $\omega$ is a period and $m$ an integer, $ m\geqslant 1$, then  
\begin{align}
	\label{sigma(z+omega)}	\sigma (z+   \omega)&= \epsilon(\omega) \sigma(z) \rme^{ \eta (z+\frac{\omega}{2})},\\
	\label{sigma(mz)}	\sigma(mz)&= (-1)^{m-1} \sigma(z)^{m^2} \psi_m(\wp(z),\wp'(z)),
\end{align}
where $\epsilon(\omega) =1$ if $\omega \in 2 \Omega$, and $-1$ if $\omega \notin 2 \Omega,$ and $\psi_m(X,Y)$ is a polynomial in $ \QQ(g_2,g_3)[X,Y]$ (see \cite[page 185]{F} for an recursive description of the polynomials $\psi_m(X,Y)$ for any $m$). 
In particular $\psi_1(X,Y)$ is the constant polynomial 1.

For $(z_1,z_2) \in \CC\times \CC\setminus\Omega$, consider Serre function
\[
f_{z_2}(z_1) = \frac{\sigma(z_1+z_2)}{\sigma(z_1)\sigma(z_2)} \rme^{-\zeta(z_2) z_1 }.
\]
Using the quasi--periodicity relations \eqref{zeta(z+omega)} and \eqref{sigma(z+omega)} for the $\zeta$ and $\sigma$ functions, we have that 
\begin{equation}\label{f(z+omega)}
	f_{z_2}(z_1+\omega)= f_{z_2}(z_1) \rme^{\eta z_2 -\omega\zeta(z_2)}  \quad \mathrm{and} \quad f_{z_2+ \omega }(z_1)= f_{z_2}(z_1) \qquad 	\forall \omega\in\Omega.
\end{equation}
Serre function is $\Omega$--periodic with respect to the second variable $z_2$, since the extension $G$ of the elliptic curve $\cE$ by $\GG_m$ depends on the point $Q$ of $\cE^*$ and not on its elliptic logarithm $q \in \Lie \cE^*.$ Moreover we have the equality $f_{z_2}(-z_1)=-f_{-z_2}(z_1).$

There is a linear relation between the periods $\omega_1, \omega_2,\eta_1,\eta_2,2\pi\rmi $, called the {\em Legendre relation}: 
\begin{equation}\label{Equation:Legendre}
	\omega_2 \eta_1 - \omega_1 \eta_2 = 2\pi\rmi,
\end{equation}
with the $+$ sign when the imaginary part of $\omega_2/\omega_1$ is positive. If $\omega=n \omega_1+ m \omega_2$ and $\omega'=n' \omega_1+ m' \omega_2$ are two periods,
by an explicit computation we have that $\omega \eta'- \omega'\eta =(mn'-nm')(\omega_2 \eta_1- \omega_1 \eta_2). $
Using \eqref{Equation:Legendre} we get the {\em generalized Legendre relation} (to arbitrary periods):
\begin{equation}\label{lemma:MixProduct}
	\omega \eta'-  \omega' \eta \in 2 \rmi \pi \ZZ.
\end{equation}
Moreover $\omega \eta'-  \omega' \eta =0$ if and only if $mn'-nm'=0.$

Assume $\cE$ to be CM. Let $\tau=\frac{\omega_2}{\omega_1}.$ Then $k$ is the imaginary quadratic extension $k=\QQ(\tau)$ of $\QQ$ and $\tau$ is a root of a polynomial 
\begin{equation}\label{Polynomial:tau}
A+BX+CX^2 \in \ZZ[X],
\end{equation}
where $A,B,C$ are relatively prime integers with $C>0$. In particular $C \tau \Omega \subseteq \Omega$.

According to \cite[Chap. III, \S3.2, Lemma 3.1] {M75}\footnote{At the beginning of the book \cite{M75}, the author assumes the invariants $g_2,g_3$ algebraic but his Lemma 3.1 remains true even without this hypothesis. \cite[Appendix B, Th. 8]{BK} does not assume that $g_2$ and $g_3$ are algebraic.} 
and \cite[Appendix B, Th. 8]{BK}, there are two independent linear relations between the periods $\omega_1, \omega_2,\eta_1,\eta_2$, namely
\begin{equation}\label{Equation:tau}
	\omega_2-\tau\omega_1 =0
\end{equation}
and 
\begin{equation}\label{Equation:kappa}
	A\eta_1- C\tau\eta_2-\kappa\omega_2=0
\end{equation}
where $\kappa$ is algebraic over the field $\QQ(g_2,g_3)$. Hence we have

\begin{lemma}\label{OmegaEta}
	If $\cE$ has complex multiplication, the numbers $\omega_2$ and $\eta_2$ are algebraic over the field $k (g_2,g_3,\omega_1,\eta_1).$ 
\end{lemma}

\section{Values of the Weierstrass $\sigma$ function and Serre functions at torsion points}\label{ValuesAtTorsionPoints}

Some of the results of this section were announced without proof in \cite[\S 1 and 2]{Wald84}; we provide complete proofs here.

\begin{proposition}\label{Prop:SigmaTorsion}
	Let $\omega \in \Omega$ and let $ m\geqslant 2$ be an integer such that $ \frac{\omega}{m} \notin \Omega.$ Then the following identity holds:
	\[
	\Big(\sigma\Big(\frac{\omega}{m}\Big) \rme^{- \frac{\eta \omega}{ 2m^2}}\Big)^{m(m+2)}= \epsilon(\omega) (-1)^m \frac{1}{\psi_{m+1}(\wp\left(\frac{\omega}{m}\right),\wp'\left(\frac{\omega}{m})\right) } \cdot 
	\] 
	In particular the number $ \sigma\big(\frac{\omega}{m}\big) \rme^{-\frac{\eta \omega}{2 m^2}} $ is algebraic over the field $\QQ(g_2,g_3).$
\end{proposition}

\begin{proof}
	Substituting $z= \frac{\omega}{m}$ in \eqref{sigma(z+omega)} we have 
	\[
	\sigma \Big(\frac{\omega}{m} (1+m)\Big)= \epsilon(\omega) \sigma\Big(\frac{\omega}{m}\Big) \rme^{ \eta \omega \frac{m+2}{2m}}.
	\]
	On the other hand applying \eqref{sigma(mz)} to $z= \frac{\omega}{m}$ and to the integer $m+1$, we get
	\[
	\sigma\Big((m +1)\frac{\omega}{m}\Big) = (-1)^{m} \sigma\Big(\frac{\omega}{m}\Big)^{(m+1)^2} \psi_{m+1}\Big(\wp\left(\frac{\omega}{m}\right),\wp'\left(\frac{\omega}{m}\right)\Big).
	\]
 Combining these identities yields the first claim.
	Using \cite[Lemma 3.1 (3)]{BW} we obtain the second statement.
\end{proof}

\begin{lemma}\label{LemmaSigma}
	Let $\omega \in \Omega $ and let $ m\geqslant 2$ be an integer such that $ \frac{\omega}{m} \notin \Omega$. We have the equality
	\[
	\Big(\frac{\sigma\left(z+\frac{\omega}{m}\right)}{\sigma(z)}\Big)^{m^2}= \epsilon(\omega)  \rme^{\eta (mz+ \frac{\omega}{2})} \frac{\psi_m(\wp(z),\wp'(z))}{\psi_{m}(\wp\left(z+\frac{\omega}{m}\right),\wp'\left(z+\frac{\omega}{m})\right) } \cdot  
	\] 
	In particular $
	\frac{\sigma\left(z+\frac{\omega}{2}\right)}{\sigma(z)} \rme^{-\eta(\frac{z}{2}+\frac{\omega}{8})} 
	$
	is algebraic over the field $\QQ(g_2,g_3, \wp(z)).$
\end{lemma}

\begin{proof} Replacing $z$ with $z+\frac{\omega}{m}$ in \eqref{sigma(mz)} we obtain
	\[
	\sigma(mz+ \omega)= (-1)^{m-1} \sigma\Big(z+\frac{\omega}{m}\Big)^{m^2} \psi_m \Big(\wp\left(z+\frac{\omega}{m}\right),\wp'\left(z+\frac{\omega}{m}\right) \Big). 
	\]
	Now replacing $z$ with $mz$ in \eqref{sigma(z+omega)} we have
	\begin{align}
		\nonumber	\sigma (mz+   \omega)&= \epsilon(\omega) \sigma(mz) \rme^{ \eta \left(mz+\frac{\omega}{2}\right)},\\
		\nonumber    & \stackrel{\eqref{sigma(mz)}}{=} \epsilon(\omega) (-1)^{m-1} \sigma(z)^{m^2} \psi_m(\wp(z),\wp'(z)) \rme^{ \eta (mz+\frac{\omega}{2})}.
	\end{align} 
	Hence
	\[
	\sigma\Big(z+\frac{\omega}{m}\Big)^{m^2} = \epsilon(\omega)  \sigma(z)^{m^2}  \rme^{\eta \left(mz+\frac{\omega}{2}\right)} \frac{\psi_m(\wp(z),\wp'(z))}{\psi_{m}\left(\wp\left(z+\frac{\omega}{m}\right),\wp'\left(z+\frac{\omega}{m}\right)\right) }\cdotp
	\]
	The final assertion follows from \cite[Lemma 3.1 (3)]{BW}.
\end{proof}

\begin{proposition}\label{Cor:f_q(p)pTorsion}
	Let $z \in \CC \setminus \Omega.$ Let $\omega \in \Omega $ and $m$ an integer, $ m\geqslant 2$, such that $ \frac{\omega}{m} \notin \Omega$.
	The function
	\[
	f_{z}\Big(\frac{\omega}{m}\Big) \rme^{\frac{1}{m} (\omega \zeta(z)- \eta z)} 
	\]
	belongs to the field $\overline{\QQ(g_2,g_3)}(\wp(z),\wp'(z)).$
\end{proposition}

\begin{proof} 
	Using Lemma \ref{LemmaSigma} and Lemma \ref{Prop:SigmaTorsion}, we obtain 
	\[
	f_{z}\Big(\frac{\omega}{m}\Big)^{m^2(m+2)}= \frac{\sigma ( \frac{\omega}{m}+z)^{m^2(m+2)}}{\sigma(\frac{\omega}{m})^{m^2(m+2)} \sigma (z)^{m^2(m+2)} } \rme^{-(m+2){m^2}\zeta(z)\frac{\omega}{m}} 
	\]
	\begin{align}
		\nonumber		&=\Big(	\epsilon(\omega)  \rme^{\eta (mz+ \frac{\omega}{2})} \frac{\psi_m(\wp(z),\wp'(z))}{\psi_{m}(\wp(z+\frac{\omega}{m}),\wp'(z+\frac{\omega}{m})) }	\Big)^{m+2}  \cdot  
		\\
		\nonumber	&  \qquad \qquad \qquad  \cdot \Big(
		\epsilon(\omega) (-1)^m \frac{1}{\psi_{m+1}(\wp(\frac{\omega}{m}),\wp'(\frac{\omega}{m})) }  \rme^{ m(m+2)\frac{\eta \omega}{ 2m^2}} \Big)^{-m}  \rme^{-(m+2){m^2}\zeta(z)\frac{\omega}{m}}	
		\\
		\nonumber	& = (-1)^m  \frac{\psi_m(\wp(z),\wp'(z))^{m+2} \psi_{m+1}(\wp(\frac{\omega}{m}),\wp'(\frac{\omega}{m}))^{m}  }{
			\psi_{m}\left(\wp\left(z+\frac{\omega}{m}\right),\wp'\left(z+\frac{\omega}{m}\right)\right)^{m+2} } \cdot  \\
		\nonumber	&  \qquad \qquad \qquad  \cdot \rme^{ (m+2) \eta (mz+ \frac{\omega}{2})}  \rme^{ -(m+2)\frac{\eta \omega}{ 2}}  \rme^{-(m+2){m^2}\zeta(z)\frac{\omega}{m}}	\\
		\nonumber	& = (-1)^m  \frac{\psi_m(\wp(z),\wp'(z))^{m+2} \psi_{m+1}\left(\wp\left(\frac{\omega}{m}\right),\wp'\left(\frac{\omega}{m}\right)\right)^{m}  }{\psi_{m}(\wp(z+\frac{\omega}{m}),\wp'(z+\frac{\omega}{m}))^{m+2} } \rme^{- m^2 (m+2) (\frac{\omega}{m} \zeta(z)- \frac{\eta}{m}z) }.
	\end{align}
	From these equalities we deduce that the function $f_{\frac{\omega}{m}}(z) \rme^{\frac{1}{m} ( \omega\zeta(z)- \eta z)}$ is elliptic with respect to the lattice $\Omega$, and hence belongs to $\CC(\wp(z),\wp'(z))$. Moreover its coefficients are algebraic over $\QQ(g_2,g_3)$ by \cite[Lemma 3.1 (3)]{BW}. Hence we conclude. 
\end{proof}

\begin{corollary}\label{Cor:f_q(p)qpTorsion}
	Let $\omega$ and $\omega'$ be in $\Omega$ and $m,l$ be two integers, $ m,l\geqslant 2$, such that $ \frac{\omega}{m}, \frac{\omega'}{l} \notin \Omega.$ Then the number
	\[
	f_{\frac{\omega'}{l}}\Big(\frac{\omega}{m}\Big) \rme^{\frac{1}{m} ( \omega \zeta(\frac{\omega'}{l})- \eta\frac{\omega'}{l})} 
	\]
	belongs to the field $\overline{\QQ(g_2,g_3)}$. In particular, if $ \frac{\omega'}{2} $ is not a period, then $	f_{\frac{\omega'}{2}}\big(\frac{\omega}{m}\big)  \in \overline{\QQ(g_2,g_3)}. $
\end{corollary}

\begin{proof} By \cite[Lemma 3.1 (3)]{BW} the number $\wp(\frac{\omega'}{l})$ is algebraic over the field $\QQ(g_2,g_3)$. Therefore Proposition \ref{Cor:f_q(p)pTorsion}  yields the expected result. For the last statement use the generalized Legendre relation \eqref{lemma:MixProduct}.
\end{proof}

Now we state without proof dual results to Proposition \ref{Cor:f_q(p)pTorsion} and Corollary \ref{Cor:f_q(p)qpTorsion} respectively.

\begin{proposition}\label{Cor:f_q(z)qTorsion}
	Let $\omega \in \Omega $ and $m$ an integer, $ m\geqslant 2$, such that $ \frac{\omega}{m} \notin \Omega$.
	The function 
	\[
	f_{\frac{\omega}{m}}(z) \rme^{(\zeta(\frac{\omega}{m})-\frac{\eta}{m})z} 
	\]
	belongs to the field $\overline{\QQ(g_2,g_3)}(\wp(z),\wp'(z))$. In particular, if $ \frac{\omega}{2} $ is not a period, then $f_{\frac{\omega}{2}}(z) \in \overline{\QQ(g_2,g_3)}(\wp(z),\wp'(z)). $
\end{proposition}

\begin{corollary}\label{Cor:f_q(p)pqTorsion}
	Let $\omega$ and $\omega'$ be in $\Omega$ and $m,l$ be two integers, $ m,l\geqslant 2$, such that $ \frac{\omega}{m}, \frac{\omega'}{l} \notin \Omega.$ Then the number
	\[
	f_{\frac{\omega}{m}}\Big(\frac{\omega'}{l}\Big) \rme^{(\zeta(\frac{\omega}{m})-\frac{\eta}{m})\frac{\omega'}{l}} 
	\]
	belongs to the field $\overline{\QQ(g_2,g_3)}$. In particular, if $ \frac{\omega}{2} $ is not a period, then $	f_{\frac{\omega}{2}}\big(\frac{\omega'}{l}\big) \in \overline{\QQ(g_2,g_3)}. $
\end{corollary}



	For the convenience of the reader we summarize the values of the Weierstrass $\wp, \zeta, \sigma$ functions and of Serre functions at division and torsion points computed in 
	\cite[\S 3]{BW} and in this section. 
Let $p  \in \CC \setminus \Omega$, $m$ be an integer, $m \geqslant 1 ,$ and $\omega$ be a period, such that $ \frac{\omega}{m} \notin \Omega.$
\medskip
\medskip

\begin{tabular}{ |p{15cm}|  }  
	\hline
	\multicolumn{1}{|c|}{} \\
	\multicolumn{1}{|c|}{\textbf{Values at torsion and division points}} \\
	\multicolumn{1}{|c|}{} \\
	\hline \rule[-4mm]{0cm}{1cm}
	$\wp(p/m) \in \overline{\QQ(g_2,g_3,\wp(p))} $ \\
	\hline \rule[-4mm]{0cm}{1cm}
	$\zeta(p/m) -\zeta(p)/m \in \overline{\QQ(g_2,g_3,\wp(p))} $ \\
	\hline \rule[-4mm]{0cm}{1cm}
 $\wp(\omega/m)
 \in \overline{\QQ(g_2,g_3)}  $ \\ 
 	\hline \rule[-4mm]{0cm}{1cm}
 $\zeta(\omega/m)-\eta(\omega)/m
 \in \overline{\QQ(g_2,g_3)}  $  \\ 
 	\hline \rule[-4mm]{0cm}{1cm}
 $ \sigma (\omega/m) \rme^{-\eta(\omega) \omega /2 m^2}  \in \overline{\QQ(g_2,g_3)} $  \\ 
 \hline \rule[-4mm]{0cm}{1cm}
 $	f_{z}(\omega /m) \rme^{(\omega \zeta(z)- \eta(\omega) z)/m}  \in \overline{\QQ(g_2,g_3)}(\wp(z),\wp'(z))$   \\ 
 	\hline \rule[-4mm]{0cm}{1cm}
 $	f_{\omega/m}(z) \rme^{(\zeta(\omega /m)-\eta(\omega) /m)z} \in \overline{\QQ(g_2,g_3)}(\wp(z),\wp'(z))$  \\
	\hline
\end{tabular}
\begin{center}
	{\bf Table 0}
\end{center}

\section{Addition and multiplication formulas for Serre functions}\label{Formulae}

In this section, we establish addition and multiplication formulas for Serre functions.

\medskip
\medskip

\begin{tabular}{ |p{15cm}|  }  
	\hline
	\multicolumn{1}{|c|}{} \\
	\multicolumn{1}{|c|}{\textbf{Addition Formulas}} \\
	\multicolumn{1}{|c|}{} \\
	\hline \rule[-4mm]{0cm}{1cm}
	$	\frac{f_q(z_1+ z_2)}{f_q(z_1)f_q( z_2)}= \frac{\sigma(z_1+z_2+q) \sigma(z_1)\sigma(z_2)\sigma(q)}{\sigma(z_1+z_2) \sigma(z_1+q)\sigma(z_2+q)} \in \QQ(g_2,g_3,\wp(q),\wp(z_1),\wp(z_2),\wp'(q),\wp'(z_1),\wp'(z_2))$ \\
	\hline \rule[-4mm]{0cm}{1cm}
	$	\frac{f_{z_1+z_2}(p)}{f_{z_1}(p)f_{z_2}( p)}= \frac{\sigma(p+z_1+z_2) \sigma(z_1)\sigma(z_2) \sigma(p)}{\sigma(z_1+z_2)\sigma(p+z_1)\sigma(p+z_2)} \rme^{- \frac{p}{2} \frac{\wp'(z_1)-\wp'(z_2)}{\wp(z_1)-\wp(z_2)}}$ \\
	\hline
\end{tabular}
\begin{center}
	{\bf Table 1}
\end{center}

\medskip
\medskip

\begin{tabular}{ |p{15cm}|  }
	\hline
	\multicolumn{1}{|c|}{} \\
	\multicolumn{1}{|c|}{\textbf{ Multiplication by $\frac{m}{n}$ formulae, with $m,n$ integers, $m,n \not=0$}} \\
	\multicolumn{1}{|c|}{} \\
	\hline \rule[-4mm]{0cm}{1cm}
	$	f_q \big(\frac{mz}{n}\big)^{n^2}  = f_{\frac{n}{m}q}(z)^{m^2}  \; \frac{f_{m,n}(z+\frac{nq}{m}) f_{n,m}(q)}{f_{m,n}(z)} \; \rme^{ \frac{f'_{n,m}(q)}{f_{n,m}(q)} \frac{mz}{n}}$ \\
	\hline \rule[-4mm]{0cm}{1cm}
	$	f_q \big(\frac{mz}{n}\big)^{n}  = f_{q} (z)^m  \; \frac{C_m(\wp(\frac{z}{n}), \wp'(\frac{z}{n}), \wp(q), \wp'(q))^n}{D_m(\wp(\frac{z}{n}), \wp'(\frac{z}{n}), \wp(q), \wp'(q))^n} \; \frac{D_n(\wp(\frac{z}{n}), \wp'(\frac{z}{n}), \wp(q), \wp'(q))^m}{C_n(\wp(\frac{z}{n}), \wp'(\frac{z}{n}), \wp(q), \wp'(q))^m} \quad (\mathrm{here}\; m,n \geqslant 1)$ \\
	\hline
\end{tabular}
\begin{center}
	{\bf Table 2}
\end{center}

\medskip
\medskip

\par\noindent where $f_{n,m}(z)$ is a rational function in $\QQ(g_2,g_3,\wp(\frac{z}{m}), \wp'(\frac{z}{m})), C_n(X,Y,Z,T)$ is a rational function in $ \QQ(g_2,g_3,X,Y,Z,T)$ and $ D_n(X,Y,Z,T) $ is a polynomial in  $ \QQ(g_2,g_3)[X,Y,Z,T]$.
In particular $C_1(X,Y,Z,T) =  D_1(X,Y,Z,T) =X-Z.$

\medskip
\medskip

\begin{tabular}{ |p{15cm}|  }
	\hline
	\multicolumn{1}{|c|}{} \\
	\multicolumn{1}{|c|}{\textbf{ Multiplication by $\tau=\frac{\omega_2}{\omega_1}$ formula}} \\
	\multicolumn{1}{|c|}{(only in the case of complex multiplication)} \\
	\multicolumn{1}{|c|}{} \\
	\hline \rule[-4mm]{0cm}{1cm}
	$	f_q (C \tau z)^2 = f_{\frac{q}{C \tau}}(z)^{2AC} \sigma(\frac{q}{C \tau})^{2AC} \;  \frac{Q(\wp(z+\frac{q}{C \tau})) }{Q(\wp(z)) \sigma(q)^2 } \;  \rme^{-\frac{\kappa q^2}{C \tau}} \;  \rme^{ - 2 z[\zeta(q) C \tau + \kappa q- \zeta(\frac{q}{C \tau}) AC]} $ \\
	\hline 
\end{tabular}

\begin{center}
	{\bf Table 3}
\end{center}

\medskip
\medskip
\par\noindent with $Q(X) \in \QQ(g_2,g_3,\tau)[X].$

\medspace

\textit{Proof of the formulae in Table 1}
The first identity follows directly from the definition of the Serre function. 
For the second equality use the addition formula for the $\zeta$ function \cite[Table 1 \S 3]{BW}

\medspace

\textit{Proof of the formulae in Table 2}
Applying the multiplication by an integer formula for the $\sigma$ function twice \cite[Table 2 \S 3]{BW} we obtain
\begin{align*}
	f_q\Big( \frac{mz}{n}\Big)^{n^2} &= \frac{\sigma(\frac{m}{n}(z+\frac{nq}{m}))^{n^2}}{\sigma(\frac{mz}{n})^{n^2} \sigma(q)^{n^2}}  \; \rme^{ -\zeta(q) mnz}\\
	&=  \frac{\sigma(z+\frac{nq}{m})^{m^2} f_{m,n}(z+\frac{nq}{m})}{\sigma(z)^{m^2} f_{m,n}(z) \sigma(q)^{n^2}} \; \rme^{ -\zeta(q) mnz}.
\end{align*}
From the multiplication by an integer formula for the $\zeta$ function \cite[Table 2 \S 3]{BW}  we get
\[ \zeta \Big(  \frac{nq}{m}\Big)  m^2 z=   \zeta(q) nmz + \frac{f'_{n,m}(q)}{nf_{n,m}(q)} mz\]
\[ \rme^{-   \zeta(q)nmz } = \rme^{-  \zeta (  \frac{nq}{m})m^2 z} \; \rme^{ \frac{f'_{n,m}(q)}{nf_{n,m}(q)}mz } .\]
The multiplication by an integer formula for the $\sigma$ function gives $\sigma(q)^{n^2} = \sigma (\frac{nq}{m})^{m^2} \frac{1}{f_{n,m}(q)}$  and so we can conclude
\begin{align*}
	f_q\Big( \frac{mz}{n}\Big)^{n^2} &= \frac{\sigma(z+\frac{nq}{m})^{m^2} f_{m,n}(z+\frac{nq}{m})}{\sigma(z)^{m^2} f_{m,n}(z) } 
	\;	\frac{f_{n,m}(q)}{\sigma (\frac{nq}{m})^{m^2}} \;  \rme^{- \zeta (  \frac{nq}{m})  m^2 z} \; \rme^{ \frac{f'_{n,m}(q)}{f_{n,m}(q)} \frac{mz}{n}} \\
	&= 	f_{\frac{n}{m}q} (z)^{m^2} \; \frac{ f_{m,n}(z+\frac{nq}{m}) f_{n,m}(q) }{ f_{m,n}(z)} \; \rme^{ \frac{f'_{n,m}(q)}{f_{n,m}(q)}\frac{mz}{n}}.
\end{align*}

For the last equality involving Serre function, using the addition formula for the $\sigma$ function \cite[Table 1 \S 3]{BW} and  the equality \eqref{sigma(mz)}, we have that for any integer $m$, $ m\geqslant 1,$
\begin{equation}\label{eq:f_q(rationelz)}
	( \wp(z) - \wp(q))^m  \Psi_m ( \wp(z), \wp'(z)) \frac{f_q(mz)}{f_q(z)^{m}} = (-1)^{m-1} \frac{\sigma (q-z)^m \sigma(mz+q)}{\sigma(z)^{m+m^2} \sigma(q)^{m+1}.} 
\end{equation}
If we consider the ratio on the right as a function in the variable $z$, it is an elliptic function with zero as unique pole. On the other hand, if we consider it as a function in the variable $q$, it is an elliptic function with zero as unique pole. Hence there exists a rational function $C_m (X,Y,Z,T) \in \QQ(g_2,g_3,X,Y,Z,T)$ such that 
\[C_m (\wp(z),\wp'(z),\wp(q),\wp'(q)) = (-1)^{m-1} \frac{\sigma (q-z)^m \sigma(mz+q)}{\sigma(z)^{m+m^2} \sigma(q)^{m+1}}\]
Set 
\[D_m (X,Y,Z,T) =( X - Z)^m  \Psi_m ( X,Y) \in \QQ(g_2,g_3)[X,Y,Z,T].\]
From the equality \eqref{eq:f_q(rationelz)}, we obtain that 
\[f_q(mz) = f_q(z)^m \frac{C_m(\wp(z), \wp'(z), \wp(q), \wp'(q))}{D_m(\wp(z), \wp'(z), \wp(q), \wp'(q))} \]
Hence, for two integers $m,n \geqslant 1,$
\[ \frac{f_q(mz)^n}{f_q(nz)^m} =   \frac{C_m(\wp(z), \wp'(z), \wp(q), \wp'(q))^n}{D_m(\wp(z), \wp'(z), \wp(q), \wp'(q))^n} \; \frac{D_n(\wp(z), \wp'(z), \wp(q), \wp'(q))^m}{C_n(\wp(z), \wp'(z), \wp(q), \wp'(q))^m}\cdotp\]
Replacing $z$ with $\frac{z}{n}$ we get 
\[f_q\Big(\frac{mz}{n}\Big)^n =  f_q(z)^m \; \frac{C_m(\wp(\frac{z}{n}), \wp'(\frac{z}{n}), \wp(q), \wp'(q))^n}{D_m(\wp(\frac{z}{n}), \wp'(\frac{z}{n}), \wp(q), \wp'(q))^n} \; \frac{D_n(\wp(\frac{z}{n}), \wp'(\frac{z}{n}), \wp(q), \wp'(q))^m}{C_n(\wp(\frac{z}{n}), \wp'(\frac{z}{n}), \wp(q), \wp'(q))^m}\cdotp\]

\medspace

\textit{Proof of the formulae in Table 3}
Applying twice the multiplication by $C \tau$ formula for the $\sigma$ function \cite[Table 3 \S 3]{BW} we have
\begin{align*}
	f_q(C \tau z)^{2} &= \frac{\sigma(C \tau(z+\frac{q}{C \tau}))^2}{\sigma(C \tau z )^2 \sigma(q)^2}  \; \rme^{ - 2\zeta(q) C \tau z}\\
	&=  \frac{\sigma(z+\frac{q}{C \tau})^{2AC} \rme^{-\kappa C \tau (z+\frac{q}{C \tau})^2} Q ( \wp(z+\frac{q}{C \tau}))}{\sigma(z)^{2AC} \rme^{-\kappa C \tau z^2} Q ( \wp(z)) \sigma(q)^2} \; \rme^{ - 2\zeta(q) C \tau z}.
\end{align*}
On the other hand we have 
\[ 	f_{\frac{q}{C \tau}}(z)^{2AC} = \frac{\sigma(z+\frac{q}{C \tau})^{2AC}}{\sigma( z )^{2AC} \sigma(\frac{q}{C \tau})^{2AC}}  \; \rme^{ - 2 \zeta(\frac{q}{C \tau}) AC z}.\]
Hence we can conclude
\begin{align*}
	f_q(C \tau z)^{2} &= 	f_{\frac{q}{C \tau}}(z)^{2AC}  \sigma\Big(\frac{q}{C \tau}\Big)^{2AC} \rme^{  2 \zeta(\frac{q}{C \tau}) AC z}  \frac{Q ( \wp(z+\frac{q}{C \tau})) }{Q ( \wp(z)) \sigma(q)^2 \rme^{ \frac{\kappa q^2}{C \tau} }}  \; \rme^{ - 2 z[\zeta(q) C \tau + \kappa q]} \\
	&= 	f_{\frac{q}{C \tau}}(z)^{2AC}  \sigma\Big(\frac{q}{C \tau}\Big)^{2AC}   \frac{Q ( \wp(z+\frac{q}{C \tau})) }{Q ( \wp(z)) \sigma(q)^2 \rme^{ \frac{\kappa q^2}{C \tau} }}  \; \rme^{ - 2 z[\zeta(q) C \tau + \kappa q- \zeta(\frac{q}{C \tau}) AC]}.
\end{align*}


\section{Action of endomorphisms on the Weierstrass $\sigma$ function and on Serre functions}\label{actionEndo}

We extend \cite[Proposition 3.3]{BW} to the Weierstrass $\sigma$ function and to Serre functions.
Concerning Serre functions, particular attention is devoted to antisymmetric endomorphisms, namely endomorphisms $\phi$ satisfying $\phi+ \overline {\phi}=0.$ Such endomorphisms play a distinguished role in the geometry of 1-motives, since they give rise to the exceptional pairs $(P,Q)$ for which the Lie bracket contribution vanishes.

\begin{proposition}\label{Proposition:MultiplicationFormulaSigma}
	Let $\alpha$ be a nonzero element of $k$, and write $\alpha=r_1+r_2\tau$ where $r_1, r_2 \in \QQ $ are not both zero. Let $m\in\ZZ$ be the least positive integer such that $mr_1$ and $mr_2/C$ are integers. Then
	the function $\Sigma_{r_1,r_2},$ defined by 
	\[
	\sigma\bigl( \alpha z \bigr) \sigma \bigl( \overline{\alpha} z \bigr)= \varepsilon(r_1) \sigma(z)^{2(r_1^2+ \frac{A r_2^2}{C})} \rme^{- \frac{\kappa  r_2^2 \tau z^2}{C}} \Sigma_{r_1,r_2}(z),
	\] 
	where $\varepsilon(\alpha )=-1$ if $r_1$ is zero, 1 otherwise, belongs to the field $k(g_2,g_3,\wp(z/m),\wp'(z/m))$. 
\end{proposition} 

\begin{proof}
	Write $r_1=n_1/m$ and $r_2=Cn_2/m$, so that $n_1$ and $n_2$ are integers and 
	\[
	m\alpha=n_1+n_2 C\tau\in \End(\cE)\smallsetminus\{0\}.
	\] 
	If $r_2=0$ and $r_1\not=0$ we apply the third row of Table 2 in \cite{BW}: 
	\[
	\Sigma_{r_1,0}(z)=f_{n_1,m}(z)^{- 2m^2}\cdotp
	\]
	In case $r_1=0$ and $r_2=C$, we have $n_2=m=1$ and we apply the third row of Table 3 in \cite{BW}: 
	\[
	\Sigma _{0,C}(z)= (C \tau)^2 Q(\wp(z)) \cdotp
	\]
Combining these two cases, we obtain the result when $r_1=0$ and $r_2\not=0$ with
	\[
	\Sigma_{0,r_2}(z)=  (\Sigma_{0,C}(z))^{(\frac{r_2}{C})^2} \; \Sigma_{r_2/C,0}(C\tau z).
	\] 
	Finally when $r_1r_2\not=0$ we apply the third row of Table 1 in \cite{BW}:
	\[
	\Sigma_{r_1,r_2}=\Sigma_{r_1,0}(z)  \; \Sigma_{0,r_2}(z) \; \big(\wp(r_2 \tau z)-\wp(r_1 z)\big)\cdotp
	\]
\end{proof}

The following Lemma explains why in Table 3 we do not have written an expression of the complex number $	f_q (C \tau z)$ in terms of $f_q (z).$

\begin{lemma}\label{ExtensionEnd}
	If a non rational endomorphism of the elliptic curve $\cE$ lifts to an endomorphism of the extension $G$, then $G$ is parametrized by torsion points.
\end{lemma}

\begin{proof}
	Let $G$ be the extension of $\cE$ by $\GG_m$ parametrized by the point $Q \in \cE (\CC).$ According to \cite[Proposition 1]{Bertrand} the endomorphism $ \tau: \cE \to \cE$ extends to an endomorphism of $G$ if and only if the $\ZZ$--module $\ZZ Q$ is stable under the transposed endomorphism $ \overline{\tau}: \cE^* \to \cE^*,$ where $\overline{\tau}$ is the complex conjugate of $\tau$, that is if and only if $ \overline{\tau} Q=NQ$ for some integer $N.$ The relation then
	\[ (A+B\overline{\tau} +C\overline{\tau}^2) Q=0\]
	implies that the point $Q \in \cE^* (\CC)$ is a point of $A+BN+CN^2$ torsion, that is $q \in  (\Omega \otimes_\ZZ \QQ)\smallsetminus \Omega .$ 
	
	If the extension $G$ is parametrized by several points of $\cE^* $ we reduce to the case of one point using the additivity of the category of extensions.
\end{proof}

\begin{proposition}\label{Proposition:PhiAntisymmetricP}
	Let $q_1$ and $q_2$ be two complex numbers in $\CC \smallsetminus \Omega,$ for which it exists a period $\omega$ and an integer $m$, $m \geqslant 2 ,$ such that $ \frac{\omega}{m} \notin \Omega$ and $q_1 + q_2 = \frac{\omega}{m}.$ Then the following statements hold.

	\begin{enumerate}
				\item The function $\Phi_1(z, q_1, \frac{\omega}{m})$ defined by 
		\[f_{q_1}(z) f_{q_2}(z)= (\wp(z) -\wp( q_1)) 	\rme^{ \frac{z}{2} \frac{\wp'( q_1)-\wp'(q_2)}{\wp( q_1)-\wp(q_2)}} f_{\frac{\omega}{m}}(z)  \Phi_1\Big(z, q_1, \frac{\omega}{m}\Big)\]
		belongs to $\QQ(g_2,g_3,\wp(z), \wp'(z),\wp( \frac{\omega}{m}), \wp'( \frac{\omega}{m}) ,\wp(q_1), \wp'(q_1)).$ 
		\item $f_{q_2}(z) = \rme^{- \frac{z}{2} \frac{\wp'(- q_1)-\wp'(\frac{\omega}{m})}{\wp(- q_1)-\wp(\frac{\omega}{m})}} f_{-q_1}(z) f_{\frac{\omega}{m}}(z)  \Phi_1\big(z, q_1, \frac{\omega}{m}\big).$
			\end{enumerate}
			
			Moreover, for any $q \in \CC \smallsetminus \Omega,$ we have the equality
		\[	f_{q}(z)f_{-q}(z) = \wp(z) - \wp (q) \in \QQ (g_2,g_3, \wp(z),\wp(q)).\footnote{The three equalities of this Proposition imply that
			\[ \frac{\wp'(- q_1)-\wp'(\frac{\omega}{m})}{\wp(- q_1)-\wp(\frac{\omega}{m})} +  \frac{\wp'( q_1)-\wp'(q_2)}{\wp( q_1)-\wp(q_2)} =0 .\]}\]
\end{proposition}

\begin{proof}
(1) Using the addition formula for Serre function Table 1 \S 3, the oddness of the Weierstrass $\sigma$ function and Lemmas \ref{Prop:SigmaTorsion} and \ref{LemmaSigma} we have
\[
\frac{	f_{\frac{\omega}{m}}(z) } {f_{q_1}(z) f_{q_2}(z) }=  
\frac{\sigma (z+ \frac{\omega}{m}) \sigma(q_1) \sigma (- q_1 +\frac{\omega}{m})\sigma (z) }{\sigma (\frac{\omega}{m}) \sigma (z+ q_1 ) \sigma (z - q_1 +\frac{\omega}{m})} 	\rme^{- \frac{z}{2} \frac{\wp'( q_1)-\wp'(q_2)}{\wp( q_1)-\wp(q_2)}}  \]
	\begin{align*}
	& = 	\frac{ \sigma(q_1)\sigma (z) }{\sigma (z+ q_1 ) }
	\rme^{- \frac{z}{2} \frac{\wp'( q_1)-\wp'(q_2)}{\wp( q_1)-\wp(q_2)}}\cdot\\
	\nonumber	&  \qquad \qquad  \cdot \sigma(z) \rme^{\frac{\eta}{m^2} (mz+ \frac{\omega}{2})} \Big( \epsilon(\omega)   \frac{\psi_m(\wp(z),\wp'(z))}{\psi_{m}(\wp(z+\frac{\omega}{m}),\wp'(z+\frac{\omega}{m})) }\Big)^{\frac{1}{m^2}} \\
		\nonumber	&  \qquad \qquad  \cdot \sigma(-q_1)   \rme^{\frac{\eta}{m^2} (m(-q_1)+ \frac{\omega}{2})} \Big( \epsilon(\omega) \frac{\psi_m(\wp(-q_1),\wp'(-q_1))}{\psi_{m}(\wp(-q_1+\frac{\omega}{m}),\wp'(-q_1+\frac{\omega}{m})) }\Big)^{\frac{1}{m^2}} \\
			\nonumber	&  \qquad \qquad  \cdot \frac{1}{ \rme^{ \frac{\eta \omega}{ 2m^2}}} \Big( \epsilon(\omega) (-1)^m \frac{1}{\psi_{m+1}(\wp(\frac{\omega}{m}),\wp'(\frac{\omega}{m})) } \Big)^{-\frac{1}{m(m+2)}} \cdot\\
	\nonumber	&  \qquad \qquad  \cdot \frac{1}{	 \sigma (z- q_1) \rme^{ \frac{\eta}{m^2} (m(z- q_1)+ \frac{\omega}{2})}} \Big(\epsilon(\omega)  \frac{\psi_m(\wp(z- q_1),\wp'(z- q_1))}{\psi_{m}(\wp(z- q_1+\frac{\omega}{m}),\wp'(z- q_1+\frac{\omega}{m})) } \Big)^{-\frac{1}{m^2}} \cdot\\
	&= -	\frac{ \sigma(q_1)^2\sigma (z)^2 }{\sigma (z+ q_1 ) \sigma (z- q_1 )} 	\rme^{- \frac{z}{2} \frac{\wp'( q_1)-\wp'(q_2)}{\wp( q_1)-\wp(q_2)}} \epsilon(\omega)^{ \frac{2}{m^2(m+2)}} (-1)^{-\frac{1}{m+2}} \\
	\nonumber	&  \qquad \qquad  \cdot  \Big(  \frac{\psi_m(\wp(z),\wp'(z))}{\psi_{m}(\wp(z+\frac{\omega}{m}),\wp'(z+\frac{\omega}{m})) }\Big)^{\frac{1}{m^2}} \Big(  \frac{\psi_m(\wp(-q_1),\wp'(-q_1))}{\psi_{m}(\wp(-q_1+\frac{\omega}{m}),\wp'(-q_1+\frac{\omega}{m})) }\Big)^{\frac{1}{m^2}} \cdot \\
	\nonumber	&  \qquad \qquad  \cdot  \psi_{m+1}\Big(\wp(\frac{\omega}{m}),\wp'(\frac{\omega}{m}) \Big)^{\frac{1}{m(m+2)}}    \Big(  \frac{\psi_m(\wp(z- q_1),\wp'(z- q_1))}{\psi_{m}(\wp(z- q_1+\frac{\omega}{m}),\wp'(z- q_1+\frac{\omega}{m})) } \Big)^{-\frac{1}{m^2}}\\
	&= \frac{1}{\wp(z) -\wp( q_1)} 	\rme^{- \frac{z}{2} \frac{\wp'( q_1)-\wp'(q_2)}{\wp( q_1)-\wp(q_2)}} \epsilon(\omega)^{ \frac{2}{m^2(m+2)}} (-1)^{-\frac{1}{m+2}} \\
	\nonumber	&  \qquad \qquad  \cdot  \Big(  \frac{\psi_m(\wp(z),\wp'(z))}{\psi_{m}(\wp(z+\frac{\omega}{m}),\wp'(z+\frac{\omega}{m})) }\Big)^{\frac{1}{m^2}} \Big(  \frac{\psi_m(\wp(-q_1),\wp'(-q_1))}{\psi_{m}(\wp(-q_1+\frac{\omega}{m}),\wp'(-q_1+\frac{\omega}{m})) }\Big)^{\frac{1}{m^2}} \cdot \\
	\nonumber	&  \qquad \qquad  \cdot  \psi_{m+1}\Big(\wp(\frac{\omega}{m}),\wp'(\frac{\omega}{m}) \Big)^{\frac{1}{m(m+2)}}    \Big(  \frac{\psi_m(\wp(z- q_1),\wp'(z- q_1))}{\psi_{m}(\wp(z- q_1+\frac{\omega}{m}),\wp'(z- q_1+\frac{\omega}{m})) } \Big)^{-\frac{1}{m^2}}.\\
\end{align*}
	Define $ \Phi_1(z, q_1, \frac{\omega}{m})$ as the inverse of the function 
\[ \epsilon(\omega)^{ \frac{2}{m^2(m+2)}} (-1)^{-\frac{1}{m+2}}	 \Big(  \frac{\psi_m(\wp(z),\wp'(z))}{\psi_{m}(\wp(z+\frac{\omega}{m}),\wp'(z+\frac{\omega}{m})) }\Big)^{\frac{1}{m^2}} 
 \Big(  \frac{\psi_m(\wp(-q_1),\wp'(-q_1))}{\psi_{m}(\wp(-q_1+\frac{\omega}{m}),\wp'(-q_1+\frac{\omega}{m})) }\Big)^{\frac{1}{m^2}}\cdot \]
\[   \psi_{m+1}\Big(\wp(\frac{\omega}{m}),\wp'(\frac{\omega}{m}) \Big)^{\frac{1}{m(m+2)}}    \Big(  \frac{\psi_m(\wp(z- q_1),\wp'(z- q_1))}{\psi_{m}(\wp(z- q_1+\frac{\omega}{m}),\wp'(z- q_1+\frac{\omega}{m})) } \Big)^{-\frac{1}{m^2}}.\]

	(2) Using the addition formula for Serre function Table 1 \S 3 and Lemmas \ref{Prop:SigmaTorsion} and \ref{LemmaSigma} we have
	\[
	\frac{	f_{q_2}(z) } {f_{-q_1}(z) f_{\frac{\omega}{m}}(z) }=  
		\frac{\sigma (z- q_1 +\frac{\omega}{m}) \sigma(-q_1) \sigma (\frac{\omega}{m})\sigma (z) }{\sigma (-q_1 +\frac{\omega}{m}) \sigma (z- q_1 ) \sigma (z +\frac{\omega}{m})} 	\rme^{- \frac{z}{2} \frac{\wp'(- q_1)-\wp'(\frac{\omega}{m})}{\wp(- q_1)-\wp(\frac{\omega}{m})}}  \]
		\begin{align*}
		& = 	\frac{ \sigma(-q_1)\sigma (z) }{\sigma (z- q_1 ) }
		\rme^{- \frac{z}{2} \frac{\wp'(- q_1)-\wp'(\frac{\omega}{m})}{\wp(- q_1)-\wp(\frac{\omega}{m})}} \cdot\\
			\nonumber	&  \qquad \qquad  \cdot 	 \sigma (z- q_1) \rme^{ \frac{\eta}{m^2} (m(z- q_1)+ \frac{\omega}{2})} \Big(\epsilon(\omega)  \frac{\psi_m(\wp(z- q_1),\wp'(z- q_1))}{\psi_{m}(\wp(z- q_1+\frac{\omega}{m}),\wp'(z- q_1+\frac{\omega}{m})) } \Big)^{\frac{1}{m^2}} \cdot\\
	\nonumber	&  \qquad \qquad  \cdot \rme^{ \frac{\eta \omega}{ 2m^2}} \Big( \epsilon(\omega) (-1)^m \frac{1}{\psi_{m+1}(\wp(\frac{\omega}{m}),\wp'(\frac{\omega}{m})) } \Big)^{\frac{1}{m(m+2)}} \cdot\\
	\nonumber	&  \qquad \qquad  \cdot \frac{1}{\sigma(-q_1)   \rme^{\frac{\eta}{m^2} (m(-q_1)+ \frac{\omega}{2})}}  \Big( \epsilon(\omega) \frac{\psi_m(\wp(-q_1),\wp'(-q_1))}{\psi_{m}(\wp(-q_1+\frac{\omega}{m}),\wp'(-q_1+\frac{\omega}{m})) }\Big)^{-\frac{1}{m^2}} \\
	\nonumber	&  \qquad \qquad  \cdot \frac{1}{\sigma(z) \rme^{\frac{\eta}{m^2} (mz+ \frac{\omega}{2})}}  \Big( \epsilon(\omega)   \frac{\psi_m(\wp(z),\wp'(z))}{\psi_{m}(\wp(z+\frac{\omega}{m}),\wp'(z+\frac{\omega}{m})) }\Big)^{-\frac{1}{m^2}} \\
	&=  \epsilon(\omega)^{- \frac{2}{m^2(m+2)}} (-1)^{\frac{1}{m+2}}	\rme^{- \frac{z}{2} \frac{\wp'(- q_1)-\wp'(\frac{\omega}{m})}{\wp(- q_1)-\wp(\frac{\omega}{m})}} \Big(  \frac{\psi_m(\wp(z- q_1),\wp'(z- q_1))}{\psi_{m}(\wp(z- q_1+\frac{\omega}{m}),\wp'(z- q_1+\frac{\omega}{m})) } \Big)^{\frac{1}{m^2}} \cdot \\
		\nonumber	&  \qquad \qquad  \cdot \psi_{m+1}\Big(\wp(\frac{\omega}{m}),\wp'(\frac{\omega}{m}) \Big)^{-\frac{1}{m(m+2)}}   \Big(  \frac{\psi_m(\wp(-q_1),\wp'(-q_1))}{\psi_{m}(\wp(-q_1+\frac{\omega}{m}),\wp'(-q_1+\frac{\omega}{m})) }\Big)^{-\frac{1}{m^2}} \cdot \\
			\nonumber	&  \qquad \qquad  \cdot  \Big(  \frac{\psi_m(\wp(z),\wp'(z))}{\psi_{m}(\wp(z+\frac{\omega}{m}),\wp'(z+\frac{\omega}{m})) }\Big)^{-\frac{1}{m^2}} \\
			&= \rme^{- \frac{z}{2} \frac{\wp'(- q_1)-\wp'(\frac{\omega}{m})}{\wp(- q_1)-\wp(\frac{\omega}{m})}}  \Phi_1\Big(z, q_1, \frac{\omega}{m}\Big).
	\end{align*}

	(3)  Using the addition formula for the $\sigma$ function \cite[Table 1 \S 3]{BW} and the oddness of the Weierstrass functions $\zeta$ and $\sigma$ we get
\begin{align*}
	f_{q}(z)f_{-q}(z) &= \frac{\sigma (z+ q) }{ \sigma (z)\sigma(q)} 
	\frac{\sigma (z- q) }{ \sigma (z)\sigma(- q)} \\
	& =-\frac{\sigma (z)^2  \sigma(q)^2 \big( \wp (q) - \wp(z) \big)}
	{ \sigma (z)^2\sigma( q)^2}\\
	&= \wp(z) - \wp (q)\in \QQ (g_2,g_3, \wp(z),\wp(q)).
\end{align*}

\end{proof}

Because of the dual nature of 1-motives, we have the following dual result to
Proposition \ref{Proposition:PhiAntisymmetricP} that we leave to the reader:

\begin{proposition}\label{Proposition:PhiAntisymmetricQ}
Let $p_1$ and $p_2$ be two complex numbers, for which it exists a period $\omega$ and an integer $m$, $m \geqslant 2 ,$ such that $ \frac{\omega}{m} \notin \Omega$ and $p_1 + p_2 = \frac{\omega}{m}.$ Then the following statements hold.
	
	\begin{enumerate}
		\item The function $\Phi_1(z, p_1, \frac{\omega}{m})$ defined by 
		\[f_{z}(p_1) f_{z}(p_2)= \big(\wp(z) -\wp( p_1)\big) f_{z}\Big(\frac{\omega}{m}\Big)  \Phi_1\Big(z, p_1, \frac{\omega}{m}\Big)\]
		belongs to $\QQ(g_2,g_3,\wp(z), \wp'(z),\wp( \frac{\omega}{m}), \wp'( \frac{\omega}{m}) ,\wp(p_1), \wp'(p_1)).$ 
		\item $f_{z}(p_2) = f_{z}(-p_1) f_{z}\big(\frac{\omega}{m}\big)    \Phi_1\big(z, p_1, \frac{\omega}{m}\big).$
	\end{enumerate}
	Moreover, for any $p \in \CC$ we have the equality
	\[	f_{z}(p)f_{z}(-p) = \wp(z) - \wp (p) \in \QQ (g_2,g_3, \wp(z),\wp(p)). \]
\end{proposition}

\section{Dimensions of motivic Galois groups of 1-motives}\label{MotGalGroup}
	Let $M=[u:\ZZ \rightarrow  G^n   ],
u(1)=(R_1, \dots, R_n ) \in G^n(\CC),$ be the 1-motive \eqref{GPC-1-motive}. Denote by 
 $ \Galmot (M)$ its motivic Galois group. 
By the weight filtration, the motivic Galois group $ \Galmot (M)$ of $M$ fits into 
the exact sequence 
\begin{equation}\label{eq:shortexactsequenceUR}
	0 \longrightarrow \UR(M) \longrightarrow \Galmot (M) \longrightarrow \Galmot (\cE) \longrightarrow 0
\end{equation} 
where $\UR(M)$ is its unipotent radical and $\Galmot (\cE)$ is the motivic Galois group of $\cE$, which is its largest reductive quotient (see for example \cite[\S 3.1]{BPSS}). 
In order to describe the Lie algebra of the unipotent radical $\UR(M)$ we proceed in three steps:

\begin{enumerate}
	\item Let 
\[B\]
 be 
the smallest abelian sub-variety (modulo isogenies) of $\cE^n\times \cE^{*r}$
which contains a multiple of the point $(P_1,\dots,P_n,Q_1,\dots,Q_r) \in \cE^n \times \cE^{*r} (\CC).$ \textit{Its dimension is governed by the points $P_1,\dots,P_n,Q_1,\dots,Q_r $} and
$$\dim B \leqslant r+n .$$

\medspace

\item  Using the Poincar\'e biextension $\mathcal{P}$ of $(\cE,\cE^*)$ by $\GG_m,$
in \cite[Example 2.8]{B03} we have constructed explicitly a biextension $\mathcal{B}$ of $(\cE^n \times \cE^{*r},\cE^n \times \cE^{*r})$ by $\GG_m^{nr},$ whose pull-back $d^* \mathcal{B}$ via the diagonal morphism $d:\cE^n \times \cE^{*r} \to (\cE^n \times \cE^{*r}) \times (\cE^n \times \cE^{*r})$ is a $\GG_m^{nr}$-torsor over $\cE^n \times \cE^{*r}$ inducing a Lie bracket 
$ [\cdot,\cdot]: (\cE^n \times \cE^{*r}) \otimes (\cE^n \times \cE^{*r}) \to \GG_m^{nr}$
(see \cite[Lemma 3.3, p.600]{B03} and see \cite[(2.8.4)]{B03} for an explicit description of this Lie bracket).
Now we restrict the basis of the $\GG_m^{nr}$-torsor $d^* \mathcal{B}$ to the abelian variety $B$ by taking the pull-back 
$I^*d^* \mathcal{B}$ of $d^* \mathcal{B}$ via the inclusion $I: B \hookrightarrow \cE^n \times \cE^{*r}$. Let
 \[Z'(1)\]
  be the smallest sub-torus of $\GG_m^{nr}$ which contains the image of the restriction of Lie bracket to $B$, that is the image of $[\cdot,\cdot]: B \otimes B \to \GG_m^{nr}.$ \textit{Its dimension is governed by the abelian variety $B$} (see \cite[4.2]{BP}) and
  $$\dim Z'(1) \leqslant nr.$$
In \cite[Lemma 3.1]{BP} we have showed that $Z'(1 )$ coincides with the smallest sub-torus of $\GG_m^{nr}$ which contains the values of the factor of automorphy of the  $\GG_m^{nr}$-torsor $I^*d^*\mathcal{B}$. In particular,  the push-down ${pr}_*I^*d^* \mathcal{B}$ via the projection $pr:\GG_m^{nr} \twoheadrightarrow \GG_m^{nr} / Z'(1)$ of the torsor $I^*d^* \mathcal{B}$
 is \textit{the trivial $\GG_m^{nr}  / Z'(1)$-torsor over $B$}, i.e. 
${pr}_*I^*d^* \mathcal{B}= B \times \GG_m^{nr} / Z'(1) .$

\medspace

\item Let $ v:\ZZ^n\to \cE$ and $v^*:\ZZ^r\to \cE^* $ 
be the group homomorphisms determined by $v(x_i)=P_i$ and $v^*(y_j^\vee)=Q_j,$
where $x_1,\ldots,x_n$ and $y_1^\vee,\ldots,y_r^\vee$ denote the standard bases of
\(\ZZ^n\) and \(\ZZ^r\), respectively. Let $G_j$ be the extension of the elliptic curve $\cE$ by $\GG_m$ parametrized by the point $Q_j=\exp_{\cE^*}(q_j).$ The additivity of the category of extensions implies that
$G$ is isomorphic to the extension  $G_{1}\times \dots\times G_{r}$.
Therefore for $i=1, \dots , n,$ having the point $R_i $ \eqref{R} in the fibre $G_{P_i}$ of $G$ above the point $P_i $ is equivalent to having the $r$ points $R_{i1}=\exp_{G_1}(p_i,t_{i1}), \dots , R_{ir}=\exp_{G_r}(p_i,t_{ir})$ in the fibres $(G_{1})_{P_i} , \dots , (G_{r})_{P_i}$ respectively.
By \cite[\S 1.2]{BP}, having the group homomorphism  $u: \ZZ \to G^n, u(1)=(R_{ij})_{i,j}$ is equivalent to having
a trivialization (= biadditive section) $\psi : \ZZ \times \ZZ \longrightarrow  (v \times v^*)^* \cP$ of the pull-back $(v \times v^*)^* \cP$ via $v \times v^*$ of the Poincar\'e biextension $\mathcal{P}$ of $(\cE,\cE^*)$ by $\GG_m$:
\[R_{ij} = \psi(x_i,y^\vee_j) \in \cP_{P_i,Q_j} \simeq (G_j)_{P_i}. \]

\par\noindent Consider the two group homomorphisms $V: \ZZ \to \cE^n$ and $V^*: \ZZ \to \cE^{*r} $  defined by the points $P_1,\dots,P_n$ in $ \cE^n(\CC)$ and $Q_1,\dots,Q_r $ in $ \cE^{*r}(\CC)$ respectively. According to \cite[\S 3.1]{BP}, having
the trivialization $\psi : \ZZ \times \ZZ \longrightarrow  (v \times v^*)^* \cP$ is equivalent to having
 a trivialization  $\Psi : \ZZ \times \ZZ \longrightarrow  (V \times V^*)^* I^*d^* \mathcal{B}$ of the pull-back $(V \times V^*)^*I^*d^* \mathcal{B}$ via $V \times V^*$ of the $\GG_m^{nr}$-torsor $I^*d^* \mathcal{B}$ over $B$.
 The trivialization $\Psi$ defines  
 a point 
 $	\Psi(1,1) = \big(\psi(x_i,y^\vee_j)\big)_{i,j }  \in   \big( (V \times V^*)^* I^*d^*\mathcal{B}\big)_{1,1} $
 which in turn furnishes a point
 \[\tR \in (I^*d^*\mathcal{B})_{(P,Q)}\]
 in the fibre of $I^*d^*\mathcal{B}$ over the point $(P,Q)=(P_1,\dots,P_n,Q_1,\dots,Q_r)  \in  B .$ 
  Because of the equality
 $ (V \times V^*)^*  {pr}_*I^*d^* \mathcal{B}  =  {pr}_*  (V \times V^*)^* I^*d^* \mathcal{B} ,$
 the  trivialization $\Psi$ defines a trivialization 
 $
 	{pr}_*\Psi : \ZZ \times \ZZ \longrightarrow  (V \times V^*)^*  {pr}_* I^*d^* \mathcal{B} 
$
 of the pull-back via $ V \times V^*$ of the trivial torsor $  {pr}_*I^*d^* \mathcal{B}$. Denote by $\pi: {pr}_*I^*d^* \mathcal{B} \twoheadrightarrow \GG_m^{nr}/Z'(1)$ the projection on the second factor.
 The point 
$	{pr}_*\Psi(1,1) \in  \big(  (V \times V^*)^* {pr}_* I^*d^*\mathcal{B}\big)_{1,1} $
 corresponds to the point 
 \[	pr_* \tR =\big((P,Q), \pi (pr_* \tR)\big) =  \big((P,Q), \pi (pr_*  (\psi(x_i,y^\vee_j))_{i= 1, \dots, n \atop j = 1, \dots, r } )\big).
 \]
 in the fibre of the trivial torsor $pr_* I^*d^*\mathcal{B} =B\times \GG_m^{nr}/Z'(1)$ over the point $(P,Q) \in  B.$ Let
 \[Z(1)\]
 be the smallest sub-torus of $\GG_m^{nr}$ which contains $Z'(1)$ and such that $Z(1)/Z'(1)$ contains the point 
\begin{equation}\label{piprR}
	\pi (pr_* \tR) =\pi \big(pr_*  \big(\psi(x_i,y^\vee_j)\big)_{i= 1, \dots, n \atop j = 1, \dots, r } \big) .
\end{equation}
 \textit{The dimension of the torus $Z(1)/Z'(1)$ is governed by the point  $\pi (pr_* \tR)$} (see \cite[4.3]{BP}) and
  $$\dim Z(1) =  \dim Z'(1) + \dim Z(1)/Z'(1)  \leqslant nr.$$
Notice that the point $
\pi(pr_*\widetilde R)$
depends on the trivialization $\psi$ and therefore on the points
$R_{ij}$ defining the $1$-motive.
 \end{enumerate}

  By 
\cite[Theorem 1.2.1]{A19} the motivic Galois group of $M$ coincides with its Mumford-Tate group and hence
\cite[Lemma 3.5]{BPSS} implies that $\dim \UR(M)= 2 \dim B + \dim Z(1)$ (recall that in the introduction we have assumed the field of definition $K$ of $M$ to be algebraically closed).
From the short exact sequence \eqref{eq:shortexactsequenceUR} we get that 
\begin{equation} \label{DimGalmot}
 	\dim \Galmot (M)= \dim \Galmot (\cE) + 2 \dim B + \dim Z(1).
\end{equation}
 In order to compute the dimension of the torus $Z(1)$ we proceed by \textit{dévissage}.
For $i=1, \dots n$ and $j=1, \dots,r,$ consider the 1-motive
\begin{equation}\label{Mij}
M_{ij}=[u_{ij}:\ZZ \to  G_j],\qquad  u_{ij}(1)=R_{ij}  \in G_j(\CC).
\end{equation}
According to \cite[Lemma 2.2]{B19}, the 1-motive $M$ defined in \eqref{GPC-1-motive} and the 1-motive $\oplus_{j=1}^r\oplus_{i=1}^n M_{ij}$ generate the same tannakian category and so they have the same motivic Galois group. We therefore obtain the inequality
\[
\dim \Galmot (M) = \dim \Galmot (\oplus_{j=1}^r\oplus_{i=1}^n M_{ij})  \leqslant \oplus_{j=1}^r \oplus_{i=1}^n \dim \Galmot (M_{ij})
\]
and in particular
\begin{equation}\label{reduction}
	\dim \UR (M) = \dim \UR (\oplus_{j=1}^r\oplus_{i=1}^n M_{ij})  \leqslant \oplus_{j=1}^r \oplus_{i=1}^n \dim \UR(M_{ij}).
\end{equation}
We add the index $ij$ to the pure motives underlying the Lie algebra of the unipotent radical of the 1-motive $M_{ij}: B_{ij} \subseteq \cE\times \cE^*, Z'_{ij}(1) \subseteq \GG_m, Z_{ij}(1) \subseteq \GG_m.$
In the same way let ${\tR}_{ij}$ the point introduced in step (3) of the description of the Lie algebra of $\UR(M_{ij}).$

Denote by 
\begin{equation}\label{def:NoLieBracket}
	\mathrm{NoLieBracket}
\end{equation}
 the subset of $\{1,\dots,n\} \times \{1,\dots,r\}$ consisting of couples $(i,j)$ such that $\dim Z_{ij}'(1) =0.$ According to \cite[Corollary 4.5]{BP} $\mathrm{NoLieBracket}$ is
the set of couples $(i,j)$ such that one of the following conditions is satisfied:
 \begin{itemize}
 	\item  $P_i$ and $Q_j$ are both torsion,
 	\item  $P_i$ or $Q_j$ is a torsion point,
 	\item  $P_i$ and $Q_j$ are $k$-linearly dependent via an antisymmetric homomorphism, that is $\phi(P_i)=Q_j$ (or $\phi(Q_j)=P_i$) with $\phi+\overline{\phi}=0.$
 \end{itemize}
 Since it is the smallest sub-torus of $\GG_m$ which contains the values of the factor of automorphy of the $\GG_m$-torsor $I^*d^*\mathcal{B}_{ij}$, $ Z_{ij}'(1)$ is the trivial torus if and only if $I^*d^*\mathcal{B}_{ij}$ is the trivial $\GG_m$-torsor over ${B}_{ij}:$ $I^*d^*\mathcal{B}_{ij} = {B}_{ij} \times \GG_m.$ 
 
  Set 
  \begin{equation}\label{def:LieBracket}
  	\mathrm{LieBracket} = \big(\{1,\dots,n\} \times \{1,\dots,r\} \big) \smallsetminus \mathrm{NoLieBracket}.
  \end{equation}
For any couple $(i,j) \in \mathrm{LieBracket}$,  $\dim Z_{ij}'(1) =1$ and consequently  $\dim Z_{ij} (1)/ Z_{ij}'(1) =0.$

\begin{theorem}\label{Teo:dimGal(M)}
	Let $M=[u:\ZZ \rightarrow  G^n   ],
	u(1)=(R_1, \dots, R_n ) \in G^n(\CC),$ be the 1-motive \eqref{GPC-1-motive} defined by the complex numbers $q_j,p_i,t_{ij}$ \eqref{Points}. Then
	\begin{align*}
		\dim B &= \dim_k < p_i,q_j>_{i,j}\\
		\dim Z'(1) &= \dim_\QQ < \beta_{i,j}+\beta_{i,j}^t>_{  (i,j) \in \mathrm{LieBracket}}\\
		\dim Z(1)/Z'(1) &= \dim_{\QQ}
		< \textstyle{\sum_{(i,j)\in \mathrm{LieBracket}}}	\alpha_{ij}\log( s_{ij}) >
		+\dim_\QQ < t_{ij}>_{ (i,j) \in \mathrm{NoLieBracket}}
	\end{align*}

	where 
	\begin{itemize}
		\item  $< p_i,q_j>_{i,j}$ is the sub $k$--vector space of $\CC / (\Omega \otimes_\ZZ \QQ)$ generated by the classes of the complex numbers $ p_1, \dots, p_n, q_1, \dots,q_r $  modulo $\Omega \otimes_\ZZ \QQ,$
		\item $< \beta_{i,j}+\beta_{i,j}^t>_{  (i,j) \in \mathrm{LieBracket}}$ is the sub $\QQ$--vector space of $\mathrm{Hom}_\QQ(B,B^*)  := \mathrm{Hom}(B,B^*) \otimes_\ZZ \QQ $ generated by the group homomorphisms $\beta_{i,j}+\beta_{i,j}^t$ with $(i,j) \in  \mathrm{LieBracket},$
		\item $	< \textstyle{\sum_{(i,j)\in \mathrm{LieBracket}}}	\alpha_{ij}\log( s_{ij}) >$ is the sub $\QQ$--vector subspace of  $\CC/ 2 \pi \ii \QQ$ generated by the classes of the logarithms  $\sum_{i,j}\alpha_{ij}\log(s_{ij})$, with $\sum_{i,j}\alpha_{ij} x_i\otimes y_j^\vee$ running over ${Z'}^\perp$ and $(s_{ij})$ any point of $\GG_m^{nr}$ projecting onto $ \pi \big( pr_*\big(\psi(x_i, y_j^\vee)\big)_{(i,j)\in \mathrm{LieBracket}} \big),$
		\item  $< t_{ij}>_{ (i,j) \in \mathrm{NoLieBracket}}$ is the sub $\QQ$--vector space of $\CC / 2 \pi \ii \QQ$ generated by the classes of the complex numbers $t_{ij} $  modulo $2 \pi \ii \QQ$ with $(i,j) \in  \mathrm{NoLieBracket}$.
	\end{itemize} 
	In particular 
	\[	\dim \Galmot(M)	= \frac{4}{\dim_\QQ k} +2 \dim_k < p_i,q_j>_{i,j}+ \dim_\QQ < \beta_{i,j}+\beta_{i,j}^t>_{  (i,j) \in \mathrm{LieBracket}} + \]
	\[ \dim_{\QQ}
	< \textstyle{\sum_{(i,j)\in \mathrm{LieBracket}}}	\alpha_{ij}\log( s_{ij}) >+ \dim_\QQ < t_{ij}>_{ (i,j) \in \mathrm{NoLieBracket}}.\]
\end{theorem}

\begin{proof}
By \cite[Example 4.1]{BP}, $  \dim B = \dim_k < p_i,q_j>_{i,j}. $

According to \cite[Theorem 4.2]{BP} the dimension of the torus $ Z'(1)$ is equal to the dimension of the sub $\QQ$--vector space of $\mathrm{Hom}_\QQ(B,B^*) $ generated by the homomorphisms  $ \beta_{i,j}+\beta_{i,j}^t.$ But if $(i,j) \in \mathrm{NoLieBracket}$, the torus $Z'_{ij}(1)$ is trivial and so by \cite[Theorem 4.2]{BP} $\beta_{i,j}+\beta_{i,j}^t=0.$ Hence we have to consider only generators $\beta_{i,j}+\beta_{i,j}^t$ with $ (i,j) \in \mathrm{LieBracket}.$

The dimension of the torus 
$ Z(1) /Z'(1) $ is equal to the dimension of the sub $\QQ$--vector space of $\CC / 2 \pi \ii \QQ$ generated by the class of the complex number $\log \pi (pr_* {\tR})= \log \pi \big(pr_*  \big(\psi(x_i,y^\vee_j)\big)_{i,j} \big)  $  modulo $2 \pi \ii \QQ.$ 
The contributions of $\mathrm{LieBracket}$ and $\mathrm{NoLieBracket}$ to the quotient torus
$Z(1)/Z'(1)$ can be computed separately. Indeed, if
$(i,j)\in \mathrm{NoLieBracket}$, then $\beta_{ij}+\beta_{ij}^t=0,$
and hence the character $x_i\otimes y_j^\vee$ belongs to $Z'^\perp$. Therefore the corresponding
coordinate survives unchanged in the quotient by $Z'(1)$. On the other hand, the coordinates indexed
by $\mathrm{LieBracket}$ are precisely those on which the equations defining $Z'(1)$ may impose
non-trivial relations. Thus the image of $\pi(pr_*\widetilde R)$ in $Z(1)/Z'(1)$ splits into the direct
contribution of the trivial fibres and the contribution, computed modulo $Z'(1)$, of the non-trivial
fibres.
If $(i,j) \in \mathrm{NoLieBracket},$ the torus $ Z_{ij}'(1)$ is trivial, which implies that $I^*d^*\mathcal{B}_{ij} = {B}_{ij} \times \GG_m,$ and so
\[{\tR}_{ij} = pr_* {\tR}_{ij}= (P_i,Q_j, \e^{t_{ij}}).\]
Hence $\log \pi (pr_* {\tR}_{ij})= \log \pi (P_i,Q_j, \e^{t_{ij}}) =t_{ij} $ for any $ (i,j) \in \mathrm{NoLieBracket}.$
The contribution of  $\mathrm{NoLieBracket}$ is then $\dim_\QQ < t_{ij}>_{ (i,j) \in \mathrm{NoLieBracket}}$ (see also \cite[Corollary 4.8]{BP}).
According to \cite[Theorem 4.7]{BP}, the contribution of $\mathrm{LieBracket}$ is
$ \dim_{\QQ}
< \textstyle{\sum_{(i,j)\in \mathrm{LieBracket}}}	\alpha_{ij}\log( s_{ij}) >.$

The result now follows from \eqref{DimGalmot}.
\end{proof}

\begin{remark}\label{Contribution}
	The LieBracket- and NoLieBracket-contributions to
	$Z(1)/Z'(1)$ arise from distinct fibres. The former is obtained from
	the image of the trivialization point modulo the relations defining
	$Z'(1)$, whereas the latter survives unchanged in the quotient since
	the corresponding characters belong to $Z'^\perp$. Therefore these
	contributions are independent and their dimensions add.
\end{remark}

Consider the 1-motive $M=[u:\ZZ \rightarrow  G^n   ],
u(1)=(R_1, \dots, R_n ) \in G^n(\CC)$ \eqref{GPC-1-motive} defined by the complex numbers $g_2,g_3,q_j,p_i$ and $t_{ij}$ \eqref{Points}.
The field of definition of $M$ is 
\[
K:=\QQ \big(g_2,g_3,Q_j,R_{i }\big)_{ j=1, \dots, r \atop i=1, \dots,n }=\QQ \big(g_2,g_3, \wp(q_j),\wp(p_i),e^{t_{ij}}f_{q_j}(p_i) \big)_{ j=1, \dots, r \atop i=1, \dots,n },
\]
and by \cite[Proposition 2.3]{B19} and \cite[Example 5.4]{BP} its periods are
\[ \omega_{1}, \;\; \eta_{1},\;\;\omega_{2}, \;\; \eta_{2},\;\;  2 \pi \ii, \;\; p_i ,  \;\; \zeta(p_i),  \;\;q_j,  \;\;  \zeta(q_j),   \;\; t_{ij}. \]
Therefore 
\begin{equation}\label{field}
	K \big(\mathrm{periods}(M)\big) =
\end{equation}
\[ \QQ\big( g_2,g_3,  \wp(q_j),\wp(p_i), \rme^{ t_{ij}} f_{q_j}(p_i)  , \omega_{1}, \eta_{1},\omega_{2}, \eta_{2},  2 \ii \pi, p_i  ,  \zeta(p_i),  q_j,  \zeta(q_j), t_{ij}   \big)_{ j=1, \dots, r \atop i=1, \dots,n }.
\]
By the previous Theorem, we have

\begin{conjecture}[Grothendieck-André period Conjecture applied to $M$ \eqref{GPC-1-motive}]\label{GApCforM}
	For the 1-motive $M=[u:\ZZ \rightarrow  G^n   ],
	u(1)=(R_1, \dots, R_n ) \in G^n(\CC)$ \eqref{GPC-1-motive} defined by the complex numbers $q_j,p_i$ and $t_{ij}$ \eqref{Points},
	\[ \mathrm{t.d.}\,\QQ\big( g_2,g_3,  \wp(q_j),\wp(p_i), \rme^{ t_{ij}} f_{q_j}(p_i)  , \omega_{1}, \eta_{1},\omega_{2}, \eta_{2},  2 \ii \pi, p_i  ,  \zeta(p_i),  q_j,  \zeta(q_j), t_{ij}   \big)_{ j=1, \dots, r \atop i=1, \dots,n }\geq\]
\[	 \frac{4}{\dim_\QQ k} +2 \dim_k < p_i,q_j>_{i,j}+ \dim_\QQ <  \beta_{i,j}+\beta_{i,j}^t>_{  (i,j) \in \mathrm{LieBracket}} + \]
\[\dim_{\QQ}
< \textstyle{\sum_{(i,j)\in \mathrm{LieBracket}}}	\alpha_{ij}\log( s_{ij}) >+ \dim_\QQ < t_{ij}>_{ (i,j) \in \mathrm{NoLieBracket}}\]	
If the $1$-motive $M$ is defined over $\overline{\QQ},$ then equality holds.
\end{conjecture}

We furnish now two examples of this conjecture.
For an elliptic curve $\cE$ defined over $\CC,$ we have 
		$\dim \Galmot (\cE) = \frac{4}{\dim_\QQ k}$
		(see for example \cite[Formula (5.51)]{B02}). Hence the Grothendieck-André period Conjecture applied to $\cE$ reads
		\begin{equation*}
			\mathrm{t.d.}\, \QQ (g_2,g_3,\omega_1, \omega_2,\eta_1,\eta_2) \geqslant \frac{4}{\dim_\QQ k} =
			\begin{cases}
				2, \mathrm{if}\; \cE \;\mathrm{has\; complex\; multiplication}\\
				4, \mathrm{otherwise}.
			\end{cases} 
		\end{equation*}
		If $g_2$ and $g_3$ are algebraic, the above inequality is conjecturally an equality, and if moreover we assume complex multiplication, this equality is Chudnovsky Theorem: 
		\begin{equation}\label{Chudnovsky}
			\mathrm{t.d.}\, \QQ (\omega_1, \omega_2,\eta_1,\eta_2) =2.
		\end{equation} 
	
	 Consider the 1-motive  \eqref{GPC-1-motive} $	M=[u:\ZZ \rightarrow  G^n   ],
		u(1)=(R_1, \dots, R_n ) \in G^n(\CC),$ with the complex numbers $p_1, \dots, p_n,q_1, \dots, q_r$ \eqref{Points} satisfying the condition $\dim_k <p_i,q_j >_{i,j}=n+r.$ Then the dimension of the torus $Z'(1)$ is maximal by \cite[Corollary 4.6]{BP}, that is $ \dim Z(1)= \dim Z'(1)=nr.$ Hence
	\[
			\dim \Galmot (M)= \frac{4}{\dim_\QQ k} + 2 (n+r) + nr.
		\]
		In other words, \textit{if the dimension of $B$ is maximal}, that is $\dim_k <p_i,q_j >_{i,j}=n+r,$ \textit{the toric part $Z(1)$ of $\UR(M)$ is filled completely by the Lie bracket $[\cdot,\cdot]$} and moreover \textit{it is maximal}. Hence \textit{the dimension of $\UR (M)$ is governed only by the abelian part of $M$ and it is maximal.}
		 In this case the Grothendieck-André period Conjecture applied to this 1-motive reads
			\[ \mathrm{t.d.}\,\QQ\big(  g_2,g_3,  \wp(q_j),\wp(p_i), \rme^{ t_{ij}} f_{q_j}(p_i)  , \omega_{1}, \eta_{1},\omega_{2}, \eta_{2},  2 \ii \pi, p_i  ,  \zeta(p_i),  q_j,  \zeta(q_j), t_{ij}  \big)_{ j=1, \dots, r \atop i=1, \dots,n }\geqslant\]
		\[	 \frac{4}{\dim_\QQ k} +2 (n+r)+nr \]	
		with the lower bound completely independent of the points $t_{ij}.$ 

\medspace
We finish this section proving the Gro\-then\-dieck-Andr\'{e} period Conjecture for 1-motives defined by an elliptic curve with algebraic invariants and complex multiplication and by torsion points.

\begin{proposition}\label{ProofGP-torsion} 
	Let $M=[u:\ZZ \rightarrow  G^n], u(1)=(R_1, \dots, R_n)$ be the 1-motive \eqref{GPC-1-motive} defined by the torsion points 
	\begin{equation}\label{qpt}
		q_j := \frac{\omega_j'}{l_j} \qquad \qquad  p_i := \frac{\omega_i'}{m_i} \qquad \qquad  t_{ij}:=	\frac{2 k_{ij} \pi \ii }{m_i}  
	\end{equation}
	with $k_{ij},m_i, l_j\in\ZZ$, $m_i,l_j>0$, $0\leqslant k_{ij}\leqslant m_i-1$, $\omega_i', \omega_j'\in\Omega$, $\frac{\omega_i'}{m_i},\frac{\omega_j'}{l_j} \not\in\Omega$ for $j=1,\dots,r$ and $i=1, \dots,n.$ Then the Grothendieck-Andr\'{e} period Conjecture applied to $M$ reads
	\[	 \mathrm{t.d.}\, K \big(\mathrm{periods}(M)\big)= \mathrm{t.d.}\, \QQ \big(g_2,g_3, \omega_1, \omega_2,\eta_1, \eta_2\big)
	\geq \frac{4}{\dim_\QQ k} .\]
	\par\noindent Therefore, if the elliptic curve $\cE$ underlying $M$ has algebraic invariants and complex multiplication, 
	the Grothendieck-Andr\'{e} period Conjecture applied to $M$ is true.
\end{proposition}

\begin{proof}
For $j=1, \dots,r,$ we set $\alpha_j := \zeta(\frac{\omega_j'}{l_j})-\frac{\eta(\omega_j')}{l_j}.$
According to \cite[Lemma 3.1 (3)]{BW} the complex numbers $\alpha_j$ are algebraic over $\QQ(g_2,g_3).$
Consider the 1-motive
\[
M^T=[u^T:\ZZ \to  \GG_m^{nr}],\quad  u^T(1)= \Big(\rme^{\alpha_j \frac{\omega_i'}{m_i}} \Big)_{i,j} \in   \GG_m^{nr} (\CC).
\]
We proceed in three steps:

(1) First we prove that $
	 \mathrm{t.d.}\, K_{ M \oplus M^T} \big(\mathrm{periods}(M \oplus M^T)\big) = \mathrm{t.d.} \, \QQ(g_2,g_3) \big(\mathrm{periods}(\cE), \rme^{
		\alpha_j \frac{\omega_i'}{m_i}} \big)_{i,j},$ where $K_{ M \oplus M^T}$ is the field of definition of $ M \oplus M^T.$
Explicitly we have to show that the two fields
	\[
		K_{ M \oplus M^T} \big(\mathrm{periods}(M \oplus M^T)\big)=  \QQ\Big( g_2,g_3, \wp\Big(\frac{\omega_j'}{l_j}\Big),\wp\Big(\frac{\omega_i'}{m_i}\Big), \rme^{ \frac{2 k_{ij} \pi \ii}{m_i}} f_{\frac{\omega_j'}{l_j}}\Big(\frac{\omega_i'}{m_i}\Big)  , \rme^{
		\alpha_j \frac{\omega_i'}{m_i}}  , \]
		\[	 \omega_{1}, \eta_{1}, \omega_{2}, \eta_{2}, \frac{\omega_i'}{m_i} ,  \zeta\Big(\frac{\omega_i'}{m_i}\Big),  \frac{\omega_j'}{l_j},  \zeta\Big(\frac{\omega_j'}{l_j}\Big), \frac{2 k_{ij} \pi \ii}{m_i}  , \alpha_j \frac{\omega_i'}{m_i} \Big)_{i,j}
	\]
	and
	\[ \QQ\big( g_2,g_3, \omega_{1}, \eta_{1},\omega_{2}, \eta_{2}  , \rme^{
		\alpha_j \frac{\omega_i'}{m_i}}  \big)
	\]
	have the same transcendence degree over $\QQ$. But this is true because 
	for any $i,j$
	\begin{itemize}
		\item the numbers $ \wp\big(\frac{\omega_j'}{l_j}\big),\wp\big(\frac{\omega_i'}{m_i}\big)$ are algebraic over $\QQ(g_2,g_3)$ by 
		Table 0 of \S \ref{ValuesAtTorsionPoints};
		\item the numbers $ \frac{2 k_{ij} \pi \ii}{m_i}$ belong to $\QQ( \omega_1, \omega_2,\eta_1, \eta_2)$ because of Legendre relation \eqref{Equation:Legendre};
		\item the numbers $f_{\frac{\omega_j'}{l_j}}\Big(\frac{\omega_i'}{m_i}\Big)   \rme^{
			\alpha_j \frac{\omega_i'}{m_i}} $ belong to $\overline{\QQ(g_2,g_3)}$ by Corollary \ref{Cor:f_q(p)pqTorsion}\footnote{
			If the points $q_j$ parametrizing the extension $G$ are all 2-torsion points (that is $ q_j = \frac{\omega_j'}{2}$ for $j=1 \dots,r$), the numbers $\alpha_j= \zeta\Big(\frac{\omega_j'}{2}\Big)-\frac{\eta(\omega_j')}{2}$ are all zero. This implies that $	u^T(1)= \big(1, \dots,1 \big) \in   \GG_m^{nr} (\CC)$
			and the 1-motive $M^T$ plays no role in the proof of Proposition \ref{ProofGP-torsion}, which becomes much easier.
			In fact, if $l_j=2$ for any $j,$ the numbers
			$f_{\omega_j'/l_j}\Big(\frac{\omega_i'}{m_i}\Big)$ belong to $\overline{\QQ(g_2,g_3)}$ by Corollary \ref{Cor:f_q(p)qpTorsion} or by Corollary \ref{Cor:f_q(p)pqTorsion} and so we don't have to multiply them by $ \rme^{
				\alpha_j \omega_i'/m_i}$  in order to get algebraic numbers.
				};
		\item the numbers $\frac{\omega_i'}{m_i}$ and $\frac{\omega_j'}{l_j}$ belong to $\QQ( \omega_1, \omega_2);$
		\item the numbers $\zeta \big(\frac{\omega_i'}{m_i}\big) $ and $\zeta\big(\frac{\omega_j'}{l_j} \big)$ belong to $\overline{\QQ(g_2,g_3)}(\eta_1, \eta_2)$ according to Table 0 of \S \ref{ValuesAtTorsionPoints};
		\item the numbers $\alpha_j \frac{\omega_i'}{m_i}$ belong to $\overline{\QQ(g_2,g_3)}( \omega_1, \omega_2)$
		by \cite[Lemma 3.1 (3)]{BW}.
	\end{itemize}

(2) Now we show that 
\begin{align}
\nonumber	 \dim \Galmot (M) &= \dim \Galmot (\cE) =\frac{4}{\dim_\QQ k} \\
\nonumber \dim \Galmot (M \oplus M^T) &= \frac{4}{\dim_\QQ k} + \dim_\QQ <  \alpha_j \frac{\omega_i'}{m_i}>_{i,j}
\end{align}
	where $<  \alpha_j \frac{\omega_i'}{m_i}>_{i,j}$ is the sub $\QQ$--vector space of $\CC / 2 \pi \ii \QQ$ generated by the classes of the complex numbers $  \alpha_j \frac{\omega_i'}{m_i}$  modulo $2 \pi \ii \QQ.$ We get the first statement using \cite[ \S 6 0.2]{BP} and the inequality \eqref{reduction}. For the second statement, add the index $T$  to the pure motives underlying the unipotent radical of the 1-motive $M^T.$ 
	Since $M^T$ has no abelian part, according to Theorem \ref{Teo:dimGal(M)} 
	\[		\dim B_T = \dim Z_T'(1) =0 \quad  \mathrm{and}  \quad \dim Z_T(1) / Z_T'(1) = \dim_\QQ <   \alpha_j  \frac{\omega_i'}{m_i}>_{ i,j}.
	\]
	On the other hand, we have just computed that $\dim \UR(M)=0,$ and so
	\[	\dim \UR (M \oplus M^T) =  \dim \UR ( M^T) = \dim_\QQ <   \alpha_j  \frac{\omega_i'}{m_i}>_{ i,j}.
	\]

	(3) Observe that 
	\[ K \Big(\mathrm{periods}(M), \rme^{
		\alpha_j \frac{\omega_i'}{m_i}} ,\alpha_j \frac{\omega_i'}{m_i}\Big)_{i,j} = K_{ M \oplus M^T} \big(\mathrm{periods}(M \oplus M^T)\big).\] 
	Let $N$ be the dimension of the sub $\QQ$--vector space $<  \alpha_j \frac{\omega_i'}{m_i}>_{i,j}$ of $\CC / 2 \pi \ii \QQ$ generated by the classes of the complex numbers $  \alpha_j \frac{\omega_i'}{m_i}$  modulo $2 \pi \ii \QQ$ and let $\beta_1, \dots, \beta_N$ be its basis. Recalling that the numbers $\alpha_j $ are algebraic over $\QQ(g_2,g_3)$
	(see \cite[Lemma 3.1 (3)]{BW}), by \cite[Lemma 4.4]{BW} we have that the equalities
	\[	\mathrm{t.d.}\, K \big(\mathrm{periods}(M), \rme^{
		\alpha_j \frac{\omega_i'}{m_i}}, \alpha_j \frac{\omega_i'}{m_i}\big)_{i,j}= \mathrm{t.d.}\, K \big(\mathrm{periods}(M), \rme^{\beta_s}\big)_{s=1, \dots,N}\]
	and
	\[ \mathrm{t.d.} \QQ(g_2,g_3) \big(\mathrm{periods}(\cE), \rme^{
		\alpha_j \frac{\omega_i'}{m_i}}, \alpha_j \frac{\omega_i'}{m_i}\big)_{i,j} = \mathrm{t.d.} \QQ(g_2,g_3) \big(\mathrm{periods}(\cE), \rme^{\beta_s} \big)_{s=1, \dots,N}.\]
	Now the Grothendieck-Andr\'{e} period Conjecture applied to $M \oplus M^T,$ that we can made explicit thanks to step (1) and (2), furnishes 
	\[ \mathrm{t.d.}\, K \big(\mathrm{periods}(M), \rme^{\beta_s}\big)_{s=1, \dots,N} =
	 \mathrm{t.d.}\, K_{ M \oplus M^T} \big(\mathrm{periods}(M \oplus M^T)\big)=\]
	\[ \mathrm{t.d.} \QQ(g_2,g_3) \big(\mathrm{periods}(\cE),  \rme^{\beta_s}\big)_{s=1, \dots,N}\geqslant \]
	\[\frac{4}{\dim_\QQ k} + N. \]	
	Removing the $N$ numbers $\rme^{\beta_1},\dots, \rme^{\beta_N},$ we obtain 
	\[ \mathrm{t.d.}\, K \big(\mathrm{periods}(M)\big)=  \mathrm{t.d.} \QQ(g_2,g_3) \big(\mathrm{periods}(\cE) \big) \geqslant \frac{4}{\dim_\QQ k}.\]
	The last statement is just Chudnovsky Theorem.
\end{proof}

Here the sketch of another proof of the above Proposition involving the tannakian language:
	The 1-motive $M=[u:\ZZ \to  G^n] $ defined by the torsion points \eqref{qpt} is isogeneous to the 1-motive $M_0=[0:\ZZ \to  \cE^n \times \GG_m^{nr}] .$ Since tannakian categories are defined modulo isogenies, the tannakian categories generated by $M$ and by $M_0$ are isomorphic. Therefore the conjectures obtained applying the Grothendieck-Andr\'{e} period Conjecture to $M$ and to $M_0$ respectively are the same, that is
	the Grothendieck-Andr\'{e} period Conjecture applied to the 1-motive $M$ defined by the torsion points \eqref{qpt} is equivalent to the Gro\-then\-dieck-Andr\'{e} period Conjecture applied to the elliptic curve $\cE$ underlying $M.$


\section{Equivalence between the semi-elliptic Conjecture and the Gro\-then\-dieck-Andr\'{e} period Conjecture applied to 1-motives with elliptic part}\label{Section:GPCconsequences}


By \cite[Corollary 3.5, Lemma 4.4]{BW}, for a 1-motive of the form $M=[u:\ZZ \rightarrow \GG_m^s \times \cE^n ], u(1) =(\rme^{t_1}, \dots, \rme^{t_s} , P_1, \dots, P_n ) ,$ the transcendence degree of the field $K(\mathrm{periods}(M)) $
does not depend on 
the choice of the bases used in order to compute the periods of $M,$ that is it does not depend on
\begin{itemize}
	\item the chosen $\QQ$-basis of the sub $\QQ$--vector space of $\CC / 2 \pi  \ii \QQ  $ generated by the classes of $t_1, \dots, t_s$  modulo $ 2 \pi  \ii \QQ,$ and
	\item the chosen $k$-basis of the sub $k$--vector space of $\CC / (\Omega \otimes_\ZZ \QQ)  $ generated by the classes of $p_1, \dots, p_n$ modulo $\Omega \otimes_\ZZ \QQ.$
\end{itemize}
Hence also the Gro\-then\-dieck-Andr\'{e} period Conjecture applied to this 1-motive  $M=[u:\ZZ \rightarrow \GG_m^s \times \cE^n ]$ does not depend on the chosen bases.

For a general 1-motive $
M=[u:\ZZ \rightarrow G^n ], u(1) =(R_1, \dots, R_n)$ \eqref{GPC-1-motive} defined by the complex numbers $p_i, q_j$ and $t_{ij}$ \eqref{Points}, this invariance property no longer holds at the level of explicit period computations. In fact, Serre functions perform very poorly with respect to change of basis, since almost all base change formulae involve an exponential term (for example, if we replace $p_i$ with $p_i+ \omega,$ where $\omega$ is a period, by \eqref{f(z+omega)} $f_{q}(p+\omega)= f_{q}(p) \rme^{\eta q -\omega\zeta(q)}.$ See also the addition formulae, the multiplication by an integer formulae and the multiplication by $\tau$ formula of \S 3)\footnote{Remark however that this exponential term is the value of the exponential map at a complex number which belong to the field $K(\mathrm{periods}(M))$.}. Nevertheless, the tannakian category generated by $M$ is independent of these choices and so also the Gro\-then\-dieck-Andr\'{e} period Conjecture applied to this 1-motive $
M=[u:\ZZ \rightarrow G^n ]$ does not depend on them.

\medspace

Consider the 1-motive $
M=[u:\ZZ \rightarrow G^n ], u(1) =(R_1, \dots, R_n)$ \eqref{GPC-1-motive} defined by the points $p_i, q_j$ and $t_{ij}$ \eqref{Points}.

\begin{notation}\label{Notation} Throughout this section, we use the following notation.

		(1) Let 
		$$p'_1, \dots ,p'_{n'}, q'_1,\dots ,q'_{r'} $$
		 be a $k$--basis of the sub $k $--vector space $<p_i,q_j>_{i,j}$ of  $\CC/(\Omega\otimes_\ZZ\QQ)$ generated by the classes of $p_1, \dots, p_n,q_1, \dots, q_r .$ For ease of notation, we assume without loss of generality that  $p_i'=p_i$ for $1\leqslant i\leqslant n'$,  $q_j'=q_j$ for $1\leqslant j\leqslant r'.$ Remark that for $i \leqslant n'$ and $j \leqslant r'$,  $(i,j) \in \mathrm{LieBracket}$ \eqref{def:LieBracket} by \cite[Corollary 4.5]{BP}.
		
		(2) Let 
		$$\gamma_1, \dots ,\gamma_{u}$$
		 be a $\QQ$--basis of the sub $\QQ $--vector space $ < \beta_{i,j}+\beta_{i,j}^t>_{  (i,j) \in \mathrm{LieBracket}}$ of $\mathrm{Hom}_\QQ(B,B^*) $ generated by the homomorphisms $\beta_{i,j}+\beta_{i,j}^t$ with $(i,j) \in \mathrm{LieBracket}.$  Without loss of generality, for $m=1, \dots,u,$ we assume $\gamma_m =\beta_{i_m,j_m}+\beta_{i_m,j_m}^t$ for some $(i_m,j_m) \in \mathrm{LieBracket}.$ 
		
		In Lemma \ref{u} we will show that $u \geqslant n' r'$ and so 
	 for $m=1, \dots,n'r',$ we may assume $ i_m =1, \dots, n'$ and  $ j_m =1, \dots, r'.$

		(3) Let
		\[
		\theta_1,\ldots,\theta_v
		\]
		be a $\QQ$--basis of the sub $\QQ$--vector space of
		$\CC/2\pi \ii\QQ$ generated by the classes of \\
		$\sum_{(i,j)\in \mathrm{LieBracket}}
		\alpha_{ij}\log(s_{ij}),$
		with $\sum_{(i,j)\in \mathrm{LieBracket}}
		\alpha_{ij}x_i\otimes y_j^\vee
		\in Z'^\perp $
		and with $(s_{ij})$ any point of $\GG_m^{nr}$ projecting onto $
		\pi\big(pr_*\big(\psi(x_i,y_j^\vee)\big)_{(i,j)\in \mathrm{LieBracket}}\big).$
		We choose this basis among the classes of the logarithms
		of the coordinates of the trivialization point
		$ \pi(pr_*\widetilde R).$
		Thus, after reindexing if necessary, for $a=1, \dots,v,$ we assume
		\[
		\theta_a=t_{i_aj_a}
		\]
		for suitable pairs $(i_a,j_a)\in \mathrm{LieBracket}.$
		The pairs $(i_a,j_a)$ are chosen so that they do
		not contribute new dimensions to $Z'(1)$ and the classes of
		$t_{i_aj_a}$ form a basis of the $\mathrm{LieBracket}$-contribution
		to $Z(1)/Z'(1)$.
		
		(4) Let 
		$$t_1, \dots ,t_{s}$$
		 be a $\QQ$--basis of the sub $\QQ $--vector space $<t_{ij}>_{(i,j) \in \mathrm{NoLieBracket}}$ of $\CC / 2\pi\ii \QQ$ generated by the classes of $t_{ij} $ with $(i,j) \in \mathrm{NoLieBracket}$ \eqref{def:NoLieBracket}.  Without loss of generality, for $l=1, \dots,s,$ we assume $t_l= t_{i_lj_l}$ for some $(i_l,j_l) \in \mathrm{NoLieBracket}.$
		
	\end{notation}

\begin{proposition}\label{prop:IndipChoiceBase}
	Let  $
	M=[u:\ZZ \rightarrow G^n ], u(1) =(R_1, \dots, R_n ), $ be the 1-motive \eqref{GPC-1-motive} defined by the complex numbers $p_i, q_j$ and $t_{ij}$ \eqref{Points}.  Denote by $M^C$ the 1-motive \eqref{GPC-1-motive} defined by the complex numbers
	$p_{i_m},q_{j_m},t_{i_mj_m}$ for $m=1,\ldots,u$,
	by the complex numbers $p_{i_a},q_{j_a},t_{i_aj_a}$ for $a=1,\ldots,v$,
	and by the complex numbers $p_{i_l},q_{j_l},t_{i_lj_l}$ for $l=1,\ldots,s.$ Then the 1-motives $M$ and $M^C$ generate the same tannakian category. 
	 
	 In particular the conjectures obtained applying the Grothendieck-Andr\'{e} period Conjecture to $M$ and to $M^C$ respectively are equivalent.
\end{proposition}

\begin{remark}
	The 1-motive $M^C$ is the smallest quotient of the 1-motive $M$, whose motivic Galois group coincides with the one of $M$, that is $\Galmot(M^C)=\Galmot(M).$
\end{remark}

\begin{proof}
  For $i=1, \dots n$ and $j=1, \dots,r$ consider the 1-motive introduced in \eqref{Mij}
$M_{ij}=[u_i:\ZZ \to  G_{j}], u(1)=R_{ij} =\exp_{G_{j}}(p_i,t_{ij}) \in G_{j}(\CC).$
By \cite[Lemma 2.2]{B19}, the 1-motive $M$ defined in \eqref{GPC-1-motive} and the 1-motive $\oplus_{j=1}^r\oplus_{i=1}^n M_{ij}$ generate the same tannakian category.
For $m=1, \dots,u,$ let $M_{i_mj_m}$ the 1-motive defined by the points $p_{i_m}, q_{j_m}$ and $t_{i_mj_m},$ 
for $a=1, \dots,v,$ let $M_{i_aj_a}$ the 1-motive defined by the points $p_{i_a}, q_{j_a}$ and $t_{i_aj_a},$ 
and finally for $l=1, \dots,s,$ let $M_{i_lj_l}$ the 1-motive defined by the points $p_{i_l}, q_{j_l}$ and $t_{i_lj_l}.$ 
By \cite[Lemma 2.2]{B19}, the 1-motive $M^C$ and the 1-motive 
$(\oplus_{m=1}^{u} M_{i_mj_m})\oplus  (\oplus_{a=1}^v M_{i_aj_a}) \oplus  (\oplus_{l=1}^s M_{i_lj_l})$
generate the same tannakian category.
Since by construction the 1-motives $M_{i_mj_m}, M_{i_aj_a}$ and $M_{i_lj_l}$ are some of the $M_{ij},$ the 1-motive 
 $ (\oplus_{m=1}^{u} M_{i_mj_m}) \oplus  (\oplus_{a=1}^v M_{i_aj_a}) \oplus  (\oplus_{l=1}^s M_{i_lj_l})$
is a quotient of the 1-motive $\oplus_{j=1}^r\oplus_{i=1}^n M_{ij},$ and so also $M^C$ is a quotient of the 1-motive $M$.
 We have then an inclusion of motivic Galois groups. 
\[ \Galmot(M^C)\hookrightarrow  \Galmot(M).\]
In order to conclude that $ M$ and $ M^C$ generate the same tannakian category, we show that the two motivic Galois groups $ \Galmot(M^C) $ and $ \Galmot(M)$ have same dimension.

Since $ \Galmot(M^C) $ and $ \Galmot(M)$ have the same reductive parts, we reduce to prove that the corresponding unipotent radicals $\UR(M^C)$ and $\UR(M)$ have same dimension. Set
\[ M^{LB-Z'} =\oplus_{m=1}^{u} M_{i_mj_m},  \qquad  M^{LB-Z/Z'} =\oplus_{a=1}^{v} M_{i_aj_a}, \qquad M^{NLB} =\oplus_{l=1}^s M_{i_lj_l}\]
so that 
\[\dim \big(\UR(M^C)\big) =\dim \big(\UR( M^{LB-Z'} \oplus  M^{LB-Z/Z'} \oplus  M^{NLB})\big) .\]
We add the index $C$ (\textit{resp}. $LB-Z'$, $LB-Z/Z'$ and $NLB$) to the pure motives underlying the Lie algebra of the unipotent radical of the 1-motive $M^C$ (\textit{resp}. $ M^{LB-Z'}$, $ M^{LB-Z/Z'}$ and $ M^{NLB}$).
Clearly
\begin{equation}\label{eq:dimAb}
\dim B = \dim_k < p_i,q_j>_{i,j}  = n'+r'= \dim_k <p_{i_m},q_{j_m},
p_{i_a},q_{j_a},p_{i_l},q_{j_l}>_{m,a,l} = \dim B_C.
\end{equation} 
Since $\gamma_m=\beta_{i_mj_m}+\beta_{i_mj_m}^t$ for $m=1, \dots,u$ 
form a basis of
$ <\beta_{ij}+\beta_{ij}^t >_{(i,j)\in \mathrm{LieBracket}},$
each pair $(i_m,j_m)$ contributes an independent dimension to
$Z'(1)$. In particular, 
$\dim Z_{i_mj_m}(1)/Z'_{i_mj_m}(1)=0,$ and so the pairs $(i_m,j_m)$  do
not contribute dimensions to $Z(1)/Z'(1)$.
Recall that the pairs $(i_a,j_a)$ are chosen in such a way that they do
not contribute new dimensions to $Z'(1)$, whereas the classes of
$t_{i_aj_a}$ form a $\QQ$--basis of the $\mathrm{LieBracket}$-contribution
to the quotient torus $Z(1)/Z'(1)$. Clearly, the pairs $(i_l,j_l)$ do not contribute
 any dimension to $Z'(1),$ whereas the  classes of
$t_{i_lj_l}$ form a $\QQ$--basis of the $\mathrm{NoLieBracket}$-contribution
to the quotient torus $Z(1)/Z'(1).$ Recalling Remark \ref{Contribution}, we have 
\begin{align}\label{eq:dimTorus}
	\dim Z_C(1)&= \dim Z_{LB-Z'}'(1) + \dim Z_{LB-Z/Z'}(1)/Z'_{LB-Z/Z'}(1) + \dim Z_{NLB}(1)/Z'_{NLB}(1) \\
\nonumber	&= \dim_\QQ < \beta_{i,j}+\beta_{i,j}^t>_{  (i,j) \in \mathrm{LieBracket}} + \dim_\QQ < \sum_{(i,j)\in \mathrm{LieBracket}}
\alpha_{ij}\log(s_{ij})> \\
\nonumber & +\dim_\QQ < t_{ij}>_{ (i,j) \in \mathrm{NoLieBracket}}  \\
\nonumber	&=  \dim Z'(1) + \dim Z(1)/Z'(1) = \dim Z(1).
\end{align}
The equalities  \eqref{eq:dimAb} and \eqref{eq:dimTorus} show that the pure motives underlying the Lie algebras of
$\UR(M^C)$ and $\UR(M)$ have the same dimensions.
Hence $\dim \UR(M^C) = \dim \UR(M).$

\end{proof}

\begin{lemma}\label{u}
	$u \geqslant n' r'.$
	
	In particular, for $m=1, \dots,n'r',$ we may assume $ i_m=1, \dots, n'$ and  $ j_m=1, \dots, r'.$ 
\end{lemma}
\begin{proof}
 Set
\[ M^{Ab} =  \oplus_{i=1}^{n'} \oplus_{j=1}^{r'} M_{ij}.\]	
In particular 
$ M^{Ab} =  M^{LB-Z'} / (\oplus_{i_m  > n' \atop j_m > r' } M_{i_mj_m}).$
For $i \leqslant n'$ and $j \leqslant r'$,  $(i,j) \in \mathrm{LieBracket}$ and so $M^{Ab}$ is a quotient of $M^{LB-Z'},$ and we have an inclusion of motivic Galois groups 
$ \Galmot(M^{Ab})\hookrightarrow  \Galmot(M^{LB-Z'})$ which induces the inequality
\begin{equation}\label{formule-u}
	 \dim Z'_{Ab}(1) \leqslant \dim Z'_{LB-Z'}(1).
\end{equation}
Since $p'_1, \dots ,p'_{n'}, q'_1,\dots ,q'_{r'} $ are $k$--linearly independent, \cite[Corollary 4.6]{BP} implies
$ \dim Z_{Ab}(1)=\dim Z_{Ab}'(1) =n'r'$ and $ \dim Z_{Ab}(1)/Z'_{Ab}(1)=0.$
By \eqref{eq:dimTorus}, the inequality \eqref{formule-u} reads $n'r' \leqslant u.$
\end{proof}

\begin{example}
	Consider the 1-motive  \eqref{Mij} $M_{ij} =[u_{ij}:\ZZ \to G_{j}], u_{ij}(1) = \exp_{G_j} (p_i,t_{ij}) \in G_j(\CC) ,$ where $\phi(P_i)=Q_j$ (or $\phi(Q_j)=P_i$) with $\phi$ a non antisymmetric homomorphism. Then $\dim B_{ij}= 1, \dim Z'_{ij}(1)=1$ and $\dim Z_{ij}(1) / Z'_{ij}(1)=0$. In particular, for this 1-motive $n'=1, r'=0, u=1$, and $M=M^{LB-Z'}.$
\end{example}

\begin{corollary}\label{GAPconjectureClean}
	 The Gro\-then\-dieck-Andr\'{e} period Conjecture applied to the 1-motive \eqref{GPC-1-motive} $
	M=[u:\ZZ \rightarrow G^n ], u(1) =(R_1, \dots, R_n ) ,$   defined by the complex numbers $p_i, q_j$ and $t_{ij}$ \eqref{Points} reads 
		\[ \mathrm{t.d.}\,\QQ\big(  g_2,g_3,  \wp(q_j),\wp(p_i), \rme^{ t_{i_mj_m}} f_{q_{j_m}}(p_{i_m}) ,  \rme^{ t_{i_aj_a}} f_{q_{j_a}}(p_{i_a}) ,\rme^{ t_{i_lj_l}} , \]
		\[ \omega_{1}, \eta_{1},\omega_{2}, \eta_{2},  p_i  ,  \zeta(p_i),  q_j,  \zeta(q_j), t_{i_mj_m}  , t_{i_aj_a},t_{i_lj_l}  \big)_{ \substack{j=1, \dots, r' \\  i=1, \dots,n' \\ m=1, \dots, u \\ a=1, \dots,v \\ l=1, \dots,s }}\]
	\[\geqslant	 \frac{4}{\dim_\QQ k} +2 (n'+r')+ u +v+s.\]	
	In particular, this conjecture is independent 
	of the choice of the bases used in order to compute the periods of $M.$
\end{corollary}

\begin{proof} According to Proposition \ref{prop:IndipChoiceBase} we can work with the 1-motive $M^C$ instead of $M$. Denote by $K_{M^C}$ the field of definition of $M^C$. We have that
	\begin{equation}\label{K(periods(M^C))}
		K_{M^C} \big(\mathrm{periods}(M^C)\big) =
	\end{equation}
	\[ \QQ\big(  g_2,g_3,  \wp(q_{j_m}),\wp(p_{i_m}), \rme^{ t_{i_mj_m}} f_{q_{j_m}}(p_{i_m})  ,
	\wp(q_{j_a}),\wp(p_{i_a}), \rme^{ t_{i_aj_a}} f_{q_{j_a}}(p_{i_a}),
	\wp(q_{j_l}),\wp(p_{i_l}), \rme^{ t_{i_lj_l}} f_{q_{j_l}}(p_{i_l}),
 \]
	\[	\omega_{1}, \eta_{1},\omega_{2}, \eta_{2},   p_{i_m}  ,  \zeta(p_{i_m}), q_{j_m},  \zeta(q_{j_m}),
	 p_{i_a}  ,  \zeta(p_{i_a}), q_{j_a},  \zeta(q_{j_a}),\]
\[	p_{i_l}  ,  \zeta(p_{i_l}),  q_{j_l},  \zeta(q_{j_l}),
	t_{i_mj_m}, t_{i_aj_a} , t_{i_lj_l}  \big)_{ \substack{m=1, \dots, u\\ a=1, \dots, v \\ l=1, \dots,s}}.\]	
By definition any index $(i_l,j_l ) \in \mathrm{NoLieBracket}$ satisfies one of the following three conditions: $P_{i_l}$ and $Q_{j_l}$ are both torsion, $P_{i_l}$ or $Q_{j_l}$ is a torsion point, and $P_{i_l}$ and $Q_{j_l}$ are $k$-linearly dependent via an antisymmetric homomorphism. Since all our results involving Serre function in \S 2 and \S 4 are dual, without loss of generality

- for $l=1, \dots,s'$ with $s' \leqslant s,$ we assume $q_{j_l}=\frac{\omega_{j_l}}{a_{j_l}},$  with  $\omega_{j_l}$ a period and $a_{j_l}$ an integer, $a_{j_l} \geqslant 2,$ such that $\frac{\omega_{j_l}}{a_{j_l}} \notin \Omega$. We do not fix the point $p_{i_l}$ since we will treat the first two cases ($P_{i_l}$ and $Q_{j_l}$ are both torsion, $P_{i_l}$ or $Q_{j_l}$ is a torsion point) together; 

-for $l=s'+1, \dots,s,$ let $p_{i_l}=\phi_{i_lj_l} q_{j_l},$ with $\phi_{i_lj_l}$ antisymmetric homomorphism of $\cE.$
Since $\phi_{i_lj_l}$ is antisymmetric, it exists a period  $\omega_{j_l}$ and an integer $a_{j_l}$, $a_{j_l} \geqslant 2,$  such that $\frac{\omega_{j_l}}{a_{j_l}} \notin \Omega$ and $(\phi_{i_lj_l} + \overline{\phi_{i_lj_l}})(q_{j_l}) = \frac{\omega_{j_l}}{a_{j_l}} .$

For $l=1, \dots,s',$ set $ \alpha_{j_l} := \zeta (\frac{\omega_{j_l}}{a_{j_l}})-\frac{\eta(\omega_{j_l})}{a_{j_l}}$ and $\beta_{j_l} :=0 .$ For
 $l=s'+1, \dots,s,$ set $ \alpha_{j_l} :=0$ and
$ \beta_{j_l} := \frac{1}{a_{j_l}} (\omega_{j_l} \zeta(q_{j_l})- \eta(\omega_{j_l}) q_{j_l}) .$
By \cite[Lemma 3.1 (3)]{BW} $\alpha_{j_l}$ is algebraic over $\QQ(g_2,g_3)$ and 
by \cite[Corollary 3.5]{BW} $\beta_{j_l}$ is algebraic over the field $\QQ( g_2,g_3,\omega_1,\omega_2, \eta_1,\eta_2, p_i, \wp(p_i), \zeta(p_i), q_j, \wp(q_j), \zeta(q_j))_{\substack{j=1, \dots, r' \\ i=1, \dots, n'}}.$ 
Consider the 1-motive
\[
M^T=[u^T:\ZZ \to  \GG_m^{s}],\quad  u^T(1)= \Big( \rme^{t_{i_lj_l}}, \rme^{\alpha_{j_l} p_{i_l}} , \rme^{\beta_{j_l} } \Big)_{l=1, \dots,s} \in   \GG_m^{3s} (\CC).
\]

Set
\[ F:=\]
\[ \QQ\big(  g_2,g_3,  \wp(q_j),\wp(p_i), \rme^{ t_{i_mj_m}} f_{q_{j_m}}(p_{i_m}),\rme^{ t_{i_aj_a}} f_{q_{j_a}}(p_{i_a}), \rme^{ t_{i_lj_l}}  ,\]
\[ \omega_{1}, \eta_{1},\omega_{2}, \eta_{2},  p_i  ,  \zeta(p_i),  q_j,  \zeta(q_j), t_{i_mj_m} , t_{i_aj_a}   ,t_{i_lj_l}  \big)_{ \substack{j=1, \dots, r' \\  i=1, \dots,n' \\ m=1, \dots, u \\ a=1, \dots, v\\ l=1, \dots,s }}.\]
We proceed in three steps:

(1) We first show that $ \mathrm{t.d.}\,	K_{ M^C \oplus M^T} \big(\mathrm{periods}(M^C \oplus M^T)\big) =  \mathrm{t.d.}\, F(\rme^{\alpha_{j_l} p_{i_l}} , \rme^{\beta_{j_l} })_l, $ where 
$K_{ M^C \oplus M^T}$ is the field of definition of $M^C \oplus M^T.$ 

	We have to show that the two fields
	\begin{equation}\label{K(M^C+M^T)}
		K_{ M^C \oplus M^T} \big(\mathrm{periods}(M^C \oplus M^T)\big)=
	\end{equation}	
	\[ \QQ\big(  g_2,g_3,  \wp(q_{j_m}),\wp(p_{i_m}), \rme^{ t_{i_mj_m}} f_{q_{j_m}}(p_{i_m})  ,
	 \wp(q_{j_a}),\wp(p_{i_a}), \rme^{ t_{i_aj_a}} f_{q_{j_a}}(p_{i_a}) ,
	\wp(q_{j_l}),\wp(p_{i_l}), \rme^{ t_{i_lj_l}} f_{q_{j_l}}(p_{i_l}),\]
	\[\rme^{t_{i_lj_l}},  \rme^{\alpha_{j_l} p_{i_l}} , \rme^{\beta_{j_l} } ,  \omega_{1}, \eta_{1}, \omega_{2}, \eta_{2},  p_{i_m}  ,  \zeta(p_{i_m}), q_{j_m},  \zeta(q_{j_m}), \]
	\[ 
	 p_{i_a}  ,  \zeta(p_{i_a}), q_{j_a},  \zeta(q_{j_a}),
	p_{i_l}  ,  \zeta(p_{i_l}),  q_{j_l},  \zeta(q_{j_l}),
	t_{i_mj_m} ,t_{i_aj_a}, t_{i_lj_l} ,\alpha_{j_l} p_{i_l},  \beta_{j_l}  \Big)_{m,a,l}
	\]
	and
\[ F(\rme^{\alpha_{j_l} p_{i_l}} , \rme^{\beta_{j_l} })_l=\]
\[ \QQ\big(  g_2,g_3,  \wp(q_j),\wp(p_i), \rme^{ t_{i_mj_m}} f_{q_{j_m}}(p_{i_m}), \rme^{ t_{i_aj_a}} f_{q_{j_a}}(p_{i_a}),\rme^{ t_{i_lj_l}}  , \omega_{1}, \eta_{1},\omega_{2}, \eta_{2},\]
\[  p_i  ,  \zeta(p_i),  q_j,  \zeta(q_j), t_{i_mj_m}   , t_{i_aj_a},t_{i_lj_l},\rme^{\alpha_{j_l} p_{i_l}} , \rme^{\beta_{j_l} }  \big)_{ \substack{j=1, \dots, r' \\  i=1, \dots,n' \\ m=1, \dots, u\\ a=1, \dots, v\\ l=1, \dots,s }}.\]
	have the same transcendence degree over $\QQ$. But this is true because 
	\begin{itemize}
		\item  since $p_1, \dots ,p_{n'}, q_1,\dots ,q_{r'} $ is a $k$--basis of the vector space $<p_i,q_j>_{i,j}$,
		the numbers $p_{i_m} ,q_{j_m},p_{i_a} ,q_{j_a}, p_{i_l},  q_{j_l} $ are algebraic over $F$ for any $i_a,j_a, i_l,j_l,$ for any $i_m >n',$ and for any $j_m>r'$  (recall that  if $m=1, \dots,n'r',$ we have assumed $ i_m=1, \dots, n'$ and  $ j_m=1, \dots, r'$);
		\item  by Corollary \cite[Corollary 3.5]{BW} $\wp(q_{j_m}),\wp(p_{i_m}),\wp(q_{j_a}),\wp(p_{i_a}), \wp(q_{j_l}),\wp(p_{i_l}), \zeta(p_{i_m}),  \zeta(q_{j_m}), \\ \zeta(p_{i_a}),  \zeta(q_{j_a}), \zeta(p_{i_l}), \zeta(q_{j_l})$ are algebraic over $F$ for any $i_a,j_a,i_l,j_l,$ for any $i_m >n'$ and for any $j_m>r';$ 
		\item  let  $l=1, \dots,s'.$ The number $f_{q_{j_l}} (p_{i_l})   \rme^{\alpha_{j_l} p_{i_l}} $ is algebraic over $\QQ(g_2,g_3, \omega_1,\omega_2, \eta_1,\eta_2,p_i, \wp(p_i) , \\ \wp'(p_i), q_j,\wp(q_j) ,  \wp'(q_j))_{\substack{j=1, \dots, r' \\ i=1, \dots, n'}}$ by Proposition \ref{Cor:f_q(z)qTorsion} and by \cite[Corollary 3.5]{BW};		
		\item let  $l=s'+1, \dots,s.$ We will apply Proposition \ref{Proposition:PhiAntisymmetricQ} with $p_1=p_{i_l}= \phi_{i_lj_l}q_{j_l} $ and $p_2= \overline{\phi_{i_lj_l}}q_{j_l}$ (recall that hypothesis we have $(\phi_{i_lj_l} + \overline{\phi_{i_lj_l}})(q_{j_l}) = \frac{\omega_{j_l}}{a_{j_l}}$). Since $f_{q_{j_l}}( p_{i_l})$ belongs to $K_{ M^C \oplus M^T} \big(\mathrm{periods}(M^C \oplus M^T)\big)$, according to the last statement of Proposition \ref{Proposition:PhiAntisymmetricQ} $f_{q_{j_l}}(- p_{i_l})$ also belongs to it.
By Proposition \ref{Cor:f_q(p)pTorsion}  and by \cite[Corollary 3.5]{BW} the number $f_{q_{j_l}}(\frac{\omega_{j_l}}{a_{j_l}}) \rme^{\beta_{j_l} } $ is algebraic over $\QQ(g_2,g_3, \omega_1,\omega_2, \eta_1,\eta_2,p_i, \wp(p_i), \wp'(p_i), q_j, \wp(q_j), \wp'(q_j))_{\substack{j=1, \dots, r' \\ i=1, \dots, n'}},$	which implies according to Proposition \ref{Proposition:PhiAntisymmetricQ} (2) and to \cite[Corollary 3.5]{BW}  that also the number $f_{q_{j_l}}(\overline{\phi_{i_lj_l}}q_{j_l})  \rme^{\beta_{j_l} } $ is algebraic over this field. Finally by Proposition \ref{Proposition:PhiAntisymmetricQ} (1) 
 and \cite[Corollary 3.5]{BW} we can conclude that
 the number $f_{q_{j_l}}(p_{i_l})f_{q_{j_l}}(\overline{\phi_{i_lj_l}}q_{j_l})  \rme^{\beta_{j_l} } $ is algebraic over  $\QQ( g_2,g_3,\omega_1,\omega_2, \eta_1,\eta_2, p_i, \wp(p_i),  q_j,\wp(q_j))_{i,j};$
		\item the number $\alpha_{j_l} p_{i_l}$ is algebraic over $\QQ(g_2,g_3, p_i, q_j)_{\substack{j=1, \dots, r' \\ i=1, \dots, n'}}$ by \cite[Lemma 3.1 (3)]{BW};
	\item the number $\beta_{j_l}$ is algebraic over the field $\QQ( g_2,g_3,\omega_1,\omega_2, \eta_1,\eta_2,p_i, \wp(p_i), \zeta(p_i), q_j, \wp(q_j),\\ \zeta(q_j))_{\substack{j=1, \dots, r' \\ i=1, \dots, n'}}$ by \cite[Corollary 3.5]{BW}. 
	\end{itemize}

	(2) Now we compute the dimensions of the motivic Galois groups of $M^C \oplus M^T$ and of $M^C$.
	\par\noindent	By Theorem \ref{Teo:dimGal(M)} 
\[ \dim \Galmot (M) = \dim \Galmot (M^C) =\frac{4}{\dim_\QQ k} + 2(n'+r')+u+v+s. \]
We add the index $T$ (\textit{resp}. $CT$) to the pure motives underlying the unipotent radical of the 1-motive $M^T$ (\textit{resp}. $M^C \oplus M^T$). 
Since the 1-motive $M^T$ has no abelian part, according to Theorem \ref{Teo:dimGal(M)} 
\[		\dim B_T = \dim Z_T'(1) =0 \quad  \mathrm{and}  \quad \dim Z_T(1) / Z_T'(1) = \dim_\QQ <t_{i_lj_l}, \alpha_{j_l} p_{i_l}, \beta_{j_l}   >_{l=1, \dots,s},
\]
where  $<t_{i_lj_l}, \alpha_{j_l} p_{i_l}, \beta_{j_l} >_{l}$ is the sub $\QQ$--vector space of $\CC / 2 \pi \ii \QQ$ generated by the classes of the complex numbers $t_{i_lj_l}, \alpha_{j_l} p_{i_l}$ and $ \beta_{j_l} $  modulo $2 \pi \ii \QQ.$
Because of the equalities $\eqref{eq:dimAb}$ and \eqref{eq:dimTorus} and Remark \ref{Contribution}

\begin{align*}
	\dim B_{CT}&= \dim B_C = n'+r', \\
	\dim Z_{CT}'(1)&= \dim Z_{C}'(1) =u, \\
	\dim Z_{CT}(1)/Z_{CT}'(1)&= v+ \dim_\QQ < t_{i_lj_l} ,\alpha_{j_l} p_{i_l}, \beta_{j_l} >_{l}  .
\end{align*}
The contribution $v$
comes from the $\mathrm{LieBracket}$-part of $M^C$ and is independent of the toric contribution of
$M^T$, since the latter has no abelian part. On the other hand, the classes
$t_{i_lj_l}$
already occur in the $\mathrm{NoLieBracket}$-contribution of
$Z_C(1)/Z'_C(1)$ and therefore may overlap with the toric part of $M^T$.
Consequently, only the $\mathrm{LieBracket}$-contribution contributes an additional direct summand, whereas the $\mathrm{NoLieBracket}$-contribution must be considered together with the classes $
\alpha_{j_l}p_{i_l},  \beta_{j_l}.$
Using the short exact sequence \eqref{eq:shortexactsequenceUR} we conclude that 
	\[ \dim \Galmot (M^C \oplus M^T) = \frac{4}{\dim_\QQ k} + 2(n'+r')+u+v+  \dim_\QQ < t_{i_lj_l} ,\alpha_{j_l} p_{i_l}, \beta_{j_l} >_{l} . \]
	(3) By definition of the two fields \eqref{K(periods(M^C))} and \eqref{K(M^C+M^T)} we have that 
	\[\mathrm{t.d.}\,   K_{M^C} \big(\mathrm{periods}(M^C), \rme^{
		\alpha_{j_l} p_{i_l}} ,  \rme^{
	\beta_{j_l}} ,\alpha_{j_l} p_{i_l}, \beta_{j_l}\big)_{l=1, \dots,s} = \mathrm{t.d.}\,  K_{ M^C \oplus M^T} \big(\mathrm{periods}(M^C \oplus M^T)\big).\] 
	Let $N$ be the dimension of the sub $\QQ$--vector space $< t_{i_lj_l} ,\alpha_{j_l} p_{i_l}, \beta_{j_l} >_{l}$ of $\CC / 2 \pi \ii \QQ$ generated by the classes of the complex numbers $  t_{i_lj_l} ,\alpha_{j_l} p_{i_l}$ and $ \beta_{j_l}$  modulo $2 \pi \ii \QQ.$ Since $ \dim_\QQ < t_{i_lj_l} >_l=s$, we have that $N \geqslant s$ and without loss of generalities, we assume that $ t_{i_1j_1}, \dots  t_{i_sj_s}, \delta_1, \dots, \delta_{N-s}$ is a basis of $< t_{i_lj_l} ,\alpha_{j_l} p_{i_l}, \beta_{j_l} >_{l}$. Recalling that by \cite[Lemma 3.1 (3)]{BW} the complex number 
	 $\alpha_{j_l} p_{i_l}$ is algebraic over $\QQ(g_2,g_3, p_i,q_j)_{\substack{ j=1, \dots, r' \\ i=1, \dots, n'}}$ and 
	by \cite[Corollary 3.5]{BW} the complex number $\beta_{j_l}$ is algebraic over the field $\QQ( g_2,g_3,\omega_1,\omega_2, \eta_1,\eta_2,p_i ,\wp(p_i), \zeta(p_i), q_j, \wp(q_j), \zeta(q_j))_{\substack{j=1, \dots, r' \\ i=1, \dots, n'}},$ according to \cite[Lemma 4.4]{BW}) we have the equalities
	\[	\mathrm{t.d.}\, K_{M^C} \big(\mathrm{periods}(M^C), \rme^{
		\alpha_{j_l} p_{i_l}}, \rme^{
		\beta_{j_l}}, \alpha_{j_l} p_{i_l}, 	\beta_{j_l} \big)_{l=1,\dots,s}= \mathrm{t.d.}\, K_{M^C} \big(\mathrm{periods}(M^C), \rme^{\delta_l}\big)_{l=1, \dots,N-s}\]
	and
	\[ \mathrm{t.d.} F(\rme^{\alpha_{j_l} p_{i_l}}, \rme^{\beta_{j_l}}, \alpha_{j_l} p_{i_l}, 	\beta_{j_l} )_{l=1,\dots,s} = \mathrm{t.d.}  F(\rme^{\delta_l})_{l=1, \dots,N-s} .\]
	Now the Grothendieck-Andr\'{e} period Conjecture applied to $M^C \oplus M^T,$ that we can made explicit thanks to step (1) and (2), furnishes 
	\[ \mathrm{t.d.}\, K_{M^C} \big(\mathrm{periods}(M^C), \rme^{\delta_l}\big)_{l=1, \dots,N-s} = \mathrm{t.d.}\, K_{ M \oplus M^T} \big(\mathrm{periods}(M \oplus M^T)\big)=\mathrm{t.d.}  F(\rme^{\delta_l})_{l=1, \dots,N-s}  \geqslant \]
	\[\frac{4}{\dim_\QQ k} + 2(n'+r')+u+v+  N . \]	
	Removing the $N-s$ numbers $\rme^{\delta_1}, \dots, \rme^{\delta_{N-s}}$ we obtain 
	\[ \mathrm{t.d.}\, K_{M^C} \big(\mathrm{periods}(M^C)\big)=  \mathrm{t.d.} F \geqslant \frac{4}{\dim_\QQ k} + 2(n'+r')+u+ v+ s.\]
\end{proof}

We can now prove the main theorem of this paper:

\begin{proof}[Proof of Theorem \ref{GA<=>Schanuel}] We have to prove that the semi-elliptic Conjecture \ref{s-eSC} and the Gro\-then\-dieck-André period Conjecture applied to the 1-motive $M$ \eqref{GPC-1-motive}, made explicit in Corollary \ref{GAPconjectureClean}, are equivalent. To simplify the notation, we assume throughout this proof that no
	Cartier-dual pairs occur. The necessary correction in the presence of
	Cartier-dual pairs is accounted for by the term $-\frac{1}{2}|D|$
	appearing in Conjecture \ref{s-eSC}.

\medspace

\par\noindent
\textit{Conjecture \ref{s-eSC}  $ \Longrightarrow $ Conjecture stated in  Corollary \ref{GAPconjectureClean}.}

Consider the 1-motive $M$ \eqref{GPC-1-motive} defined by the complex numbers \eqref{Points} $q_j,p_i$ and $t_{ij}$ for $j=1,\dots,r$ and $i=1, \dots,n.$

Let $p_1, \dots ,p_{n'}, q_1,\dots ,q_{r'} $ be a basis of the sub $k $--vector space $<p_i,q_j>_{i,j}$ of  $\CC/(\Omega\otimes_\ZZ\QQ)$ generated by the classes of $p_1, \dots, p_n,q_1, \dots, q_r. $

Let $\gamma_1, \dots ,\gamma_{u}$ be a basis of the sub $\QQ $--vector space $ < \beta_{i,j}+\beta_{i,j}^t>_{  (i,j) \in \mathrm{LieBracket}}$ of $\mathrm{Hom}_\QQ(B,B^*) $ generated by the homomorphisms $\beta_{i,j}+\beta_{i,j}^t$ with $(i,j) \in \mathrm{LieBracket}.$ Recall that for $m=1, \dots,u,$ we have assumed $\gamma_m =\beta_{i_m,j_m}+\beta_{i_m,j_m}^t$ for some $(i_m,j_m) \in \mathrm{LieBracket}$, and $ i_m =1, \dots, n',$  $ j_m =1, \dots, r'$ for $m \leqslant n'r'.$

 Let $
\theta_1,\ldots,\theta_v $
be a $\QQ$--basis of the sub $\QQ$--vector space of
$\CC/2\pi \ii\QQ$ generated by the classes of
$\sum_{(i,j)\in \mathrm{LieBracket}}
\alpha_{ij}\log(s_{ij}).$
 Recall that for $a=1, \dots,v,$ we have assumed $\theta_a=t_{i_aj_a}$ for some $(i_a,j_a) \in \mathrm{LieBracket}.$

Let $t_1, \dots ,t_{s}$ be a basis of the sub $\QQ $--vector space $<t_{ij}>_{(i,j) \in \mathrm{NoLieBracket}}$ of $\CC / 2\pi\ii \QQ$ generated by the classes of $t_{ij} $ with $(i,j) \in \mathrm{NoLieBracket}$ \eqref{def:NoLieBracket}. For $l=1, \dots,s,$ we have assumed $t_l= t_{i_lj_l}$ for some $(i_l,j_l) \in \mathrm{NoLieBracket}.$ Let $N$ be the dimension of the sub $\QQ$--vector space $< t_l,t_{i_mj_m} >_{l,m}$ of $\CC / 2 \pi \ii \QQ$ generated by the classes of the complex numbers $  t_l$ and $ t_{i_mj_m} $  modulo $2 \pi \ii \QQ.$ Since $ \dim_\QQ < t_l>_l=s$, we have that $N \geqslant s$ and without loss of generalities, we may assume that $ t_1, \dots  t_s, \delta_1, \dots, \delta_{N-s}$ is a basis of $< t_l ,t_{i_mj_m} >_{l,m}$. 

We distinguishes two cases:

(1) CM case. 
Set $t_0= 2 \pi \ii $ and $q_0= \frac{\omega_1}{2}.$ Remark that $\zeta(q_{0})= \frac{\eta_1}{2}$, by \cite[Lemma 3.1 (3)]{BW} the number $\wp(q_{0})$ is algebraic over the field $\QQ (g_2,g_3),$ and finally by Lemma \ref{OmegaEta} the numbers $\omega_2$ and $\eta_2$ are algebraic over the field $k(g_2,g_3,\omega_1,\eta_1).$

Since $ t_0,t_1, \dots  t_s, \delta_1, \dots, \delta_{N-s}$ are $\QQ$--linearly independent and $p_1, \dots ,p_{n'}, q_0, q_1,\dots ,q_{r'} $ are $k$--linearly independent, the semi-elliptic Conjecture \ref{s-eSC} applied to the complex numbers $t_0,t_1, \dots  t_s,\\ \delta_1, \dots, \delta_{N-s},p_{i_m},q_0, q_{j_m} $ for $m=1, \dots,u$ and $p_{i_a}, q_{j_a}, t_{i_a j_a} $ for $a=1, \dots,v$ reads
 
 \[ \mathrm{t.d.}\,  \QQ \Big(g_2,g_3,  \wp(q_j),\wp(p_i), \rme^{t_{i_mj_m}} f_{q_{j_m}}(p_{i_m})  , \rme^{t_{i_aj_a}} f_{q_{j_a}}(p_{i_a})  ,\rme^{t_{i_lj_l}}, \omega_{1}, \eta_{1},\omega_{2}, \eta_{2},  2 \pi \ii,\]
 \[ p_i  ,  \zeta(p_i),  q_j,  \zeta(q_j) , t_{i_mj_m}, t_{i_aj_a},t_{i_lj_l},\rme^{\delta_o} \Big)_{ \substack{l=1, \dots, s \\m=1, \dots,u \\a=1, \dots,v  \\o=1, \dots, N-s \\  j=1, \dots,r' \\ i=1, \dots, n' }  }  = \]
 \[ \mathrm{t.d.}\,  \QQ \Big(t_l, \delta_o,\rme^{t_l}, \rme^{\delta_o},   g_2,g_3,q_0,q_j,p_i, \wp(q_0), \wp(q_j),\zeta(q_0), \zeta(q_j), \wp(p_i), \zeta(p_i),\]
 \[ f_{q_{j_m}}(p_{i_m})  , t_{i_aj_a} ,  \rme^{t_{i_aj_a}} f_{q_{j_a}}(p_{i_a})\Big)_{\substack{l=0, \dots, s \\m=1, \dots,u \\a=1, \dots,v  \\o=1, \dots, N-s \\  j=1, \dots,r' \\ i=1, \dots, n' } }   \geqslant \]
\[ N+1+ 2 (n'+r'+1)+ u+v -1=  \frac{4}{\dim_\QQ k}  +
\dim \UR(M) +N-s.  \]
Removing the $N-s$ complex numbers $\rme^{\delta_1}, \dots , \rme^{\delta_{N-s}},$ we conclude. 

(2) Non--CM case. 
Set $t_0= 2 \pi \ii, q_{0}= \frac{\omega_1}{2}, q_{-1}= \frac{\omega_2}{2}.$
 Observe that $\zeta(q_{0})= \frac{\eta_1}{2}, \zeta(q_{-1})= \frac{\eta_2}{2}$, and finally by \cite[Lemma 3.1 (3)]{BW} the numbers $\wp(q_{-1}), \wp(q_0)$ are algebraic over the field $\QQ (g_2,g_3).$
 
 Since $ t_0,t_1, \dots  t_s, \delta_1, \dots, \delta_{N-s}$ are $\QQ$-- linearly independent and $p_1, \dots ,p_{n'},q_{-1}, q_0, q_1,  \dots ,q_{r'} $ are $k$--linearly independent,
  the semi-elliptic Conjecture \ref{s-eSC} applied to the complex numbers $t_0,t_1, \dots , \\ t_s, \delta_1, \dots, \delta_{N-s},p_{i_m},q_{-1},q_0, q_{j_m} $ for $m=1, \dots,u$ and $p_{i_a}, q_{j_a}, t_{i_a j_a} $ for $a=1, \dots,v$ furnishes
 
 \[ \mathrm{t.d.}\,  \QQ \Big(g_2,g_3,  \wp(q_j),\wp(p_i),  \rme^{t_{i_mj_m}} f_{q_{j_m}}(p_{i_m}) ,\rme^{t_{i_aj_a}} f_{q_{j_a}}(p_{i_a})  ,\rme^{t_{i_lj_l}}, \omega_{1}, \eta_{1},\omega_{2}, \eta_{2},  2 \ii \pi,\]
 \[ p_i  ,  \zeta(p_i),  q_j,  \zeta(q_j) , t_{i_mj_m},t_{i_aj_a}, t_{i_lj_l},\rme^{\delta_o} \Big)_{ \substack{l=1, \dots, s \\m=1, \dots,u \\a=1, \dots,v \\o=1, \dots, N-s \\  j=1, \dots,r' \\ i=1, \dots, n' }  }  = \]
 \[ \mathrm{t.d.}\,  \QQ \Big(t_l, \delta_o,\rme^{t_l}, \rme^{\delta_o}, g_2,g_3,q_{-1},q_0,q_j,p_i, \wp(q_{-1}), \wp(q_0), \wp(q_j),\zeta(q_{-1}), \zeta(q_0), \zeta(q_j),\]
 \[ \wp(p_i), \zeta(p_i), f_{q_{j_m}}(p_{i_m}) , t_{i_aj_a} ,\rme^{t_{i_aj_a}} f_{q_{j_a}}(p_{i_a})  \Big)_{\substack{l=0, \dots, s \\m=1, \dots,u \\a=1, \dots,v  \\o=1, \dots, N-s \\  j=1, \dots,r' \\ i=1, \dots, n' } }   \geqslant \]
 \[ N+1+ 2 (n'+r'+2)+u+v-1 =  \frac{4}{\dim_\QQ k}  +
 \dim \UR(M) +N-s.  \]
 Removing the $N-s$ complex numbers $\rme^{\delta_1}, \dots , \rme^{\delta_{N-s}},$ we conclude.

\medspace

\par\noindent
\textit{Conjecture stated in  Corollary \ref{GAPconjectureClean}  $ \Longrightarrow $ Conjecture \ref{s-eSC}.}
Let $t_1, \dots,t_s$ be $\QQ$--linearly independent complex numbers and
	$ q_1, \dots,q_r,p_1,\dots,p_n$ be $k$--linearly independent complex numbers in $\CC \smallsetminus \Omega.$
	 Choose subsets $I\subseteq \{1,\ldots,n\}$ and $J\subseteq \{1,\ldots,r\}$ such that the classes of $
	\{p_i, q_j\}_{\substack{ i\in I \\ j\in J}}$ form a $k$--basis of the sub $k$--vector space $ <p_i,q_j>_{\substack{i=1, \dots,n \\ j=1, \dots,r}}$ of  $ \CC / (\Omega\otimes_{\ZZ}\QQ)$ generated by the classes of
	$p_1,\ldots,p_n,q_1,\ldots,q_r$ modulo $\Omega\otimes_{\ZZ}\QQ.$
	Set $\tau_p:=n-|I|$ and $\tau_q:=r-|J|.$ 
	After a suitable reordering of the points, we may assume
\[ I= \{\tau_p+1,\ldots,n \} \qquad \mathrm{and} \qquad J=\{\tau_q+1,\ldots,r \}.\]
		Hence the classes of $
	p_{\tau_p+1},\ldots,p_n,
	q_{\tau_q+1},\ldots,q_r$
	form a $k$-basis of the vector space $ <p_i,q_j>_{\substack{i=1, \dots,n \\ j=1, \dots,r}}.$
	Moreover let 	$ q_{r+1}, \dots,q_{r'},p_{n+1},\dots,p_{n'}$ be complex numbers in $\CC \smallsetminus \Omega$ such that
	\begin{align*}
		\dim_k  <p_i,q_j>_{\substack{i=1, \dots,n' \\ j=1, \dots,r'}} &= (n-\tau_p) + (r- \tau_q) ,\\
		\dim_\QQ  < \beta_{i,j}+\beta_{i,j}^t>_{\substack{i=\tau_p +1, \dots,n' \\ j=\tau_q+1, \dots,r'}} &=(n' - \tau_p)(r'- \tau_q).
	\end{align*}
Finally let	$ q_{r'+1}, \dots,q_{r''},p_{n'+1},\dots,p_{n''}$ be complex numbers in $\CC \smallsetminus \Omega$ and for $i=n'+1,\ldots,n'', j=r'+1,\ldots,r'',$ $t_{ij}$ be complex numbers such that
\begin{align*}
		\dim_k  <p_i,q_j>_{\substack{i=1, \dots,n'' \\ j=1, \dots,r''}} &= (n-\tau_p) + (r- \tau_q)\\
	\dim_\QQ  < \beta_{i,j}+\beta_{i,j}^t>_{\substack{i \in I \cup \{n +1, \dots,n''\} \\ j \in J \cup \{ r+1 , \dots,r''\}}}& = (n' - \tau_p)(r'- \tau_q),\\
	 \dim_\QQ 	< \textstyle{\sum_{\substack{ n'+1\leqslant i \leqslant n'' \\ r'+1 \leqslant j \leqslant r''}}}	\alpha_{ij}t_{ij}> & = (n''-n')(r''-r')
\end{align*}
Set
\[F:= \QQ \Big(t_l,\rme^{t_l}, g_2,g_3, q_j, p_i, \wp(q_j), \zeta(q_j),\wp(p_i), \zeta(p_i),  f_{q_{j'}}(p_{i'}), t_{i''j''} ,\rme^{t_{i''j''}} f_{q_{j''}}(p_{i''})\Big)_{\substack{l=1, \dots, s \\  j=1, \dots,r \\ i=1, \dots, n \\  j'=\tau_q+1, \dots,r' \\ i'=\tau_p+1, \dots, n' \\  j''=r', \dots,r'' \\ i''=n', \dots, n''}}.\]

If $t_1, \dots,t_s$ are $\QQ$--linearly independent complex numbers, we set
\[\tor (t_l): = \dim_\QQ \Big( ( 2 \pi \ii \QQ ) \cap \big(  \sum_{l=1}^s \QQ t_l \big) \Big) . \]

We have to find the following lower bounds for the transcendence degree of $F$ over $\QQ:$

\begin{enumerate}
	\item $s+2(r+n)+r'n' + (n''-n')(r''-r')$ if  $\tor(t_l)= \tor (p_i,q_j)=0, $
	\item $s+ 2(n+r) +n' (r'- 1)+ (n''-n')(r''-r')$ if $\tor(t_l)=\tau_p=0$ and  $\tau_q=1,$
	\item  $s+ 2(n+r) + r'(n'-1)+ (n''-n')(r''-r')$ if $\tor(t_l)= \tau_q=0$ and  $\tau_p=1,$
	\item $s+ 2(n+r) +(n'-1)(r'-1)+ (n''-n')(r''-r')$ if $\tor(t_l)= 0,$ and $\tau_q=\tau_p=1,$
	\item  $s+ 2(n+r) +n'(r'-2)+ (n''-n')(r''-r')$ if $\tor(t_l)= 0$ and $\tau_q=2,$
	\item  $s+ 2(n+r) + r'(n'-2)+ (n''-n')(r''-r')$ if $\tor(t_l)= 0$ and  $\tau_p=2,$
	
	\item $s+ 2(n+r) +n'r'+ (n''-n')(r''-r')$ if $\tor(t_l)= 1$,  $\tor (p_i,q_j)=0,$
	
	\item 
	\begin{enumerate}
		\item 	 $s+ 2(n+r) +n'(r'-1)+ (n''-n')(r''-r')-1$ if $\tor(t_l)=\tau_q=1$, $\tau_p=0,$ and $\cE$ is CM,
		\item 	 $s+ 2(n+r) +n'(r'-1)+ (n''-n')(r''-r')$ if $\tor(t_l)=\tau_q=1$, $	\tau_p=0,$ and $\cE$ is non--CM,
	\end{enumerate}

	\item 
	\begin{enumerate}
		\item 	 $s+ 2(n+r) +r'(n'-1)+ (n''-n')(r''-r')-1$ if $\tor(t_l)=\tau_p=1$, $\tau_q=0,$ and $\cE$ is CM,
		\item 	 $s+ 2(n+r) +r'(n'-1)+ (n''-n')(r''-r')$ if $\tor(t_l)=\tau_p=1$, $	\tau_q=0,$ and $\cE$ is non--CM,
	\end{enumerate}
	
	
	\item  $s+ 2(n+r) +(n'-1)(r'-1)+ (n''-n')(r''-r')-1 $ if $\tor(t_l)= \tau_q=\tau_p=1,$			
	\item  $s+ 2(n+r) +n'(r'-2) + (n''-n')(r''-r')-1$ if $\tor(t_l)= 1$,  $\tau_q=2,$
	\item $s+ 2(n+r) +r'(n'-2)+ (n''-n')(r''-r')-1$ if $\tor(t_l)= 1,$   $	\tau_p=2.$
	
\end{enumerate}

The cases \textit{(2)} and \textit{(3)}, \textit{(5)} and \textit{(6)},\textit{ (8)} and \textit{(9)}, \textit{(11)} and \textit{(12)} are dual in the complex numbers $p_i$ and $q_j$. The cases \textit{(4),(5),(6),(10), (11)} and \textit{(12)} appear only if the elliptic curve is non--CM. Moreover remark that $ 2 \pi \ii \subset \sum_l \QQ t_l$ and $\Omega \subset \sum_i  k p_i + \sum_j k q_j$ in the cases \textit{(8) (a), (9)(a), (10),(11)} and \textit{(12)}.

\vskip 0.4 true cm

We consider two cases for the numbers $t_1, \dots,t_s$, denoted $(T)$ and ($T^\star$), with some assumptions that will introduce no loss of generality:

-  Case ($T$): $\tor (t_l)=1.$ Let $t_s =2\pi\rmi$ and $\dim_\QQ < t_1,\dots,t_{s-1} > =s-1.$ Notice that $F$ contains $2\pi\rmi.$

- Case ($T^\star$): $\tor (t_l)=0.$
 
 \medspace
 
We consider six cases for the numbers $ q_1, \dots,q_r,p_1,\dots,p_n$ denoted ($PQ$), ($PQ^\star$), ($P^\star Q$), ($\tilde{P}Q^\star$), ($P^\star \tilde{Q}$) and ($P^\star Q^\star$), with again some harmless assumptions:

-  Case ($P Q^\star$):  (only if $\cE$ is non--CM) $\tau_p=2,$ 
$p_1=\frac{\omega_1}{2}$, $p_{2}=\frac{\omega_2}{2} $ and $\dim_k <p_3, \dots, p_{n}, q_1, \dots,  \\ q_r>  =(n-2)+r.$
Notice that $\Omega\subset k p_1+\cdots+k p_n$  and  $F (\omega_{1}, \eta_{1},\omega_{2}, \eta_{2}) =  F.$

- Case ($P^\star Q$):  (only if $\cE$ is non--CM) $\tau_q=2,$ $q_{1}=\frac{\omega_1}{2}$, $q_{2}=\frac{\omega_2}{2} $ and $\dim_k <p_1, \dots, p_n, q_3, \dots, q_{r}> \\ = n+(r-2).$ Notice that $\Omega\subset k q_1+\cdots+k q_r$  and  $ F (\omega_{1}, \eta_{1},\omega_{2}, \eta_{2}) =  F.$

- Case ($PQ$):  (only if $\cE$ is non--CM) $\tau_p=\tau_q=1,$ $p_1=\frac{\omega_1}{2}$, $q_1=\frac{\omega_2}{2} $ and $\dim_k <p_2, \dots, p_{n}, q_2, \dots, \\ q_{r}> =(n-1)+(r-1).$  Notice that $\Omega\subset k p_1+\cdots+k p_n+ k q_1+\cdots+k q_r$  and  $F (\omega_{1}, \eta_{1},\omega_{2}, \eta_{2}) =  F.$

- Case $(\tilde{P}Q^\star):$ $\tau_p=1$ and  $\tau_q=0,$ $p_1=\frac{\omega_1}{2}$ and $\dim_k<p_2, \dots, p_{n}, q_1, \dots, q_{r}>=(n-1)+r.$ 
Notice that $F (\omega_{1}, \eta_{1},\omega_{2}, \eta_{2}) =  F (\omega_{2}, \eta_{2}).$

- Case $(P^\star\tilde{Q}):$ $\tau_q=1$ and  $\tau_p=0,$ $q_1=\frac{\omega_1}{2}$ and $\dim_k<p_1, \dots, p_{n}, q_2, \dots, q_{r}>= n+(r-1).$ 
Notice that  $F (\omega_{1}, \eta_{1},\omega_{2}, \eta_{2}) =  F (\omega_{2}, \eta_{2}).$ 

- Case $(P^\star Q^\star):$ $\tor (p_i,q_j)=0.$

\vskip 0.4 true cm
	Consider the 1-motive $M=[u:\ZZ \longrightarrow  G^{n} ],  u(1)=(R_1, \dots, R_{n})  \in G^{n}(\CC),$ defined by the points $p_{i'},q_{j'},t_{i'j'}=0$ for $j'=\tau_q+1, \dots,r'$ and $ i'=\tau_p+1, \dots, n'  ,$ and
	 $ p_{i''}, q_{j''}, t_{i''j''} $
	 for $j''= r', \dots,r''$ and $ i''= n', \dots, n''. $
 By  Corollary \ref{GAPconjectureClean} the field \eqref{field} generated by the periods of $M$ over the field of definition of $M$ is
	\begin{equation}\label{periodesM}
		K \big(\mathrm{periods}(M)\big) =
	\end{equation}
	\[\QQ\Big( g_2,g_3,  \wp(q_j),\wp(p_i),  f_{q_{j'}}(p_{i'}) ,\rme^{t_{i''j''}} f_{q_{j''}}(p_{i''}), \omega_{1}, \eta_{1},\omega_{2}, \eta_{2},  2 \ii \pi, p_i  ,  \zeta(p_i),  q_j,  \zeta(q_j) , t_{i''j''}  \Big)_{ \substack{j=\tau_q+1, \dots, r \\  i=\tau_p+1, \dots,n \\j'=\tau_q+1, \dots, r' \\  i'=\tau_p+1, \dots,n'\\j''= r', \dots, r'' \\  i''= n', \dots,n''} },
	\]
	and concerning the pure motives underlying the unipotent radical of $M,$ we have
	\begin{align}\label{pureM}
		\dim_k B&= (n-\tau_p) + (r- \tau_q), \\
	\nonumber 	\dim_\QQ Z(1)&= \dim_\QQ Z'(1)= (n'-\tau_p)( r'- \tau_q),\\
	\nonumber 	\dim_\QQ Z(1)/Z'(1) &= (n''-n')(r''-r').
	\end{align}

	Consider the 1-motive
	\[
	M^T=[u^T:\ZZ \to  \GG_m^{s}],\quad  u^T(1)= (\rme^{t_1}, \dots ,\rme^{t_s} ) \in   \GG_m^{s} (\CC).
	\]
Adding the index $T$ to the pure motives underlying $\UR(M^T)$, we have 
	\begin{align}\label{pureM^T}
		\dim_k B_T&= 0, \\
	\nonumber 	\dim_\QQ Z'_T(1)&= 0,\\
	\nonumber 	\dim_\QQ Z_T(1)&= \dim_\QQ Z_T(1)/Z'_T(1) = \dim_\QQ <  t_1, \dots, t_s  > ,
	\end{align}
	where  $< t_l   >_{l=1, \dots,s}$ is the sub $\QQ$--vector space of $\CC / 2 \pi \ii \QQ$ generated by the classes of the complex numbers $ t_1, \dots, t_s $  modulo $2 \pi \ii \QQ.$

	If $K_{ M \oplus M^T}$ denotes the field of definition of the direct sum $M \oplus M^T,$ observe that
	\begin{equation}\label{periodsMM^T}
	K_{ M \oplus M^T}\big(\mathrm{periods}(M \oplus M^T)\big) = \QQ\Big( g_2,g_3,  \wp(q_j),\wp(p_i),  f_{q_{j'}}(p_{i'}) ,\rme^{t_{i''j''}} f_{q_{j''}}(p_{i''}), \rme^{t_l} ,
\end{equation}
	\[ \omega_{1}, \eta_{1},\omega_{2}, \eta_{2},  2 \ii \pi, p_i  ,  \zeta(p_i),  q_j,  \zeta(q_j) , t_{i''j''},t_l  \Big)_{ \substack{j= \tau_q+1, \dots, r \\  i=\tau_p+1, \dots,n \\j'=\tau_q+1, \dots, r' \\  i'=\tau_p+1, \dots,n' \\l=1, \dots,s \\j''= r', \dots, r'' \\  i''= n', \dots,n''} }.
\]
We add the index $-T$ to the pure motives underlying  $\UR(M \oplus M^T).$
 Since by Remark \ref{Contribution} the contributions of the tori underlying the Lie algebra of the unipotent radicals $\UR(M)$ and $\UR(M^T)$ are independent, from equalities \eqref{pureM} and \eqref{pureM^T} we conclude that 
\begin{align*}
	\dim B_{-T}&= \dim B = (n-\tau_p) + (r- \tau_q), \\
	\dim Z_{-T}'(1)&= \dim Z'(1) =(n'- \tau_p)( r'- \tau_q), \\
	\dim Z_{-T}(1)/Z_{-T}'(1)&= (n''-n')(r''-r')+ \dim_\QQ <  t_1, \dots, t_s  >.
\end{align*}	
	Hence by the short exact sequence \eqref{eq:shortexactsequenceUR} we conclude that 
\begin{equation} \label{dimGmotMM^T}
	\dim \Galmot (M \oplus M^T) = 
\end{equation}
\[	 \frac{4}{\dim_\QQ k} + 2\big(  (n-\tau_p) + (r- \tau_q) \big) +(n'-\tau_p)( r'- \tau_q)+ \dim_\QQ <  t_1, \dots, t_s  >. \]
Applying the Grothendieck-Andr\'{e} period Conjecture to the 1-motive $M \oplus M^T,$ that we can made explicit thanks to \eqref{periodsMM^T} and to \eqref{dimGmotMM^T}, we get
\begin{equation}\label{GACMM^T}
	\QQ\Big( g_2,g_3,  \wp(q_j),\wp(p_i),  f_{q_{j'}}(p_{i'}) ,\rme^{t_{i''j''}} f_{q_{j''}}(p_{i''}), \rme^{t_l} ,
\end{equation}
\[ \omega_{1}, \eta_{1},\omega_{2}, \eta_{2},  2 \ii \pi, p_i  ,  \zeta(p_i),  q_j,  \zeta(q_j) ,t_{i''j''}, t_l  \Big)_{ \substack{j=\tau_q+1, \dots, r \\  i=\tau_p+1, \dots,n \\j'=\tau_q+1, \dots, r' \\  i'=\tau_p+1, \dots,n' \\l=1, \dots,s\\j''= r', \dots, r'' \\  i''= n', \dots,n''} }\geqslant\]
\[\frac{4}{\dim_\QQ k} + 2\big( (n-\tau_p) + (r- \tau_q)\big) +(n'-\tau_p)( r'- \tau_q)+ (n''-n')(r''-r')+ \dim_\QQ <  t_1, \dots, t_s  > . \]

We now consider each of the twelve cases:

\vskip 0.3truecm
\par
(1) ($T^\star P^\star Q^\star$): Legendre relation \eqref{Equation:Legendre} implies that
\[
	\mathrm{t.d.}\,K_{ M \oplus M^T} \big(\mathrm{periods}(M \oplus M^T)\big)=  \mathrm{t.d.} F (\omega_{1}, \eta_{1},\omega_{2}, \eta_{2}).
\]
 We distinguished two cases:

1.1) CM case: By Lemma \ref{OmegaEta} the numbers $\omega_2$ and $\eta_2$ are algebraic over the field $k(g_2,g_3,\omega_1,\eta_1),$ and so the inequality \eqref{GACMM^T} furnishes
\begin{align*} \mathrm{t.d.} F  (\omega_1, \eta_1)=\mathrm{t.d.} F (\omega_{1}, \eta_{1},\omega_{2}, \eta_{2}) &= \mathrm{t.d.}\, K_{ M \oplus M^T} \big(\mathrm{periods}(M \oplus M^T)\big) \\
	&=
2 + 2(n+r) +n'r'+ (n''-n')(r''-r')+ s . 
\end{align*}
Removing the two numbers $\omega_1, \eta_1,$ we get the expected result.

1.2) non--CM case: From the inequality \eqref{GACMM^T} we get
\begin{align*} \mathrm{t.d.} F (\omega_{1}, \eta_{1},\omega_{2}, \eta_{2}) &=
 \mathrm{t.d.}\, K_{ M \oplus M^T} \big(\mathrm{periods}(M \oplus M^T)\big)\\
 &=4+ 2(n+r) +n'r'+ (n''-n')(r''-r')+ s .
\end{align*}
Removing the four numbers $\omega_{1}, \eta_{1},\omega_{2}, \eta_{2}$, we conclude.

\vskip 0.3truecm
\par 
(2) ($T^\star P^\star\tilde{Q}$): $ q_1=\frac{\omega_1}{2}$ and $ F (\omega_{1}, \eta_{1},\omega_{2}, \eta_{2}) =  F (\omega_{2}, \eta_{2}).$ 
 The number $\wp(q_1)$ is algebraic over the field $\QQ (g_2,g_3)$ and $\zeta(q_1)= \frac{\eta_1}{2} \in \overline{\QQ(g_2,g_3)}(\eta_1,\eta_2)$ by \cite[Lemma 3.1 (3)]{BW}. According to Proposition \ref{Cor:f_q(z)qTorsion}, for $i=1, \dots, n,$ the number $f_{q_1}(p_i)$ belongs to $\overline{\QQ (g_2,g_3)} ( \wp(p_i),\wp'(p_i)).$ We distinguished two cases:

2.1) CM case:  By Lemma \ref{OmegaEta} the numbers $\omega_2$ and $\eta_2$ are algebraic over the field $k(g_2,g_3,\omega_1,\eta_1).$ Hence from Legendre relation \eqref{Equation:Legendre} and the inequality \eqref{GACMM^T} we obtain
\[ \mathrm{t.d.} F = \mathrm{t.d.} F (\omega_{2}, \eta_{2}) = 	\QQ\Big( t_l,\rme^{t_l}, g_2,g_3,\omega_1, q_j, p_i,\]
\[ \wp\Big(\frac{\omega_1}{2}\Big), \wp(q_j),\eta_1, \zeta(q_j), \wp(p_i), \zeta(p_i),  f_{\frac{\omega_1}{2}}(p_{i'}), f_{q_{j'}}(p_{i'}), \rme^{t_{i''j''}} f_{q_{j''}}(p_{i''}),t_{i''j''}, \omega_{2}, \eta_{2}\Big)_{ \substack{j=2, \dots, r \\  i=1, \dots,n \\j'=2, \dots, r' \\  i'=1, \dots,n' \\l=1, \dots,s \\j''= r', \dots, r'' \\  i''= n', \dots,n''}}   \geqslant \]
\[ 2 + 2[n+(r-1)] +n'(r'-1)+ (n''-n')(r''-r')+s = 2(n+r) +n'(r'-1)+ (n''-n')(r''-r')+s .\]

2.2) non--CM case: From Legendre relation \eqref{Equation:Legendre}  and the inequality \eqref{GACMM^T} we get
\[ \mathrm{t.d.} F (\omega_{2}, \eta_{2}) = 	\QQ\Big( t_l,\rme^{t_l}, g_2,g_3,\omega_1, q_j, p_i, \]
\[	 \wp\Big(\frac{\omega_1}{2}\Big), \wp(q_j),\eta_1, \zeta(q_j), \wp(p_i), \zeta(p_i),  f_{\frac{\omega_1}{2}}(p_{i'}), f_{q_{j'}}(p_{i'}),\rme^{t_{i''j''}} f_{q_{j''}}(p_{i''}),t_{i''j''},  \omega_{2}, \eta_{2}\Big)_{ \substack{j=2, \dots, r \\  i=1, \dots,n \\j'=2, \dots, r' \\  i'=1, \dots,n' \\l=1, \dots,s \\j''= r', \dots, r'' \\  i''= n', \dots,n''}}  \geqslant \]
\[ 4 + 2[n+(r-1)] +n'(r'-1)+(n''-n')(r''-r')+ s = 2+ 2(n+r) +n'(r'-1)+ (n''-n')(r''-r')+s .\]
 Removing the two numbers $\omega_{2}, \eta_{2},$ we conclude. 
 
 \vskip 0.3truecm
\par 
(3) ($T^\star \tilde{P} Q^\star$): $ p_1=\frac{\omega_1}{2}$ and $ F (\omega_{1}, \eta_{1},\omega_{2}, \eta_{2}) =  F (\omega_{2}, \eta_{2}).$ 
 It is the dual case of ($T^\star P^\star\tilde{Q}$) that we left to the reader (Hint: use Lemma \ref{Cor:f_q(p)pTorsion} and applied the Grothendieck-Andr\'{e} period Conjecture to the 1-motive $M \oplus M^T \oplus M^{T'},$ where $	M^{T'}=[u^{'T}:\ZZ \to  \GG_m^{r'}], u^{T'}(1)= (\rme^{\frac{1}{2}(\omega_1 \zeta(q_{j'})-\eta_1 q_{j'})})_{j'} \in   \GG_m^{r'} (\CC)$).

\vskip 0.3truecm
\par 
(4) ($T^\star P Q$): (only if $\cE$ is non--CM)  $ p_1=\frac{\omega_1}{2} ,q_1=\frac{\omega_2}{2}$ and $ F (\omega_{1}, \eta_{1},\omega_{2}, \eta_{2}) =  F.$ 
For $j'=2, \dots, r',$ set $\alpha_{j'}:= \frac{1}{2}(\omega_1 \zeta(q_{j'})-\eta_1 q_{j'}).$ Consider the 1-motive
\[
	M^{T'}=[u^{'T}:\ZZ \to  \GG_m^r],\quad  u^{T'}(1)= (\rme^{\alpha_{j'}})_{{j'}=2, \dots, r'} \in   \GG_m^{r'-1} (\CC).
\]
Adding the index $T'$ to the pure motives underlying $\UR(M^{T'})$, we have 
\begin{align}\label{pureM^T'}
	\dim_k B_{T'}&= 0, \\
	\nonumber 	\dim_\QQ Z'_{T'}(1)&= 0,\\
	\nonumber 	\dim_\QQ Z_{T'}(1)&= \dim_\QQ Z_{T'}(1)/Z'_{T'}(1) = \dim_\QQ <\omega_1 \zeta(q_{j'})-\eta_1 q_{j'}  >_{j'=1, \dots, r'},\end{align}
where  $< \omega_1 \zeta(q_{j'})-\eta_1 q_{j'}  >_{j'}$ is the sub $\QQ$--vector space of $\CC / 2 \pi \ii \QQ$ generated by the classes of $\omega_1 \zeta(q_{j'})-\eta_1 q_{j'} $  modulo $2 \pi \ii \QQ.$

 The field of  definition of the direct sum $M \oplus M^T \oplus M^{T'}$ is
\begin{equation}\label{periodsMM^TM^T'}
	K_{ M \oplus M^T \oplus M^{T'}}\big(\mathrm{periods}(M \oplus M^T\oplus M^{T'})\big) = \QQ\Big( g_2,g_3,  \wp(q_j),\wp(p_i),  f_{q_{j'}}(p_{i'}) ,  \rme^{t_{i''j''}} f_{q_{j''}}(p_{i''}),\rme^{t_l} , \rme^{\alpha_{j'}},
\end{equation}
\[ \omega_{1}, \eta_{1},\omega_{2}, \eta_{2},  2 \pi \ii, p_i  ,  \zeta(p_i),  q_j,  \zeta(q_j) , t_{i''j''}, t_l,\alpha_{j'}  \Big)_{ \substack{j=2, \dots, r \\  i=2, \dots,n \\j'=2, \dots, r' \\  i'=2, \dots,n' \\l=1, \dots,s \\j''= r', \dots, r'' \\  i''= n', \dots,n''}} .
\]
We add the index $-TT'$ to the pure motives underlying  $\UR(M \oplus M^T\oplus M^{T'}).$
From equalities \eqref{pureM} and \eqref{pureM^T} and Remark \ref{Contribution} we conclude that 
\begin{align*}
	\dim B_{-TT'}&= \dim B = n-1+r-1, \\
	\dim Z_{-TT'}'(1)&= \dim Z'(1) =(n'-1)(r'-1), \\
	\dim Z_{-TT'}(1)/Z_{-TT'}'(1)&= (n''-n')(r''-r')+ \dim_\QQ < t_l,\frac{1}{2}(\omega_1 \zeta(q_{j'})-\eta_1 q_{j'})  >_{j',l} .
\end{align*}	
Therefore the short exact sequence \eqref{eq:shortexactsequenceUR} yields
\begin{equation} \label{dimGmotMM^TM^T'}
	\dim \Galmot (M \oplus M^T \oplus M^{T'}) =  \frac{4}{\dim_\QQ k} + 2(n-1+r-1) +(n'-1)(r'-1)+
\end{equation}
\[ (n''-n')(r''-r')+ \dim_\QQ < t_l,\alpha_{j'}  >_{j',l}.\]
Applying the Grothendieck-Andr\'{e} period Conjecture to the 1-motive $M \oplus M^T \oplus M^{T'},$ that we can made explicit thanks to \eqref{periodsMM^TM^T'} and to \eqref{dimGmotMM^TM^T'}, we get
\begin{equation}\label{GACMM^TM^T'}
	\mathrm{t.d.} 
	\QQ\Big( g_2,g_3,  \wp(q_j),\wp(p_i),  f_{q_{j'}}(p_{i'}) ,\rme^{t_{i''j''}} f_{q_{j''}}(p_{i''}), \rme^{t_l} , \rme^{\alpha_{j'}}, 
\end{equation}
\[\omega_{1}, \eta_{1},\omega_{2}, \eta_{2},  2 \pi \ii, p_i  ,  \zeta(p_i),  q_j,  \zeta(q_j) , t_{i''j''}, t_l,\alpha_{j'}  \Big)_{ \substack{j=2, \dots, r \\  i=2, \dots,n \\j'=2, \dots, r' \\  i'=2, \dots,n' \\l=1, \dots,s \\j''= r', \dots, r'' \\  i''= n', \dots,n''}} \geqslant\]
\[  \frac{4}{\dim_\QQ k} + 2(n-1+r-1) +(n'-1)(r'-1)+ (n''-n')(r''-r')+  \dim_\QQ < t_l,\alpha_{j'}  >_{l,j'} . \]

	Let $N$ be the dimension of the sub $\QQ$--vector space $< t_l,\alpha_{j'} >_{l,j'}$ of $\CC / 2 \pi \ii \QQ$ generated by the classes of the complex numbers $  t_l$ and $ \alpha_{j'}$  modulo $2 \pi \ii \QQ.$ Since $ \dim_\QQ < t_l>_l=s$, we have that $N \geqslant s$ and without loss of generalities, we may assume that $ t_1, \dots  t_s, \delta_1, \dots, \delta_{N-s}$ is a basis of $< t_l ,\alpha_{j'}  >_{l,j'}$. 

The numbers $\wp(p_1)$ and $\wp(q_1)$ are algebraic over the field $\QQ (g_2,g_3),$ and $\zeta(p_1)= \frac{\eta_1}{2}$ and $\zeta(q_1)= \frac{\eta_2}{2} $ belong to the field $ \overline{\QQ(g_2,g_3)}(\eta_1,\eta_2)$ by \cite[Lemma 3.1 (3)]{BW}. According to Proposition \ref{Cor:f_q(z)qTorsion}, for $i'=2, \dots, n',$ the numbers $f_{q_1}(p_{i'})$ belongs to $\overline{\QQ (g_2,g_3)} ( \wp(p_{i'}),\wp'(p_{i'})).$ Proposition \ref{Cor:f_q(p)pTorsion} implies that for $j'=2, \dots, r',$ the numbers $f_{q_{j'}}(p_1) \rme^{\alpha_{j'}} $ belongs to $\overline{\QQ (g_2,g_3)} ( \wp(q_{j'}),\wp'(q_{j'})).$
 By Corollary \ref{Cor:f_q(p)pqTorsion} or Corollary \ref{Cor:f_q(p)qpTorsion} the number  $f_{q_1}(p_1)$ belongs to $\overline{\QQ (g_2,g_3)}.$	
 From Legendre relation \eqref{Equation:Legendre}, \cite[Lemma4.5]{BW} and the inequality \eqref{GACMM^TM^T'} we get
\[ \mathrm{t.d.} F (\rme^{\delta_l})_{l=1, \dots,N-s} =	\mathrm{t.d.}\, \QQ \Big(t_l,\rme^{t_l},\rme^{\delta_o}, g_2,g_3, \omega_2,q_j, \omega_1,p_i, \wp\Big(\frac{\omega_2}{2}\Big),  \wp(q_j),\eta_2,  \zeta(q_j), \wp\Big(\frac{\omega_1}{2}\Big), \wp(p_i), \]
\[ \eta_1, \zeta(p_i),f_{\frac{\omega_2}{2}}\Big(\frac{\omega_1}{2} \Big),f_{\frac{\omega_2}{2}}(p_{i'}), f_{q_{j'}}\Big(\frac{\omega_1}{2}\Big), f_{q_{j'}}(p_{i'}),\rme^{t_{i''j''}} f_{q_{j''}}(p_{i''}),t_{i''j''} \Big)_{ \substack{j=2, \dots, r \\  i=2, \dots,n \\j'=2, \dots, r' \\  i'=2, \dots,n' \\l=1, \dots,s \\ o=1, \dots, N-s \\j''= r', \dots, r'' \\  i''= n', \dots,n''}}   \geqslant \]
\[ 4 + 2[n-1+r-1] +(n'-1)(r'-1)+ (n''-n')(r''-r')+N =\]
\[  2(n+r) +(n'-1)(r'-1)+ (n''-n')(r''-r')+s +N -s.\]
Removing the $N-s$ numbers $\rme^{\delta_1}, \dots, \rme^{\delta_{N-s}},$ we obtain the expected result.

\vskip 0.3truecm
\par 
(5) ($T^\star P^\star Q$): (only if $\cE$ is non--CM)  $ q_{1}=\frac{\omega_1}{2}, q_2=\frac{\omega_2}{2}$ and $ F (\omega_{1}, \eta_{1},\omega_{2}, \eta_{2}) =  F.$ 
The numbers $\wp(q_{1})$ and $\wp(q_2)$ are algebraic over the field $\QQ (g_2,g_3).$ The numbers $\zeta(q_{1})= \frac{\eta_1}{2}$ and $\zeta(q_2)= \frac{\eta_2}{2} $ belong to the field $ \overline{\QQ(g_2,g_3)}(\eta_1,\eta_2)$ by \cite[Lemma 3.1 (3)]{BW}. According to Proposition \ref{Cor:f_q(z)qTorsion}, for $i'=1, \dots, n',$ the numbers $f_{q_1}(p_{i'})$ and $f_{q_2}(p_{i'})$ belongs to $\overline{\QQ (g_2,g_3)} ( \wp(p_{i'}),\wp'(p_{i'})).$ From Legendre relation \eqref{Equation:Legendre}  and inequality \eqref{GACMM^T} we get
\[ \mathrm{t.d.} F =	\mathrm{t.d.}\, \QQ \Big(t_l,\rme^{t_l}, g_2,g_3,\omega_1, \omega_2, q_j, p_i, \wp(\frac{\omega_1}{2}),\wp\Big(\frac{\omega_2}{2}\Big), \wp(q_j), \eta_1,\eta_2, \zeta(q_j), \wp(p_i), \zeta(p_i), \]
\[ f_{\frac{\omega_1}{2}}(p_{i'}),f_{\frac{\omega_2}{2}}(p_{i'}) , f_{q_{j'}}(p_{i'}) ,\rme^{t_{i''j''}} f_{q_{j''}}(p_{i''}),t_{i''j''}\Big)_{\substack{j=3, \dots, r \\  i=1, \dots,n \\j'=3, \dots, r' \\  i'=1, \dots,n' \\l=1, \dots,s \\j''= r', \dots, r'' \\  i''= n', \dots,n''}}   \geqslant \]
\[ 4 + 2[n+(r-2)] +n'(r'-2)+ (n''-n')(r''-r')+s =  2(n+r) +n'(r'-2)+ (n''-n')(r''-r')+s .\]

\vskip 0.3truecm
\par 
(6) ($T^\star P Q^\star$): (only if $\cE$ is non--CM)  $p_{1}=\frac{\omega_1}{2}, p_2=\frac{\omega_2}{2}$ and $ F (\omega_{1}, \eta_{1},\omega_{2}, \eta_{2}) =  F.$ 
 It is the dual case of ($T^\star P^\star Q$)
 that we left to the reader (Hint: use Lemma \ref{Cor:f_q(p)pTorsion} and applied the Grothendieck-Andr\'{e} period Conjecture to the 1-motive $M \oplus M^T \oplus M^{T'},$ where $	M^{T'}=[u^{'T}:\ZZ \to  \GG_m^{2r'}],  u^{T'}(1)= (\rme^{\frac{1}{2}(\omega_1 \zeta(q_{j'})-\eta_1 q_{j'})},\rme^{\frac{1}{2}(\omega_2 \zeta(q_{j'})-\eta_2 q_{j'})})_{j'} \in   \GG_m^{2r'} (\CC)$).

\vskip 0.3truecm

\par
(7) ($T P^\star Q^\star$):  $  t_s= 2 \pi \ii.$ We distinguished two cases:

7.1) CM case:  Chudnovsky Theorem \eqref{Chudnovsky} and Legendre equation \eqref{Equation:Legendre} imply that 
the five complex numbers $\omega_1, \eta_1,\omega_2, \eta_2, 2 \pi \ii $ generate a field of transcendence degree 2 over $\QQ$, which has $\omega_1, 2 \pi \ii$  as transcendental basis. The inequality \eqref{GACMM^T} furnishes
\[ \mathrm{t.d.} F  (\omega_1)=\mathrm{t.d.} F (\omega_{1}, \eta_{1},\omega_{2}, \eta_2) = \mathrm{t.d.}\, K_{ M \oplus M^T} \big(\mathrm{periods}(M \oplus M^T)\big)=\] 
\[2 + 2(n+r) +n'r'+ (n''-n')(r''-r')+ s-1 = 1 + 2(n+r) +n'r'+ (n''-n')(r''-r')+ s. \]	
Removing the number $\omega_1,$ we get the expected result.

7.2) non--CM case: Grothendieck-André period Conjecture applied to $\cE$ and Legendre equation \eqref{Equation:Legendre} imply that 
the five complex numbers $\omega_1, \eta_1,\omega_2, \eta_2, 2 \pi \ii $ generate a field of transcendence degree 4 over $\QQ$, which has $\omega_1, \eta_1,\omega_2, 2 \pi \ii$  as transcendental basis.
From the inequality \eqref{GACMM^T} we get
\[ \mathrm{t.d.} F (\omega_{1}, \eta_{1},\omega_{2}) = \mathrm{t.d.}\, K_{ M \oplus M^T} \big(\mathrm{periods}(M \oplus M^T)\big)=\] 
\[ 4+ 2(n+r) +n'r'+ (n''-n')(r''-r')+ s -1 = 3 + 2(n+r) +n'r'+ (n''-n')(r''-r')+ s. \]	
Removing the three numbers $\omega_{1}, \eta_{1},\omega_{2}$, we get the expected inequality.

\vskip 0.3truecm
\par 
(8) ($T P^\star\tilde{Q}$):  $  t_s= 2 \pi \ii,  q_1=\frac{\omega_1}{2}$ and $ F (\omega_{1}, \eta_{1},\omega_{2}, \eta_{2}) =  F (\omega_{2}, \eta_{2}).$ 
The number $\wp(q_1)$ is algebraic over the field $\QQ (g_2,g_3)$ and $\zeta(q_1)= \frac{\eta_1}{2} \in \overline{\QQ(g_2,g_3)}(\eta_1,\eta_2)$ by \cite[Lemma 3.1 (3)]{BW}. According to Proposition \ref{Cor:f_q(z)qTorsion}, for $i'=1, \dots, n,$ the number $f_{q_1}(p_{i'})$ belongs to $\overline{\QQ (g_2,g_3)} ( \wp(p_{i'}),\wp'(p_{i'})).$ We distinguished two cases:

8.1) CM case:  Chudnovsky Theorem \eqref{Chudnovsky} and Legendre equation \eqref{Equation:Legendre} imply that 
the five complex numbers $\omega_1, \eta_1,\omega_2, \eta_2, 2 \pi \ii $ generate a field of transcendence degree 2 over $\QQ$, which has $\omega_1, 2 \pi \ii$  as transcendental basis. Hence the inequality \eqref{GACMM^T} implies
\[ \mathrm{t.d.} F =   \mathrm{t.d.} F (\omega_{2}, \eta_{2}) =	\mathrm{t.d.}\, \QQ \Big(t_l,\rme^{t_l}, g_2,g_3,\omega_1, q_j, p_i,\]
\[ \wp\Big(\frac{\omega_1}{2}\Big), \wp(q_j), \eta_1,\zeta(q_j), \wp(p_i), \zeta(p_i), f_{\frac{\omega_1}{2}}(p_{i'}),  f_{q_{j'}}(p_{i'}), \rme^{t_{i''j''}} f_{q_{j''}}(p_{i''}),t_{i''j''},\omega_{2}, \eta_{2}\Big)_{ \substack{j=2, \dots, r \\  i=1, \dots,n \\j'=2, \dots, r' \\  i'=1, \dots,n' \\l=1, \dots,s \\j''= r', \dots, r'' \\  i''= n', \dots,n''}}  \geqslant \]
\[ 2 + 2[n+(r-1)] +n'(r'-1)+ (n''-n')(r''-r')+s-1 = 2(n+r) +n'(r'-1)+ (n''-n')(r''-r')+s-1 .\]
Remark that $ 2 \pi \ii \subset \sum_l \QQ t_l$ and $\Omega \subset \sum_j  k q_j .$ 

8.2) non--CM case:  Grothendieck-André period Conjecture applied to $\cE$ and Legendre equation \eqref{Equation:Legendre} imply that 
the five complex numbers $\omega_1, \eta_1,\omega_2, \eta_2, 2 \pi \ii $ generate a field of transcendence degree 4 over $\QQ$, which has $\omega_1, \eta_1,\omega_2, 2 \pi \ii$  as transcendental basis. From the inequality \eqref{GACMM^T} we obtain
\[ \mathrm{t.d.} F (\omega_{2}) = 	\mathrm{t.d.}\, \QQ \Big(t_l,\rme^{t_l}, g_2,g_3,\omega_1, q_j, p_i,\]
\[ \wp\Big(\frac{\omega_1}{2}\Big), \wp(q_j), \eta_1,\zeta(q_j), \wp(p_i), \zeta(p_i), f_{\frac{\omega_1}{2}}(p_{i'}),  f_{q_{j'}}(p_{i'}),\rme^{t_{i''j''}} f_{q_{j''}}(p_{i''}),t_{i''j''}, \omega_{2}, \eta_{2}\Big)_{ \substack{j=2, \dots, r \\  i=1, \dots,n \\j'=2, \dots, r' \\  i'=1, \dots,n' \\l=1, \dots,s \\j''= r', \dots, r'' \\  i''= n', \dots,n''}}   \geqslant \]
\[ 4 + 2[n+(r-1)] +n'(r'-1)+ (n''-n')(r''-r')+s-1 = 1+ 2(n+r)+n'(r'-1)+ (n''-n')(r''-r')+s .\]
Removing the number $\omega_{2},$ we conclude. 

\vskip 0.3truecm
\par 
(9) ($T \tilde{P} Q^\star$):  $ t_s= 2 \pi \ii,  p_1=\frac{\omega_1}{2}$ and $ F (\omega_{1}, \eta_{1},\omega_{2}, \eta_{2}) =  F (\omega_{2}, \eta_{2}).$ 
It is the dual case of ($T P^\star\tilde{Q}$) that we left to the reader (Hint: use Lemma \ref{Cor:f_q(p)pTorsion} and applied the Grothendieck-Andr\'{e} period Conjecture to the 1-motive $M \oplus M^T \oplus M^{T'},$ where $	M^{T'}=[u^{'T}:\ZZ \to  \GG_m^{r'}],  u^{T'}(1)= (\rme^{\frac{1}{2}(\omega_1 \zeta(q_{j'})-\eta_1 q_{j'})})_{j'} \in   \GG_m^{r'} (\CC)$).

Remark that in the CM case, $ 2 \pi \ii \subset \sum_l \QQ t_l$ and $\Omega \subset \sum_i  k p_i .$

\vskip 0.3truecm
\par 
(10) ($T P Q$): (only if $\cE$ is non--CM)  $ t_s= 2 \pi \ii,  p_1=\frac{\omega_1}{2}, q_1=\frac{\omega_2}{2}$ and 
 $ F (\omega_{1}, \eta_{1},\omega_{2}, \eta_{2}) =  F.$ 
We use the same notation as in ($T^\star P Q$).
	Let $N$ be the dimension of the sub $\QQ$--vector space $< t_l,\alpha_j >_{l,j}$ of $\CC / 2 \pi \ii \QQ$ generated by the classes of the complex numbers $  t_l$ and $ \alpha_j$  modulo $2 \pi \ii \QQ.$ Since $ \dim_\QQ < t_l>_l=s-1$, we have that $N \geqslant s-1$ and without loss of generalities, we may assume that $ t_1, \dots  t_{s-1}, \delta_1, \dots, \delta_{N-s+1}$ is a basis of $< t_l ,\alpha_j  >_{l,j}$. 
The inequality \eqref{GACMM^TM^T'} implies
\[ \mathrm{t.d.} F (\rme^{\delta_l})_{l=1, \dots,N-s+1} =	\mathrm{t.d.}\, \QQ \Big(t_l,\rme^{t_l},\rme^{\delta_o}, g_2,g_3, \omega_2,q_j, \omega_1,p_i, \wp\Big(\frac{\omega_2}{2}\Big),  \wp(q_j),\eta_2,  \zeta(q_j), \wp\Big(\frac{\omega_1}{2}\Big), \wp(p_i),  \]
\[\eta_1, \zeta(p_i), f_{\frac{\omega_2}{2}}\Big(\frac{\omega_1}{2} \Big),f_{\frac{\omega_2}{2}}(p_{i'}), f_{q_{j'}}\Big(\frac{\omega_1}{2}\Big) f_{q_{j'}}(p_{i'}) ,\rme^{t_{i''j''}} f_{q_{j''}}(p_{i''}),t_{i''j''} \Big)_{ \substack{j=2, \dots, r \\  i=2, \dots,n \\j'=2, \dots, r' \\  i'=2, \dots,n' \\l=1, \dots,s \\ o=1, \dots, N-s+1 \\j''= r', \dots, r'' \\  i''= n', \dots,n''}}   \geqslant \]
\[ 4 + 2[n-1+r-1] +(n'-1)(r'-1)+ (n''-n')(r''-r')+N =\]
\[  2(n+r) +(n'-1)(r'-1)+ (n''-n')(r''-r')+(s-1) +N-s+1 .\]
Removing the $N-s+1$ numbers $\rme^{\delta_1}, \dots, \rme^{\delta_{N-s+1}}$ we obtain the expected result.

Remark that $ 2 \pi \ii \subset \sum_l \QQ t_l$ and $\Omega \subset \sum_i  k p_i + \sum_j  k q_j  .$ 

\vskip 0.3truecm
\par 
(11) ($T P^\star Q$): (only if $\cE$ is non--CM)  $   t_s= 2 \pi \ii,  q_{1}=\frac{\omega_1}{2}, q_2=\frac{\omega_2}{2}$ 
 and  $ F (\omega_{1}, \eta_{1},\omega_{2}, \eta_{2}) =  F.$ 
The numbers $\wp(q_{1})$ and $\wp(q_2)$ are algebraic over the field $\QQ (g_2,g_3).$ The numbers $\zeta(q_{1})= \frac{\eta_1}{2}$ and $\zeta(q_2)= \frac{\eta_2}{2} $ belong to the field $ \overline{\QQ(g_2,g_3)}(\eta_1,\eta_2)$ by \cite[Lemma 3.1 (3)]{BW}. According to Proposition \ref{Cor:f_q(z)qTorsion}, for $i'=1, \dots, n',$ the numbers $f_{q_{1}}(p_{i'})$ and $f_{q_2}(p_{i'})$ belongs to $\overline{\QQ (g_2,g_3)} ( \wp(p_{i'}),\wp'(p_{i'})).$

From the inequality \eqref{GACMM^T} we get
\[ \mathrm{t.d.} F =	\mathrm{t.d.}\, \QQ \Big(t_l,\rme^{t_l}, g_2,g_3,\omega_1, \omega_2, q_j,  p_i, \wp\Big(\frac{\omega_1}{2}\Big),\wp\Big(\frac{\omega_2}{2}\Big),\wp(q_j),\eta_1,\eta_2, \zeta(q_j),  \]
\[ \wp(p_i), \zeta(p_i), f_{\frac{\omega_1}{2}}(p_{i'}),f_{\frac{\omega_2}{2}}(p_{i'}) f_{q_{j'}}(p_{i'}), \rme^{t_{i''j''}} f_{q_{j''}}(p_{i''}),t_{i''j''}\Big)_{ \substack{j=3, \dots, r \\  i=1, \dots,n \\j'=3, \dots, r' \\  i'=1, \dots,n' \\l=1, \dots,s \\j''= r', \dots, r'' \\  i''= n', \dots,n''}}   \geqslant \]
\[ 4 + 2[n+(r-2)] +n'(r'-2)+ (n''-n')(r''-r')+s-1 =  2(n+r)   +n'(r'-2)+ (n''-n')(r''-r')+s-1 .\]

Remark that $ 2 \pi \ii \subset \sum_l \QQ t_l$ and $\Omega \subset \sum_j  k q_j  .$

\vskip 0.3truecm
\par 
(12) ($T P Q^\star$): (only if $\cE$ is non--CM)  $   t_s= 2 \pi \ii,  p_{1}=\frac{\omega_1}{2}, p_2=\frac{\omega_2}{2}$ 
and $ F (\omega_{1}, \eta_{1},\omega_{2}, \eta_{2}) =  F.$ 
It is the dual case of ($T P^\star Q$)
that we left to the reader (Hint: use Lemma \ref{Cor:f_q(p)pTorsion} and applied the Grothendieck-Andr\'{e} period Conjecture to the 1-motive $M \oplus M^T \oplus M^{T'},$ where $	M^{T'}=[u^{'T}:\ZZ \to  \GG_m^{2r'}],  u^{T'}(1)= (\rme^{\frac{1}{2}(\omega_1 \zeta(q_{j}')-\eta_1 q_{j'})},\rme^{\frac{1}{2}(\omega_2 \zeta(q_{j'})-\eta_2 q_{j'})})_{j'} \in   \GG_m^{2r'} (\CC)$).

Remark that $ 2 \pi \ii \subset \sum_l \QQ t_l$ and $\Omega \subset \sum_i  k p_i .$ 
\end{proof}

\section{$\sigma$-Conjecture}\label{sigma}

Applying the multiplication formula of the $\sigma$ function \eqref{sigma(mz)} with $m=2$, for any complex number $p$ which does not belong to $\Omega$ we have the equalities
\begin{align}\label{eq:sigma1}
	f_{p}(p) \e^{\zeta(p) p}&= \frac{\sigma (2p)}{\sigma(p)^2}  \\
	\nonumber	& = -  \frac{ \sigma(p)^4 \psi_2(\wp(p),\wp'(p))  }{\sigma(p)^2 } \\
	\nonumber	& = -   \sigma(p)^2 \psi_2(\wp(p),\wp'(p)). 
\end{align}

\soluzioni{
Let $p$ be complex numbers in $\CC \smallsetminus( k \, \omega_1 + k \; \omega_2 )$ and let $G$ be the extension of the elliptic curve $\cE$ by $\GG_m$ parametrized by the point $P \in \cE (\CC)$.
Consider the 1-motive
\begin{equation} \label{GGPC-1-motive-P=Q}
	M=[u:\ZZ \to  G],\quad  u(1)= R \in   G (\CC)
\end{equation}
where
\begin{align*}
	\nonumber	R&=\exp_{G}(p,\zeta(p) p)\\
	\nonumber	&=  \sigma(p)^3\Big[ \wp(p):\wp'(p):1 : \rme^{\zeta(p) p} f_{p}(p):  \rme^{\zeta(p) p} f_{p}(p)  \wp(p) \Big]
\end{align*}
The field of definition of the 1-motive $M=[u:\ZZ \rightarrow G] $ defined in \eqref{GGPC-1-motive-P=Q} is 
$$K:=\QQ(g_2,g_3,P,R)=\QQ(g_2,g_3, \wp(p), \rme^{\zeta(p) p} f_{p}(p) ).$$
Because of \eqref{eq:sigma1} observe that 
$$\overline{K} =\overline{\QQ (g_2,g_3, \wp(p),\sigma(p))}.$$
 The field \eqref{field} generated by the periods of $M$ over the field of definition of $M$ is
\[	K \big(\mathrm{periods}(M)\big) =
\QQ\big( g_2,g_3, \wp(p),\sigma(p),\omega_{1}, \eta_{1},\omega_{2}, \eta_{2},  2 \ii \pi, p  ,  \zeta(p)  \big).
\]
Since the identity is not an antisymmetric endomorphism, by \cite[Corollaty 4.5 (3)(b)]{BP} the equality \eqref{DimGalmot} becomes 
\[\dim \Galmot (M) = \frac{4}{\dim_\QQ k} + 2 \dim B +\dim Z'(1) = \frac{4}{\dim_\QQ k} + 3 \]

The GGP Conjecture applied to the 1-motive \eqref{GGPC-1-motive-P=Q} furnishes the inequality
\begin{equation*}
	\mathrm{t.d.}\, \QQ \big( g_2,g_3,\wp(p),\sigma(p),\omega_{1}, \eta_{1},\omega_{2}, \eta_{2}, p  ,  \zeta(p) \big) 	\geq  \frac{4}{\dim_\QQ k} + 3
\end{equation*}
Removing the two numbers $\omega_1, \eta_1$ in the case of complex multiplication (recall that by Lemma \ref{OmegaEta} the number $\omega_2$ and $\eta_2$ are algebraic over the field $k(g_2,g_3,\omega_1,\eta_1)$), and removing the four numbers $\omega_{1}, \eta_{1},\omega_{2}, \eta_{2}$ in the case of not complex multiplication we get 
\begin{equation*}
	\mathrm{t.d.}\, \QQ \big( g_2,g_3,p, \wp(p),  \zeta(p), \sigma(p) \big) 	\geq   3
\end{equation*}
}

Let	$ p_1,\dots,p_n$ be complex numbers in $\CC \smallsetminus \Omega.$
Denote by $G_{i}$ the extension of $\cE$ by $\GG_m$ parame\-tri\-zed by the point $P_i $. Let $G$ be the extension of $\cE$ by $\GG_m^n$ parametrized by $n$ points $P_1, \ldots, P_n \in  \cE (\CC): G$ is isomorphic to the product of extensions $G_{1} \times \dots \times G_{n}$.
 Consider the auto-dual 1-motive 
$	M=[u:\ZZ \to  G],  u(1) =\big(R_{1}, \dots , R_{n} \big) \in G_{1} (\CC) \times \dots \times G_{n} (\CC),$ with
\begin{equation}\label{eq:sigma2}
		R_i=\exp_{G_{i}}(p_i,\zeta(p_i) p_i) 
\end{equation}
for $i=1, \dots,n$ (auto-dual means $n=r$ and $p_i=q_i$). 

In this section we assume the complex numbers $ p_1,\dots,p_n$ to be $k$--linearly independent and in order to simplify notations we set 
\[n_1 := \dim_k < p_i>_{i} \qquad \mathrm{and}  \qquad  n_2:= \tor(p_i)= n-n_1.\]

\begin{proposition}\label{dimMpp} 
	Let $	M=[u:\ZZ \to  G], u(1) =\big(R_{1}, \dots , R_{n} \big) \in G_{1} (\CC) \times \dots \times G_{n} (\CC)$ be the auto-dual 1-motive defined by the points $p_i=q_i$ (clearly $n=r$) and $t_i=\zeta(p_i) p_i.$  Assume the complex numbers $ p_1,\dots,p_n$ to be $k$--linearly independent. Then
	\[ \dim \Galmot(M)= \frac{4}{\dim_\QQ k} +
	\begin{cases*}
		 3 \dim_k < p_i>_{i} + \dim_\QQ < \zeta(p_i) p_i >_{i=1,n_2} & if $n_2 \not= 0,$ \\
		  3 n & if $n_2=0$
	\end{cases*} \]
	where 
	\begin{itemize}
		\item $< p_i>_{i}$ is the sub $k$--vector space of $\CC / (\Omega \otimes_\ZZ \QQ)$ generated by the classes of  $ p_1, \dots, p_n $  modulo $\Omega \otimes_\ZZ \QQ,$
		\item $< \zeta(p_i) p_i>_{i=1,n_2}$ is the sub $\QQ$--vector space of $\CC / 2 \pi \ii \QQ$ generated by the classes of the complex numbers $ \zeta(p_i) p_i$  modulo $2 \pi \ii \QQ.$
	\end{itemize} 
\end{proposition}

\begin{proof} If $n_2 =n-n_1 \not=0$, without loss of generality we assume  $p_{n_2+1}, \dots, p_n \notin  (\Omega \otimes_\ZZ \QQ). $  For $i=1, \dots n$ consider the 1-motive introduced in \eqref{Mij}
\[
M_{ii}=[u_i:\ZZ \to  G_{i}],\qquad  u_i(1)=R_{i} =\exp_{G_{i}}(p_i,\zeta(p_i) p_i) \in G_{i}(\CC)
\]
Let $M_1 =\oplus_{i=n_2+1}^nM_{ii}$ and $M_2= \oplus_{i=1}^{n_2}M_{ii}.$ If $n_2=0,$ we don't use the 1-motive $M_2$ in this proof. We add the index 1 (\textit{resp}. 2, \textit{resp}. $ii$) to the pure motives underlying the unipotent radical of the 1-motive $M_1$  (\textit{resp}. $M_2$, \textit{resp}. $M_{ii}$ respectively). 

We start studying the unipotent radical of the motivic Galois group of the 1-motive $M_2.$
Since $p_i \in \Omega \otimes_\ZZ \QQ$ for $i=1, n_2,$ according to Theorem \ref{Teo:dimGal(M)} 
\begin{equation}\label{eq:UR(M_2)}
	\dim B_2 = \dim Z_2'(1) =0 \quad  \mathrm{and}  \quad \dim Z_2(1) / Z_2'(1) = \dim_\QQ <  \zeta(p_i) p_i>_{ i=1, n_2}.
\end{equation}

We now analyze the dimension of $\UR( M_1).$
The complex numbers $ p_{n_2+1},\dots,p_n,$ are $k$-linearly independent and do not belong to $\Omega \otimes_\ZZ \QQ.$ Hence
\begin{equation}\label{eq:UR(M_1)-1}
	\dim B_1 =n - n_2 =n_1=\dim_k < p_i>_{i} .
\end{equation}
Since the identity is not an antisymmetric endomorphism, by Theorem \ref{Teo:dimGal(M)} $\dim Z_{ii}'(1) =1$ and 
$\dim Z_{ii}(1) / Z_{ii}'(1) = 0$  for $i= n_2+1, \dots , n.$ The inequality
\[
\dim Z_1(1) / Z_1'(1) \leqslant  \oplus_{i= n_2+1}^n \dim Z_{ii}(1) / Z_{ii}'(1)=0,
\]
implies then the equality
\begin{equation}\label{eq:UR(M_1)-2}
	\dim Z_1(1) / Z_1'(1) = 0.
\end{equation}
	We now prove that 
	\begin{equation}\label{eq:UR(M_1)-3}
		\dim Z_1(1)  = n_1.
	\end{equation}
 Since the complex numbers $p_{n_2+1}, \dots,p_n$ are $k$-linearly independent, 
$B_1=\cE^{n_1}$ and the inclusion $I : B_1 \to  \cE^{n_1} \times \cE^{*{n_1}}$ is just the diagonal morphism $d: B_1 \to B_1 \times B_1, (P_{n_2+1}, \dots, P_n) \mapsto (P_{n_2+1}, \dots, P_n, P_{n_2+1}, \dots, P_n)$ (recall that we identity $\cE^*$ with $\cE$). Because $d(B_1)= \cE^{n_1} \times \cE^{*{n_1}}$
	we may take $\gamma_{i}\in\mathrm{Hom}_{\mathbb Q}(d(B_1),\cE)$ as the projection of $d(B_1)$ onto the $i$-th factor $\cE$ of $\cE^{n_1}\times \cE^{*{n_1}}$ for $i=1, \dots, n_1$, and $\gamma_k^*\in\mathrm{Hom}_{\mathbb Q}(d(B_1),\cE^*)$ as the projection of $d(B_1)$ onto the $k$-th factor $\cE^*$ of $\cE^{n_1}\times \cE^{*{n_1}}$ for $k=1, \dots, {n_1}$. 
	The transpose $\gamma_{i}^t: \cE^* \to \cE^{*{n_1}} \times \cE^{{n_1}}  $ is the inclusion of $\cE^*$ into the $i$-th factor $\cE^*$ of $\cE^{*{n_1}}\times \cE^{{n_1}} ,$  and 
	$\gamma_{k}^{*t}: \cE \to  \cE^{*{n_1}} \times \cE^{{n_1}}  $ is  the inclusion of $\cE$ into the $k$-th factor $\cE$ of $\cE^{*{n_1}}\times \cE^{{n_1}}.$
	Therefore $\beta_{i,i}=\gamma_i^t\circ\gamma_i^*$ is the identity between the $i$-th factor $\cE^*$ of $d(B_1)$ and the $i$-th factor $\cE^*$ of $\cE^{*{n_1}} \times \cE^{{n_1}} $ and zero on the other factors. Similarly $\beta_{i,i}^t = \gamma_i^{*t} \circ \gamma_i$ is the identity between the $i$-th factor $\cE$ of $d(B_1)$ and the $i$-th factor $\cE$ of $\cE^{*{n_1}} \times \cE^{{n_1}} $ and zero on the other factors.
	For any point $P$ in the $i$-th factor $\cE$ of $B_1$, we have that $(\beta_{i,i}+\beta_{i,i}^t) (d(P))=d(P).$ Hence $ d^{-1} \circ (\beta_{i,i}+\beta_{i,i}^t) \circ d$ is a group homomorphism from $B_1$ to $B^*_1$ which is the identity between the $i$-th factor $\cE$ of $B_1$ and the $i$-th factor $\cE^*$ of $B^*_1$ and zero on the other factors.
	Assume there is a relation with coefficients in $\QQ$
	\[\sum_{i=1}^{n_1}\lambda_{i,i} \big(d^{-1} \circ(\beta_{i,i}+\beta_{i,i}^t)  \circ d \big) = 0.\]
	Evaluating the left hand side at a point $P$ which lies in the $i$-th factor $\cE$ of $B_1$, we get  $\lambda_{i,i} \, P =0$ which  implies $\lambda_{i,i}=0$. Hence the $(d^{-1} \circ(\beta_{i,i}+\beta_{i,i}^t)  \circ d) $ are linearly independent over $\mathbb Q$ and by Theorem \ref{Teo:dimGal(M)} the dimension of $Z'(1)$ is $n_1.$

According to \cite[Lemma 2.2]{B19}, the 1-motive $M$ and the 1-motive $M_1 \oplus M_2$ generate the same tannakian category and so $\Galmot(M) = \Galmot(M_1 \oplus M_2).$
 Since the contributions of the pure motives underlying the unipotent radicals $\UR(M_1)$ and $\UR(M_2)$ are complementary, from equalities \eqref{eq:UR(M_2)}, \eqref{eq:UR(M_1)-1}, \eqref{eq:UR(M_1)-2} and \eqref{eq:UR(M_1)-3} we conclude that 
\begin{align*}
	\dim \UR (M)= \dim \UR (M_1 \oplus M_2) &= \dim \UR (M_1) +  \dim \UR ( M_2)\\
	&=   2 \dim B_1 + \dim Z_1'(1) + \dim  Z_2(1) / Z_2'(1) \\
	&=  3 \dim_k < p_i>_{i} + \dim_\QQ < \zeta(p_i) p_i>_{i=1,n_2}.
\end{align*}
\end{proof}

Adding hypotheses, we can estimate the dimension of the $\QQ$--vector space $< \zeta(p_i) p_i>_{i=1,n_2}.$ In fact, we have the following corollary due to M. Waldschmidt:

\begin{corollary}\label{ImplicationsGA} Let $n_2 =n-n_1 \not=0.$ Assume $g_2,g_3$ to be algebraic.
	\begin{enumerate}
		\item If the elliptic curve has complex multiplication ($n_2=1$), $\dim_\QQ < \zeta(p_1) p_1>=n_2,$
		and so
		\[	 \dim \Galmot(M)=  2+  2 \dim_k < p_i>_{i} + n.\]
		\item If the elliptic curve has not complex multiplication ($ 1\leqslant n_2 \leqslant 2$), assuming the Gro\-then\-dieck-André period Conjecture we have that $\dim_\QQ < \zeta(p_i) p_i>_{i=1,n_2} =n_2,$ and in particular
		\[	 \dim \Galmot(M)=  4+  2 \dim_k < p_i>_{i} + n.\]	
	\end{enumerate}
\end{corollary}

\begin{proof}
	 In the CM case the $k$--vector space of torsion points is generated by $\frac{\omega_1}{2}$ because of \eqref{Equation:tau}, while in the non CM case it is generated by $\frac{\omega_1}{2}$ and $\frac{\omega_2}{2}.$ Hence without loss of generality we may assume $p_1=\frac{\omega_1}{2}$ in the CM case, and $p_i=\frac{\omega_i}{2}$ for $i=1, n_2$ in the non CM case. Therefore 
	\[\dim_\QQ < \zeta(p_i) p_i>_{i=1,n_2} =\dim_\QQ < \eta_i \omega_i>_{i=1,n_2}.\]
	 We distinguish two cases:
	\begin{enumerate}
		\item CM case: Chudnovsky Theorem \eqref{Chudnovsky} implies that 
		 the six complex numbers $\omega_1, \eta_1,\omega_2, \\ \eta_2, 2 \pi \ii , \tau$ generate a field of transcendence degree 2 over $\QQ$, which has $\omega_1, \pi$ or $\omega_2, \pi$ as transcendental bases. Moreover the only algebraic relations between these six numbers are 
\[\begin{cases}
	\omega_2=\tau\omega_1  & \eqref{Equation:tau}	\\
 \omega_2 \eta_1- \eta_2 \omega_1 = 2 \pi \ii &   \eqref{Equation:Legendre}\\
	 A\eta_1- C\tau\eta_2-\kappa\omega_2=0 & \eqref{Equation:kappa}\\
	A+B \tau+C \tau^2=0 & \eqref{Polynomial:tau} \\
\end{cases}		
\]		
with $\kappa \in \oQQ.$ Using the first three polynomial relations we obtain
		\[ \Big(  \frac{A}{C \tau } - \tau  \Big) \omega_1 \eta_1 - \frac{\kappa}{C} \omega_1^2 + 2 \pi \ii =0 ,\]
	\[ \Big( \frac{A}{C \tau } - \tau \Big) \omega_2 \eta_2 - \frac{\kappa}{C} \omega_2^2 + 2 \pi \ii \frac{A}{C } =0.\]
	
	Assume first $\kappa \not=0.$ Since  $ \mathrm{t.d.}\, \QQ (\omega_1, \pi )=2$ and $ \mathrm{t.d.}\, \QQ (\omega_2, \pi)=2,$
		the numbers 
		\[ \frac{\kappa}{C 2 \pi \ii} \omega_1^2 -1 \qquad \mathrm{and} \qquad  \frac{\kappa}{C2 \pi \ii} \omega_2^2 - \frac{A}{C }\]
		are transcendental, which implies that $\frac{\omega_1 \eta_1}{ 2 \pi \ii} $ and $\frac{\omega_2 \eta_2}{ 2 \pi \ii}$ are also  transcendental. Hence neither $\eta_1 \omega_1$ nor $\eta_2 \omega_2$ belong to $2 \pi \ii \QQ,$ that is 
		\[ \dim_\QQ < \zeta(p_1) p_1>  =   \dim_\QQ < \eta_1 \omega_1> =1=n_2.\]
		
		Suppose now that $\kappa =0.$ The two numbers 
		\[ r_1:= \frac{2 \pi \ii}{\omega_1 \eta_1} = \tau - \frac{A}{C \tau}  \qquad \mathrm{and} \qquad  r_2:= \frac{2 \pi \ii}{\omega_2 \eta_2} = \frac{C}{A} \tau - \frac{1}{ \tau} \]
		are algebraic and satisfy the polynomial relations 
		\[ C \tau^2- r_1 C \tau-A=0 \qquad \mathrm{and} \qquad C \tau^2 - r_2 A \tau -A=0. \]
		Using \eqref{Polynomial:tau} we obtain
		\[ r_1= -\frac{B}{C}  - \frac{2A}{C \tau}  \qquad \mathrm{and} \qquad  r_2=  -\frac{B}{A}  - \frac{2}{ \tau} .\]
		 Since $\tau$ is irrational, $r_1$ and $r_2$ are also irrational.
		Hence neither $\eta_1 \omega_1$ nor $\eta_2 \omega_2$ belong to $2 \pi \ii \QQ,$ that is 
		\[ \dim_\QQ < \zeta(p_1) p_1>  =   \dim_\QQ < \eta_1 \omega_1> =1=n_2.\]
		
		\item In the non CM case, the Grothendieck-André period Conjecture implies that 
		\par\noindent	$\mathrm{t.d.}_{\QQ}\, \QQ (\omega_1, \omega_2,\eta_1,\eta_2) = 4.$ Using Legendre equation \eqref{Equation:Legendre}, we get that the complex numbers $\omega_1, \eta_1, 2 \ii \pi$ are algebraically independent over $\oQQ$, and idem for the complex numbers $\omega_2, \eta_2, 2 \ii \pi.$ Hence neither $\eta_1 \omega_1$ nor $\eta_2 \omega_2$ belong to $2 \pi \ii \QQ,$ that is 
		\[ \dim_\QQ < \zeta(p_i) p_i>_{i=1,n_2}  =   \dim_\QQ < \eta_i \omega_i>_{ i=1, n_2} =n_2.\]
	\end{enumerate}
\end{proof}

\begin{example}[by M. Waldschmidt]
 In some special cases it is possible to compute explicitly the product $\omega_i \eta_i$ for $i=1,2.$	Assume $g_2$ and $g_3$ to be algebraic.
 The condition  $\kappa =0$ is equivalent to the condition $g_2g_3=0.$
 
 If $g_3=0,$ by \cite[(5),(9)]{Wald08} we have that $A=C=1, B=0, \tau =\ii, \omega_2=\ii \omega_1$ and 
 $\eta_2= -\ii \eta_1.$ Then we obtain $r_1=r_2= 2 \ii$ and both the products $\omega_1 \eta_1$ and $\omega_2 \eta_2$ are equal to the real number $ \pi.$ More in general, for an arbitrary period $\omega =a \omega_1 +b \omega_2 \in \Omega \setminus\{0\}$ we get
 \[ \omega \eta = (a^2+b^2) \pi.\]
 Since the product $\omega \eta$ is a real number different from 0, it does not belong to $2 \pi \ii \QQ.$
 
 If now $g_2=0,$ by \cite[(6),(10)]{Wald08} we have $A=B=C=1,$ $\tau $ is the third root of unity $\rho=\frac{(-1+\ii \sqrt{3})}{2},$ $\omega_2=\rho \omega_1$ and $\eta_2=\rho^2 \eta_1.$ Then we get $r_1=r_2= \rho- \rho^2= \ii \sqrt{3,}$ and both the products $\omega_1 \eta_1$ and $\omega_2 \eta_2$ are equal to the real number $ \frac{2 \pi}{\sqrt{3}}$. For an arbitrary period $\omega =a \omega_1 +b \omega_2 \in \Omega \setminus\{0\}$ we obtain
 \[ \omega \eta = (a^2-ab+ b^2)  \frac{2 \pi}{\sqrt{3}}.\]
 Since the product $\omega \eta$ is a non-zero real number, it does not belong to $2 \pi \ii \QQ.$
\end{example}

\begin{theorem}\label{GA=>Sigma} 
	The Grothendieck-André period Conjecture applied to the auto-dual 1-motive $M=[u:\ZZ \to  G], u(1) =\big(R_{1}, \dots , R_{n} \big) \in G_{1} (\CC) \times \dots \times G_{n} (\CC)$
	 defined by the points $p_i=q_i$ (clearly $n=r$) and $t_i=\zeta(p_i) p_i,$ such that  $ p_1,\dots,p_n$ are $k$--linearly indipendent, implies the $\sigma$-Conjecture \ref{SigmaC}.
\end{theorem}

\begin{proof}
Because of \eqref{eq:sigma1} the field \eqref{field} generated by the periods of $M$ over the field of definition of $M$ is
\begin{equation}
	\overline{K} \big(\mathrm{periods}(M)\big) =
\overline{\QQ (g_2,g_3, \wp(p_i),\sigma(p_i))}_i \big(\omega_1, \omega_2 , \eta_1, \eta_2 , 2 \ii  \pi, p_i , \zeta(p_i) \big)_{ i=1, \dots, n} 
\end{equation}

We distinguish several cases:

(1) Suppose $n_2 =0,$ that is the complex numbers $p_1, \dots, p_n$ do not belong to $\Omega \otimes_\ZZ \QQ.$
 By Proposition \ref{dimMpp}, 
the Grothendieck-André period Conjecture applied to $M$ furnishes the inequality
\begin{equation*}
	\mathrm{t.d.} \, \QQ \big( g_2,g_3, \wp(p_i),\sigma(p_i), \omega_1, \omega_2 , \eta_1, \eta_2 , p_i , \zeta(p_i) \big)_{  i=1, \dots, n} 	\geqslant  \frac{4}{\dim_\QQ k} + 3n.
\end{equation*}
Removing the two numbers $\omega_1, \eta_1$ in the CM case (recall that by Lemma \ref{OmegaEta} the numbers $\omega_2$ and $\eta_2$ are algebraic over the field $k(g_2,g_3,\omega_1,\eta_1)$), and removing the four numbers $\omega_{1}, \eta_{1},\omega_{2}, \eta_{2}$ in the non CM case, we get the expected inequality
\begin{equation*}
\mathrm{t.d.}\, \QQ \big( g_2,g_3, \wp(p_i),\sigma(p_i) , p_i , \zeta(p_i) \big)_{  i=1, \dots, n} \geqslant  3n.
\end{equation*}

(2) Assume $n_2 \not=0.$ 

(2.1) CM case: Let $p_1=\frac{\omega}{m},$ with $\omega$ a period and $m$ an integer, $m\geqslant 2,$ such that $\frac{\omega}{m} \notin \Omega$, and  $p_{2}, \dots, p_n$ do not belong to $ \Omega \otimes_\ZZ \QQ. $
According to \cite[Lemma 3.1 (3)]{BW} the number $\zeta(\frac{\omega}{m}) -\frac{\eta(\omega)}{m}$ is algebraic over $\QQ(g_2,g_3)$, and by Lemma \ref{OmegaEta} the numbers $\omega_2$ and $\eta_2$ are algebraic over the field $k(g_2,g_3,\omega_1,\eta_1).$ Hence the two fields
$\overline{\QQ} \big(g_2,g_3, \wp(p_i),\sigma(p_i),\omega_1, \omega_2 , \eta_1, \eta_2 , p_i , \zeta(p_i) \big)_{  i=1, \dots, n}$ and
$\overline{\QQ} \big(g_2,g_3, \wp(p_i),\sigma(p_i), p_i , \zeta(p_i) \big)_{  i=1, \dots, n}$
have the same algebraic closure and the same transcendence degree.
Proposition \ref{dimMpp} implies that 
\begin{align*}
	\Galmot(M)=& 2+  3 (n-1) + \dim_\QQ < \zeta(p_1) p_1>\\
	=& 2n +\dim_k <p_i>_i +\dim_\QQ < \zeta(p_1) p_1>.
\end{align*}
The Grothendieck-André period Conjecture applied to the 1-motive $M$ furnishes then the expected inequality
\begin{equation*}
	\mathrm{t.d.}\, \QQ \big( \wp(p_i),\sigma(p_i),p_i, \zeta(p_i) \big)_{ i=1, \dots, n} \geqslant 2n +\dim_k <p_i>_i +\dim_\QQ < \zeta(p_1) p_1>.
\end{equation*}

(2.2) non CM case:

(2.2.1) $n_2=1$: Let $p_1=\frac{\omega}{m},$ with $\omega$ a period and $m$ an integer, $m\geqslant 2,$ such that $\frac{\omega}{m} \notin \Omega$, and  $p_{2}, \dots, p_n $ do not belong to $  \Omega \otimes_\ZZ \QQ. $
According to \cite[Lemma 3.1 (3)]{BW} the number $\zeta(\frac{\omega}{m}) -\frac{\eta(\omega)}{m}$ is algebraic over $\QQ$.
If $p_1 \notin \QQ \omega_1$,  the two fields
$\overline{\QQ} \big(g_2,g_3, \wp(p_i),\sigma(p_i),\omega_1, \omega_2 , \eta_1, \eta_2 , p_i , \zeta(p_i) \big)_{  i=1, \dots, n} $ and 
$\overline{\QQ} \big(g_2,g_3, \wp(p_i),\sigma(p_i),\omega_1,\eta_1, p_i , \zeta(p_i) \big)_{  i=1, \dots, n}$
have the same algebraic closure and the same transcendence degree.
Proposition \ref{dimMpp} furnishes that 
\begin{align*}
	\Galmot(M)=& 4+  3 (n-1) + \dim_\QQ < \zeta(p_1) p_1>\\
	=& 2+ 2n +\dim_k <p_i>_i +\dim_\QQ < \zeta(p_1) p_1>.
\end{align*}
The Grothendieck-André period Conjecture applied to $M$ furnishes then the inequality
\begin{equation*}
	\mathrm{t.d.}\, \QQ \big(g_2,g_3, \wp(p_i),\sigma(p_i),\omega_1, \eta_1,p_i, \zeta(p_i) \big)_{ i=1, \dots, n} \geqslant 2+ 2n +\dim_k <p_i>_i +\dim_\QQ < \zeta(p_1) p_1>.
\end{equation*}
Removing the two numbers $\omega_1, \eta_1,$ we get the $\sigma$-Conjecture \ref{SigmaC}
\begin{equation*}
	\mathrm{t.d.}\, \QQ \big(g_2,g_3, \wp(p_i),\sigma(p_i),p_i, \zeta(p_i) \big)_{ i=1, \dots, n} \geqslant 2n +\dim_k <p_i>_i +\dim_\QQ < \zeta(p_1) p_1>.
\end{equation*}
If $p_1=\frac{\omega_1}{m}$, the two fields
$\overline{\QQ} \big(g_2,g_3, \wp(p_i),\sigma(p_i),\omega_1, \omega_2 , \eta_1, \eta_2 , p_i , \zeta(p_i) \big)_{  i=1, \dots, n}$ and\\
$\overline{\QQ} \big(g_2,g_3, \wp(p_i),\sigma(p_i),\omega_2,\eta_2, p_i , \zeta(p_i) \big)_{  i=1, \dots, n}$
have the same algebraic closure and the same transcendence degree.
Removing this time the two numbers $\omega_2, \eta_2,$  we get the $\sigma$-Conjecture \ref{SigmaC}.

(2.2.2) $n_2=2$: Let $p_1=\frac{\omega}{m}$ and $p_2 = \frac{\omega'}{m'},$ with $\omega, \omega'$ two periods and $m, m'$ two integers, $m,m'\geqslant 2,$ such that $\frac{\omega}{m}$ and $\frac{\omega'}{m'} \notin \Omega$, and  $p_{3}, \dots, p_n $ do not belong to $ \Omega \otimes_\ZZ \QQ. $
According to \cite[Lemma 3.1 (3)]{BW} the numbers $\zeta(\frac{\omega}{m}) -\frac{\eta(\omega)}{m}$  and $\zeta(\frac{\omega'}{m'}) -\frac{\eta(\omega')}{m'}$ are algebraic over $\QQ$. Hence the two fields
$\overline{\QQ} \big(g_2,g_3, \wp(p_i),\sigma(p_i),\omega_1, \omega_2 , \eta_1, \eta_2 ,  p_i , \zeta(p_i) \big)_{  i=1, \dots, n}$ and 
$\overline{\QQ} \big(g_2,g_3, \wp(p_i),\sigma(p_i), p_i , \zeta(p_i) \big)_{  i=1, \dots, n}$ 
have the same algebraic closure and the same transcendence degree (recall that $p_1$ and $p_2$ are $\QQ$--linearly independent).
Proposition \ref{dimMpp} implies that 
\begin{align*}
	\Galmot(M)=& 4+  3 (n-2) + \dim_\QQ < \zeta(p_i) p_i>_{i=1,2}\\
	=& 2n +\dim_k <p_i>_i +\dim_\QQ < \zeta(p_i) p_i>_{i=1,2}.
\end{align*}
The Grothendieck-André period Conjecture applied to 
$M$ furnishes then the expected inequality
\begin{equation*}
	\mathrm{t.d.}\, \QQ \big( g_2,g_3,\wp(p_i),\sigma(p_i),p_i, \zeta(p_i) \big)_{ i=1, \dots, n} \geqslant 2n +\dim_k <p_i>_i +\dim_\QQ < \zeta(p_i) p_i>_{i=1,2}.
\end{equation*}

\end{proof}

\bibliographystyle{plain}

\end{document}